\documentclass[preprint,numbers,addressatend]{imsart}

\usepackage{amsthm,amsmath,natbib}
\usepackage{amsfonts,amssymb,graphicx}
\usepackage{color}
\RequirePackage[dvips]{hyperref}

\RequirePackage{hypernat}

\arxiv{1009.5327}

\startlocaldefs
\numberwithin{equation}{section}
\newtheorem{thm}{Theorem}[section]
\newtheorem{prop}[thm]{Proposition}
\newtheorem{lem}[thm]{Lemma}
\def\Indicator{\mathop{\hskip0pt{1}}\nolimits}

\theoremstyle{remark}
\newtheorem{assumption}{Assumption}
\newtheorem{remark}{Remark}

\newcommand{\var}{{\rm Var} \mspace{1mu}}

\usepackage{graphicx}
\usepackage{verbatim}

\endlocaldefs

\begin{document}

\begin{frontmatter}

\title{Uniform Approximations for the M/G/1 Queue with Subexponential Processing Times}
\runtitle{Uniform Approximations for the M/G/1 Queue}


\author{\fnms{Mariana} \snm{Olvera-Cravioto}\corref{}\ead[label=e1]{molvera@ieor.columbia.edu}}
\address{Department of Industrial Engineering  \\ and Operations Research \\ Columbia University \\ New York, NY 10027 \\ \printead{e1}}
\affiliation{Columbia University}
\and \hspace{4pt}
\author{\fnms{Peter W.} \snm{Glynn}\ead[label=e2]{glynn@stanford.edu}}
\address{Department of Management Science \\ and Engineering \\  Stanford University \\ Stanford, CA 94305 \\ \printead{e2}}
\affiliation{Stanford University}

\runauthor{M. Olvera-Cravioto and P. Glynn}

\begin{abstract}
This paper studies the asymptotic behavior of the steady-state waiting time, $W_\infty$, of the M/G/1 queue with subexponenential processing times for different combinations of traffic intensities and overflow levels. In particular, we provide insights into the regions of large deviations where the so-called heavy traffic approximation and heavy tail asymptotic hold. For queues whose service time distribution decays slower than $e^{-\sqrt{t}}$ we identify a third region of asymptotics where neither the heavy traffic nor the heavy tailed approximations are valid. These results are obtained by deriving approximations for $P(W_\infty > x)$ that are either uniform in the traffic intensity as the tail value goes to infinity or uniform on the positive axis as the traffic intensity converges to one.  Our approach makes clear the connection between the asymptotic behavior of the steady-state waiting time distribution and that of an associated random walk. 
\end{abstract}




\end{frontmatter}

\section{Introduction} 

We study in this paper the asymptotic behavior of the steady-state waiting time distribution of an M/G/1 queue with subexponential service time distribution and first-in-first-out (FIFO) discipline. The goal is to provide expressions that will allow us to identify the different types of asymptotic behavior that the queue experiences depending on different combinations of traffic intensity and overflow levels. We give our results for the special case of an M/G/1 queue with the idea that the insights that we obtain are applicable to more general queues and even to networks of queues. 

The special case of an M/G/1 queue with regularly varying processing times was previously analyzed in \cite{OlBlGl_10}, where it was shown that the behavior of $P(W_\infty > x)$, the steady-state waiting time distribution, can be fully described by the so-called heavy traffic approximation and heavy tail asymptotic  (see Theorems 2.1 and 2.2 in \cite{OlBlGl_10}). As pointed out in that work, the same type of results can be derived for a larger subclass of the subexponential family, in particular, for service time distributions whose tails decay slower than $e^{-\sqrt{t}}$. As the main results of this paper show, the behavior of $W_\infty$ for lighter subexponential service time distributions may include a third region where neither the heavy traffic approximation nor the heavy tail asympotic are valid, and where the higher order moments of the service time distribution start playing a role. The exact way in which these higher order moments appear in the distribution of $W_\infty$ is closely related to the large deviations behavior of an associated random walk and its corresponding Cram\'{e}r series. 

The approach that we take to understand the asymptotics of $P(W_\infty > x)$ over the entire line is to provide approximations that hold uniformly across all values of the traffic intensity for large values of the tail, or alternatively, uniformly across all tail values for traffic intensities close to one. From such uniform approximations it is possible to compute the exact thresholds separating the different regions of deviations of $W_\infty$, which for service time distributions decaying slower than $e^{-\sqrt{t}}$ are simply the heavy traffic and heavy tail regions, and, for lighter subexponential distributions, include a third region where neither the heavy traffic approximation nor the heavy tail asymptotic hold.  Similar uniform approximations have been derived in the literature for the tail distribution of a random walk with subexponential increments in \cite{Bor00}, \cite{Bor00b}, and \cite{Roz_93}, where the uniformity is on the number of summands for large values of the tail or across all tail values as the number of summands grows to infinity. The results in the paper are in some sense the equivalent for the single-server queue.

To explain the idea behind our main results let us recall that one can approximate the  tail distribution of the steady-state waiting time of a single-server queue with subexponential processing times, $P(W_\infty > x)$, via two well known approximations: the heavy traffic approximation and the heavy tail asymptotic
$$\exp\left\{- \frac{2(E\tau_1 - EV_1)}{\var \tau_1 + \var V_1} \, x \right\} \qquad \text{and} \qquad \frac{\rho}{1-\rho} \int_{x}^\infty \frac{P(V_1 > t)}{EV_1} dt,$$
respectively, where $V_1$ denotes the service time, $\tau_1$ the inter-arrival time, and $\rho$ the traffic intensity of the queue. We refer the reader to Chapter X of \cite{Asm2003} and the references therein for more details on the history and the exact formulation of these limit theorems. The heavy traffic approximation is valid for the general GI/GI/1 queue and can be derived by using a functional Central Limit Theorem type of analysis (see, e.g. \cite{IgWh70a, IgWh70b}). The theorem that justifies this approximation is obtained by taking the limit as the traffic intensity approaches one and is applicable for bounded values of $x$. The heavy tail asymptotic is valid for the GI/GI/1 FIFO queue with subexponential service time distribution (see, e.g., \cite{EmVe82}), and is obtained by taking the limit as $x$ goes to infinity for a fixed traffic intensity, that is, it is applicable for large values of $x$.  One can then think of combining these two approximations to obtain an expression that is uniformly valid on the entire positive axis.  

The approach we take in the derivation of the main theorems is to start with the Pollaczek-Khintchine formula for the distribution of the steady-state waiting time of the M/G/1 queue, which expresses it as a geometric random sum, and use the asymptotics for the tail distribution of the random walk. One of the difficulties in obtaining uniform asymptotics for the distribution of $W_\infty$ lies in the highly complex asymptotic behavior of the random walk. Surprisingly, most of the cumbersome details of the asymptotics for the random walk disappear in the queue, but showing that this is indeed the case requires a considerable amount of work.  The qualitative difference between queues with service time distributions with tails decaying slower than $e^{-\sqrt{t}}$ and their lighter-tailed counterparts comes from the asymptotic behavior of the random walk associated to the geometric random sum.  The function $e^{-\sqrt{t}}$ has been identified as a threshold in the behavior of heavy tailed sums and queues in \cite{Nagaev_69a, Bor00}, and \cite{Jel_Mom_03, Jel_Mom_04, Bal_Dal_Klupp_04}, respectively, to name a few references, and we provide here yet another example. 

As mentioned before, the approximations we provide can be used to derive the exact regions where the heavy traffic and heavy tail approximations hold, but we do not provide the details in this paper since our focus is on deriving uniform expressions for $P(W_\infty > x)$ under minimal conditions on the service time distribution. The setting we consider is the same from \cite{Jel_Mom_04, Bal_Dal_Klupp_04} where the busy period was analyzed. More detailed comments about the third region of asymptotics that arises when the service time distribution is lighter than $e^{-\sqrt{t}}$ can be found in Remark \ref{R.MainRemarks} right after Theorem \ref{T.Main}. For clarity, we state all our assumptions and notation in the following section, and our main results in Section \ref{S.MainResults}. 

Finally, we mention that the expressions given in the main theorems can be of practical use as numerical approximations for $P(W_\infty > x)$, and based on simulation experiments done for service times with a Pareto ($\alpha > 3$) or Weibull ($0 < \alpha < 1/2$) distribution, they seem to perform very well (see Section 4 in \cite{OlBlGl_10}). It is worth pointing out that the uniform approximations given here are far superior than the heavy traffic or heavy tail approximations individually even in the regions where these are valid, which is to be expected since they are based on the entire Pollaczek-Khintchine formula; they are also easy to compute given the integrated tail distribution of the processing times and its first few moments (cumulants).

\section{Model Description} \label{S.ModelDescription} 

Let $(W_n(\rho): n\geq 0)$ be the waiting time sequence for an M/G/1 FIFO
queue that is fed by a Poisson arrival process having arrival rate $\lambda = \rho/EV_1$ and independent iid processing times $(V_n: n\geq 0)$.  Provided that the traffic intensity $\rho$ is smaller than one, we denote by $W_\infty(\rho)$ the steady-state waiting time of the queue. We assume that $G(x) = P(V_1 \leq x)$ is such that its integrated tail distribution, given by $F(x) = \int_0^ x \overline{G}(t) dt/E V_1$ is subexponential, where $\overline{G}(t) = 1- G(t)$. The sequence $\{X_i\}_{i \geq 1}$ will denote iid random variables having distribution $F$. 

Define $Q(t) = -\log \overline{F}(t)$ to be the cumulative hazard function of $F$ and let $q(t) = (EV_1)^{-1} \overline{G}(t)/ \overline{F}(t)$ be its hazard rate function; note that $q$ is the density of $Q$. Just as in \cite{Bal_Dal_Klupp_04} and \cite{Bal_Klupp_04}, we define the hazard rate index
\begin{equation} \label{eq:r_Def} 
r = \limsup_{t \to \infty} \frac{t q(t)}{Q(t)}.
\end{equation}

All the results presented in this paper hold for subexponential distributions $G$ (its corresponding integrated tail distribution $F$) satisfying the following assumption. 

\begin{assumption} \label{A.Hazard}
\begin{enumerate} 
\item $0 \leq r < 1$; 
\item $\liminf_{t \to \infty} t q(t) > a(r)$, where $a(r) = \begin{cases} 2, & \text{ if } r = 0, \\ 4/(1-r), & \text{ if } r \neq 0. \end{cases}$
\end{enumerate}
\end{assumption}

Assumption \ref{A.Hazard} is consistent with Conditions B and C in \cite{Bal_Dal_Klupp_04} and \cite{Bal_Klupp_04}, respectively, and also very closely related to Definition 1 in \cite{Jel_Mom_04}.  All three of these works study the asymptotic behavior of random sums with subexponential increments applied to either the study of the busy period of a GI/GI/1 queue or to ruin probabilities in insurance.  Also, by Proposition 3.7 in \cite{Bal_Dal_Klupp_04}, Assumption \ref{A.Hazard} (a.) is equivalent to the function $Q(t)/t^{r+\delta}$ being decreasing on $t \geq t_0 \geq 1$ for any $0 < \delta < 1-r$, which is the same as equation (3) in \cite{Roz_93}, where uniform asymptotics for the tail behavior of a random walk with subexponential increments were derived. As mentioned in \cite{Bal_Dal_Klupp_04} and \cite{Bal_Klupp_04}, Lemma~3.6 in \cite{Bal_Dal_Klupp_04} implies that $\sup\{ k: E[X_1^k] < \infty\} \geq \liminf_{t \to \infty} t q(t)$, so Assumption~\ref{A.Hazard} (b.) guarantees that $E[X_1^k] < \infty$ for all $k \leq a(r)$. Furthermore, Assumption \ref{A.Hazard} (b.) and Lemma 3.6 in \cite{Bal_Dal_Klupp_04} together imply that $\liminf_{t \to \infty} Q(t)/\log t \geq \liminf_{t \to \infty} t q(t) > a(r)$, which in turn implies that for some $\beta > a(r) \geq 2$ and $t_0 > 1$,  \vspace{-3pt}
\begin{equation} \label{eq:LowerBoundQ}
Q(t) \geq \beta \log t \quad \text{for all }  t \geq t_0. \vspace{-3pt}
\end{equation}

Although the tail distribution of the busy period in queues with heavy tailed service times is related to that of its waiting time in the sense that it is determined by $\overline{G}(x)$ (see \cite{Zwart_01, Jel_Mom_04, Bal_Dal_Klupp_04, Bal_Klupp_04}), the approach to its analysis is rather different from that of the waiting time, so the only connection between the results in this paper and those cited above is the setting.

This family of distributions includes in particular all regularly varying distributions, $\overline{F}(x) = x^{-\alpha} L(x)$ with $\alpha > 2$, and all semiexponential distributions, $\overline{F}(x) = e^{-x^\alpha L(x)}$ with $0 \leq \alpha < 1$; in these definitions $L$ is a slowly varying function. The regularly varying case with $\alpha > 1$ was covered in detail in \cite{OlBlGl_10}. Some subexponential distributions that do not satisfy Assumption \ref{A.Hazard} are those decaying ``almost" exponentially fast, e.g. $\overline{F}(x) = e^{-x/\log x}$. 

Before stating our main results in the following section, we introduce some more notation that will be used throughout the paper. Let $\mu = EX_1 = EV_1^2/(2 EV_1)$ and $\sigma^2 = \var(X_1) = EV_1^3/(3EV_1) - (EV_1^2/ (2EV_1))^2$. Also, define
\begin{equation} \label{eq:kappa}
\kappa = \max\left\{ l \in \{0, 1, \dots\}: \limsup_{t \to \infty} \frac{Q(t)}{t^{l/(l+1)}} > 0 \right\} + 2,
\end{equation}
and note that by Proposition 3.7 in \cite{Bal_Dal_Klupp_04}, $Q(t)/t^{r+\delta}$ is eventually decreasing for all $\delta > 0$, which implies that $Q(t)/t^{r+\delta} \to 0$ for all $\delta > 0$. In particular, for $r \in [0, 1/2)$ this implies that $Q(t)/t^{1/2} \to 0$ and $\kappa = 2$. Also, we obtain the relation $(\kappa-2)/(\kappa-1) \leq r$, or equivalently, $\kappa \leq (2-r)/(1-r)$. Combining this observation with our previous remark about Assumption \ref{A.Hazard} (b.) gives that for $0 < r < 1$ and any $2 \leq s \leq (2+r)/(1-r)$ we have $E[X_1^{\kappa+s}] < \infty$.

\section{Main results} \label{S.MainResults}

As mentioned in the introduction, the idea of this paper is to use the Pollaczek-Khintchine formula to write the distribution of the steady-state waiting time as 
\begin{equation} \label{eq:Poll-Khin} 
P(W_\infty(\rho) > x) = \sum_{n=0}^\infty (1-\rho) \rho^n P(S_n > x), 
\end{equation}
where $S_n = X_1 + \dots + X_n$ and $\{ X_i\}_{i \geq 1}$ is a sequence of iid random variables having distribution $F$, and then approximate $P(S_n > x)$ by an appropriate asymptotic expression. The approximation that we use for $P(S_n > x)$ depends on the pair $(x,n)$, and for the heavy-tailed setting that we consider here, one can identify four different regions of deviations.

The first region is the one described by the Central Limit Theorem (CLT), i.e., where \vspace{-3pt}
\begin{equation} \label{eq:CLT}
P(S_n > x) \approx 1 - \Phi\left( (x - n\mu)/\sqrt{n} \sigma \right), 
\end{equation}
and $\Phi(\cdot)$ is the standard normal distribution function. The second region is the so-called Cram\'{e}r region, which provides additional correction terms to the CLT approximation. When the distribution $F$ has finite exponential moments, the Cram\'{e}r approximation is given by 
$$\frac{P\left((S_n - n\mu)/\sigma  > x \right)}{1 - \Phi(x/\sqrt{n})} = \exp\left( n \lambda\left( \frac{x}{n} \right)  \right) \left(1 + O\left( \frac{x/\sqrt{n}+1}{\sqrt{n}} \right) \right),$$
where $\lambda(t) \triangleq \sum_{j=3}^\infty \lambda_j t^{j}/j!$ is a power series with coefficients depending on the cumulants of $X_1$ known in the literature as the Cram\'{e}r series (see \cite{Petrov1975}, Chapter VIII, \S 2, or \cite{Nagaev_65}). When $F$ is heavy-tailed,  nevertheless, $\lambda(t)$ diverges for all $t$ and a truncated form of this series replaces $\lambda(\cdot)$. In the setting of this paper, only the terms up to $\kappa$ (as defined by \eqref{eq:kappa}) are needed, and we obtain the following approximation for $P(S_n > x)$ 
\begin{equation} \label{eq:CramerAsym} 
P(S_n > x) \approx \left(1 - \Phi((x-n\mu)/\sqrt{n}\sigma) \right) e^{\frac{1}{2} \left( \frac{x-n\mu}{\sigma\sqrt{n}} \right)^2 + nQ_\kappa\left( \frac{x-n\mu}{\sigma n} \right)}, 
\end{equation}
where 
\begin{equation} \label{eq:Q_kappa} 
Q_\kappa(t) = \sum_{j=2}^\kappa \frac{\lambda_j t^j}{j!},
\end{equation}
$\lambda_2 = -1$, and $\{\lambda_j\}_{j \geq 3}$ are the coefficients of the Cram\'{e}r series corresponding to $Y_1 = (X_1 - \mu)/\sigma$. Note that if $\kappa = 2$, then approximations \eqref{eq:CLT} and \eqref{eq:CramerAsym} are the same. 

The third region is known in the literature as the ``intermediate domain", and the exact asymptotics for $P(S_n > x)$ in this region can be considerably complicated (see \cite{Bor00b} and \cite{Roz_93} for more details). Fortunately, the range of values corresponding to this region in the Pollaczek-Khintchine formula is negligible with respect to the rest, and we will only need to use an upper bound for $P(S_n > x)$. The fourth and last region is the heavy-tailed region, also know as the ``big jump domain" (see \cite{Bor00} and \cite{Den_Die_Shn_08}, for example), where 
$$P(S_n > x) \approx n \overline{F}(x-n\mu).$$

In the discussion above we purposefully omitted describing the boundaries between the four different regions, since that alone requires introducing various (complicated) functions and their corresponding asymptotic behavior. In terms of the Pollaczek-Khintchine formula, it is enough to consider simpler versions of those thresholds. We start by defining the functions 
\begin{equation*} 
\omega_1(t) = t^2/(Q(t) \vee 1) \quad \text{and} \quad \omega_2(t) = t^2/(Q(t) \vee 1)^2,
\end{equation*}
where $x \vee y = \max\{x, y\}$ ($x \wedge y = \min\{x,y\}$), and let $\omega_i^{-1}(t) = \inf\{u \geq 0: t \leq \omega_i(u)\}$, $i = 1,2$. We give below some properties of the $\omega^{-1}$ operator; the proof is omitted but can be derived through straightforward analysis. 

\begin{lem} \label{L.RightInverse}
For any continuous function $\omega: [0,\infty) \to [0,\infty)$ such that $\lim_{t \to \infty} \omega(t) = \infty$, define the function $\omega^{-1}: [0,\infty) \to [0,\infty)$ as $\omega^{-1}(t) = \inf\{ u \geq 0: t \leq \omega(u) \}$. Then, the following are true 
\begin{enumerate}
\item $\omega^{-1}$ is monotone non decreasing and left-continuous.
\item $\omega^{-1}$ is a right inverse of $\omega$, that is, $\omega(\omega^{-1}(t)) = t$, for all $t \geq \omega(0)$. 
\item if $\overline{\omega}(t) = \sup_{0 \leq s \leq t} \omega(s)$, then $\omega^{-1}$ is a right inverse of $\overline{\omega}$ for all $t \geq \omega(0)$. 
\item $\omega^{-1}(\omega(t)) \leq t$ for all $t \geq 0$. 
\end{enumerate} 
\end{lem}

We now define the threshold functions delimiting the different regions of asymptotics for $P(S_n > x)$. Let
$$K_r(x) = \begin{cases}
\lfloor (x -\omega_2^{-1}(x))/\mu \rfloor \vee 0, & \text{if } r \in [0, 1/2), \\
\lfloor \min\{\omega_2(x), x/(2\mu) \} \rfloor \vee 0, & \text{if } r \in [1/2, 1),
\end{cases} $$
$$M(x) = \lfloor (x - \omega_1^{-1}(x))/\mu \rfloor \vee 0, \qquad \text{and} \qquad N(x) = \lfloor (x - \sqrt{x\log x})/\mu \rfloor \vee 0.$$
Note that if $r \in [0,1/2)$ and if $\delta > 0$ is such that $r + \delta < 1/2$, then $\omega_2(t) \geq Ct^{2(1-r-\delta)}$ for some constant $C > 0$, so $\omega_2^{-1}(t) \leq Ct^{1/(2(1-r -\delta))} = o(t)$. Also, provided $r + \delta \in (0,1)$, $\omega_1(t) \geq Ct^{2-r -\delta}$ so $\omega_1^{-1}(t) \leq Ct^{1/(2-r-\delta)} = o(t)$. Therefore, all three functions are strictly positive for large enough $x$. Moreover, as mentioned in the previous section, Assumption \ref{A.Hazard} (b.) implies that $Q(t) \geq \beta \log t$ for all $t \geq t_0$ for some $\beta > a(r) \geq 2$, which gives $\omega_2(t) \leq \omega_1(t) \leq \beta^{-1} t^2/\log t$, and $\omega_2^{-1}(x) \geq \omega_1^{-1}(x) \geq \sqrt{(\beta/2) x \log x}$ for all $x \geq x_0$. We then have that $K_r(x) \leq M(x) \leq N(x)$ for all large enough $x$. 

To better understand the definitions of the threshold functions consider the zero mean case with finite variance, for which it is well known that the CLT approximation \eqref{eq:CLT} holds for $x = O(\sqrt{n})$; translating into the positive mean case, this gives rise to the threshold $n \geq (x- \sqrt{c x})/\mu$ for some constant $c > 0$. Substituting the constant by $\log x$ gives the threshold $N(x)$. The Cram\'er approximation \eqref{eq:CramerAsym} holds, in the zero mean case, uniformly for $x \leq \sigma_1(n)$, where $\sigma_1(n)$ is the solution to the equation $x^2 = nh(x)$ and $E[ e^{h(X_1)} \Indicator(X_1 \geq 0)] < \infty$ (see, \cite{Borov_Borov_2008} \S 5.1 and the references therein); taking $h = Q$ gives the threshold $n \geq\omega_1(x)$, and translating into the positive mean case gives $n \geq (x-\omega_1^{-1}(x/\mu))/\mu$. Note that $E[e^{Q(X_1)} \Indicator(X_1 \geq 0)] = \infty$ but, for example, $E[e^{Q(x) -2\log Q(x)} \Indicator(X_1 \geq 0)] < \infty$, so this choice of $h$ is very close to  the boundary of the region. Finally, the asymptotic $P(S_n > x) \sim n \overline{F}(x)$ as $x \to \infty$ is known to hold, in the mean zero case, for $n \leq c \omega_2(x)$ (see Theorem 1 in \cite{Bal_Klupp_04}), and provided that $\omega_2^{-1}(x) = o(x)$ (which occurs when $r \in [0, 1/2)$), the translation into the positive mean case gives the threshold $n \leq (x-\omega_2^{-1}(x/\mu))/\mu$. When $r \in [1/2, 1)$ we cannot guarantee that $\omega_2(x) \leq x/\mu$, so by taking the minimum between $\omega_2(x)$ and $x/(2\mu)$ we satisfy the condition $n \leq \omega_2(x-n\mu)$, and therefore our choice of $K_r(x)$. We point out that since the thresholds do not need to be too precise, we ignored the constant $\mu$ inside of $\omega_1^{-1}$ and $\omega_2^{-1}$ in the definitions of $M(x)$ and $K_r(x)$, respectively, to simplify the expressions.

The first asymptotic for $P(W_\infty(\rho) > x)$ we propose is given by the following expression based on the Pollaczek-Khintchine formula, for $\kappa = 2$, 
\begin{equation} \label{eq:Z_Def_2}
Z_\kappa(\rho,x) = \sum_{n=1}^{K_r(x)} (1-\rho)\rho^n n \overline{F}(x-n\mu) + E\left[ \rho^{a(x,Z)} \Indicator(\sigma Z \leq \sqrt{\mu} \omega_1^{-1}(x)/\sqrt{x}) \right],
\end{equation}
and for $\kappa > 2$,
\begin{align}
Z_\kappa(\rho,x) &= \sum_{n=1}^{K_r(x)} (1-\rho)\rho^n n \overline{F}(x-n\mu) + \frac{\sigma \sqrt{x}}{\sqrt{2\pi \mu}} \sum_{n=M(x)+1}^{N(x)} (1-\rho) \rho^n \frac{e^{n Q_\kappa \left( \frac{x-n\mu}{\sigma n} \right)} }{x-n\mu} \notag  \\
&\hspace{3mm} + E\left[ \rho^{a(x,Z)} \Indicator(\sigma Z \leq \sqrt{\mu\log x}) \right], \label{eq:Z_Def_big2}
\end{align}
where $Z\sim$ N(0,1) and $a(x,z) = \left(x - \sigma z \sqrt{x/\mu}\right)/\mu$. Throughout the paper we use the convention that $\sum_{n=A}^B a_n \equiv 0$ whenever $B < A$. Our first theorem is formally stated below.

\begin{thm} \label{T.SumApprox}
Suppose Assumption \ref{A.Hazard} is satisfied, and define $Z_\kappa(\rho,x)$ according to \eqref{eq:Z_Def_2} and \eqref{eq:Z_Def_big2}. Then, 
$$\lim_{x \to \infty} \sup_{0 < \rho < 1} \left| \frac{P(W_\infty(\rho) > x)}{Z_\kappa(\rho,x) } - 1 \right|  = 0.$$
\end{thm}

\begin{remark}
(i) We point out that the approximation given by $Z_\kappa(\rho,x)$ is explicit in the sense that given the exact form of $F$, all the functions and parameters involved in the approximation are known. In particular, 
$$E\left[ \rho^{a(x,Z)} \Indicator\left(\sigma Z \leq \sqrt{\mu} \, T \right) \right] =  \rho^{\frac{x}{\mu}} e^{\frac{\sigma^2 (\log\rho)^2 x}{2 \mu^3}} \Phi\left( \frac{\sqrt{\mu} \, T}{\sigma} + \frac{\sigma \sqrt{x}}{\mu^{3/2}}\log \rho \right).$$
(ii) This approximation is suitable for numerical computations since it involves no integrals or infinite sums. (iii) With some additional work once can show that the first term in \eqref{eq:Z_Def_2} and \eqref{eq:Z_Def_big2} can be replaced by
$$\overline{F}(x) \sum_{n=1}^{K_r(x)} (1-\rho)\rho^n n,$$
which is asymptotically equivalent to the heavy tail asymptotic $\rho \overline{F}(x)/(1-\rho)$ for  appropriate values of $(x,\rho)$. We choose not to use this simpler expression because our numerical experiments show that it would result in a less accurate approximation for $P(W_\infty(\rho) > x)$.  (iv) For the case $\kappa > 2$, the middle term in \eqref{eq:Z_Def_big2} provides a direct connection between the Cram\'{e}r region of asymptotics for the random walk and the asymptotic behavior of the queue, and also reiterates the qualitative difference between distributions decaying slower than $e^{-\sqrt{x}}$ $(\kappa = 2)$ and those with lighter tails (see \cite{Nagaev_69a}, \cite{Jel_Mom_03}, \cite{Jel_Mom_04}, to name some references).  (v) Unlike the next approximation, given in Theorem \ref{T.Main}, the expression $Z_\kappa(\rho,x)$ does not work as a uniform asymptotic in $x > 0$ as $\rho \nearrow 1$ for $P(W_\infty(\rho) > x)$, since it does not converge to one for small values of $x$. Nevertheless, it is not difficult to show that
$$\lim_{\rho \nearrow 1} \sup_{x \geq \hat x(\rho)}  \left| \frac{P(W_\infty(\rho) > x)}{Z_\kappa(\rho,x) } - 1 \right|  = 0$$
for any $\hat x(\rho) \to \infty$ as $\rho \nearrow 1$ (see the proof of Lemma 3.3 in \cite{OlBlGl_10}). 
\end{remark}

In the same spirit of the heavy traffic approximations in \cite{Wh95} and \cite{BlGl07}, where $P(W_\infty(\rho) > x)$ is approximated by $e^{x S(\rho)}$ where $S(\rho)$ is a power series in $(1-\rho)$, our second result derives an approximation that involves a power series in $\log\rho$. The number of terms in this power series is also determined by $\kappa$ (as in the definition of $Q_\kappa(\cdot)$), and its coefficients are closely related to those of the Cram\'{e}r series.  This other approximation substitutes the second term in \eqref{eq:Z_Def_2} and the second and third terms in \eqref{eq:Z_Def_big2} by their corresponding asymptotic expression as $\rho \nearrow 1$. The intuition behind this substitution is that these terms only dominate the behavior of $Z_\kappa(\rho,x)$ when the effects of the heavy traffic are more important than those of the heavy tails. Besides unifying the cases $\kappa = 2$ and $\kappa > 2$, this new approximation will also have the advantage of being uniformly good for $x > 0$ as $\rho \nearrow 1$. In order to state our next theorem we need the following definitions. 

Let 
\begin{equation} \label{eq:LambdaDef}
\Lambda_\rho(t) = (1-t) \log\rho + \sum_{i=2}^{\kappa} \sum_{j=2}^i \frac{\lambda_j \mu^j}{j! \sigma^j} \binom{i-1}{i-j} t^i,
\end{equation}
where $\lambda_2 = -1$, and $\{\lambda_j\}_{j \geq 3}$ are the coefficients of the Cram\'{e}r series corresponding to $Y = (X_1 - \mu)/\sigma$. This function can be obtained by expanding $(1-t)Q_\kappa(\mu \sigma^{-1} t/ (1-t))$ into powers of $t$; the details can be found in Lemma \ref{L.Lambda}. We also need to define $u(\rho)$ to be the smallest positive solution to $\Lambda_\rho'(t) = 0$. Some properties of $\Lambda_\rho$ and $u(\rho)$ are given in the following lemma. 

\begin{lem} \label{L.ustar}
Define $\Lambda_\rho$ according to \eqref{eq:LambdaDef} and let $u(\rho)$ be the smallest positive solution to $\Lambda_\rho'(t) = 0$. Then $\Lambda_\rho$ is concave in a neighborhood of the origin, 
$$u(\rho) = \sum_{n=1}^\infty \frac{b_n}{n!} (\log\rho)^n$$
and
$$\Lambda_\rho(u(\rho)) = \begin{cases}
\log\rho + \frac{\sigma^2}{2\mu^2}( \log \rho)^2, & \kappa = 2, \\
\log\rho + \frac{\sigma^2}{2\mu^2}( \log \rho)^2  + O (|\log\rho|^3), & \kappa > 2, 
\end{cases}$$
as $\rho \nearrow 1$, where $b_1 = - \frac{\sigma^2}{\mu^2}$ and for $n \geq 2$, 
\begin{align*}
b_n &= \frac{d^{n-1}}{d t^{n-1}} \left. \left( \frac{t}{P_\kappa(t)} \right)^n \right|_{t = 0} = \sum_{(m_1, \dots, m_{n-1}) \in \mathcal{A}_{n-1}} (n+s_{n-1}-1)!  (-1)^{n} \left( \frac{\sigma^2}{\mu^2} \right)^{n+s_{n-1}}  \prod_{j=1}^{n-1} \frac{1}{m_j!} (a_j \Indicator(j \leq \kappa-2))^{m_j},
\end{align*}
$\mathcal{A}_n = \{ (m_1, \dots, m_{n}) \in \mathbb{N}^{n} : \, 1m_1 + 2 m_2 + \dots + n m_n = n \}$, $s_n = m_1 + \dots + m_n$, and
$$P_\kappa(t) = \Lambda_\rho'(t) + \log\rho \triangleq t \sum_{j=0}^{\kappa-2}  a_j t^{j}.$$
\end{lem}

The second approximation for $P(W_\infty(\rho) > x)$ that we propose is 
\begin{equation} \label{eq:A_Def}
A_\kappa(\rho,x) = \sum_{n=1}^{K_r(x)} (1-\rho) \rho^n n \overline{F}(x-n\mu) + e^{\frac{x}{\mu} \Lambda_\rho (w(\rho,x))} ,
\end{equation}
where $w(\rho,x) = \min\{ u(\rho), \omega_1^{-1}(x)/x\}$. The precise statement of our result is given below.

\begin{thm} \label{T.Main}
Suppose Assumption \ref{A.Hazard} is satisfied, and define $A_\kappa(\rho,x)$ according to \eqref{eq:A_Def}. Then, 
$$\lim_{x \to \infty} \sup_{0 < \rho < 1} \left| \frac{P(W_\infty(\rho) > x)}{A_\kappa(\rho,x) } - 1 \right|  = 0.$$
Moreover, 
$$\lim_{\rho \nearrow 1} \sup_{x > 0} \left| \frac{P(W_\infty(\rho) > x)}{A_\kappa(\rho,x) } - 1 \right|  = 0.$$
\end{thm}

\begin{remark} \label{R.MainRemarks}
(i) As mentioned earlier, the difference between $Z_\kappa(\rho,x)$ and $A_\kappa(\rho,x)$ is in the terms that correspond to the behavior of the queue when the effects of the heavy traffic dominate those of the heavy tails. In particular, what prevents $Z_\kappa(\rho,x)$ from being uniformly good for all values of $x$ as $\rho \nearrow 1$ is that if $x$ is bounded, then the second term in \eqref{eq:Z_Def_2} and the second and third terms in \eqref{eq:Z_Def_big2} do not converge to one when $\rho \nearrow 1$, which can be fixed by substituting them by their asymptotic expression as $\rho \nearrow 1$; evaluating $\Lambda_\rho$ at the value $w(\rho,x) = \min\{ u(\rho), \omega_1^{-1}(x)/x\}$ guarantees that the contribution of $e^{\frac{x}{\mu} \Lambda_\rho(w(\rho,x))}$ becomes negligible when the queue is in the heavy tail regime.  (ii) For analytical applications, Lemma \ref{L.ustar} states that $\Lambda_\rho(u(\rho))$ can be written as a power series in $\log\rho$ whose terms of order greater than $\kappa$ can be ignored. For numerical implementations, nonetheless, it might be easier to compute $u(\rho)$ by directly optimizing $\Lambda_\rho(t)$, since $\Lambda_\rho(t)$ is just a polynomial of order $\kappa$. (iii) By simply matching the leading exponents of the heavy tail asymptotic and the function $\frac{x}{\mu} \Lambda_\rho(u(\rho))$, that is, by solving the equation
$$\frac{x}{\mu} \log \rho = - Q(x),$$
we obtain that the heavy tail region is roughly $\mathcal{R}_1 = \{ (x,\rho): \rho < e^{-\mu Q(x)/x} \}$, whereas on $\mathcal{R}_2 = \{ (x,\rho): \rho > e^{-\mu Q(x)/x}\}$ one should use $e^{\frac{x}{\mu} \Lambda_\rho(u(\rho))}$ to approximate $P(W_\infty(\rho) > x)$. It follows that the heavy traffic region is given by the subset of $\mathcal{R}_2$ where $e^{\frac{x}{\mu} \Lambda_\rho(u(\rho))}$ is asymptotically equivalent to $e^{-\frac{x}{\mu} (1-\rho)}$, the heavy traffic approximation for the M/G/1 queue. We note that when $\kappa =2$, the heavy traffic region is the entire $\mathcal{R}_2$, but it is a strict subset of $\mathcal{R}_2$ if $\kappa > 2$, in which case a third region of asymptotics arises where neither the heavy traffic nor the heavy tail approximations are valid.  (iv) As mentioned before, the coefficients of $\Lambda_\rho(t)$ can be easily obtained from the first $\kappa-2$ coefficients of the Cram\'{e}r series of $Y = (X_1 -\mu)/\sigma$, which in turn can be obtained from the cumulants of $Y$. 
\end{remark}

We end this section with a formula that can be used to compute the coefficients of the Cram\'{e}r series.

\subsection{Cram\'{e}r Coefficients}

The following formula taken from \cite{Roz_99} can be used to recursively compute the coefficients in the Cram\'{e}r series, and we include it only for completeness. 

\begin{prop}
Let $Y$ be a random variable having $EY = 0$, $\var(Y) = 1$, and cumulants $\gamma_1, \gamma_2, \dots$. Let $\lambda_3, \lambda_4, \dots$ be the coefficients of the (formal) Cram\'{e}r series of $Y$, i.e., $\lambda(t) = \sum_{j=3}^\infty \lambda_j t^j / j!$. Let $\mathcal{A}_j = \{ (n_1, \dots, n_{j}) \in \mathbb{N}^{j} : \, 1n_1 + 2 n_2 + \dots + j n_j = j \}$. Then, for $j \geq 3$ and $s_{j-2} = n_1 + \dots + n_{j-2}$,
\begin{align*}
\lambda_j &= \sum_{(n_1,\dots, n_{j-2}) \in \mathcal{A}_{j-2}} (j + s_{j-2} -2)! (-1)^{s+1} \prod_{m=1}^{j-2} \frac{1}{n_m!} \left( \frac{\gamma_{m+2}}{(m+1)!} \right)^{n_m},
\end{align*}
 
\end{prop}

The first four coefficients are given by
\begin{align*}
\lambda_3 &=  \gamma_3, \qquad \lambda_4 = \gamma_4 - 3\gamma_3^2, \qquad \lambda_5 = \gamma_5 - 10 \gamma_4 \gamma_3 + 15 \gamma_3^3, \\
\lambda_6 &= \gamma_6 - 15\gamma_5 \gamma_3 - 10 \gamma_4^2 + 105 \gamma_4 \gamma_3^2 - 105 \gamma_3^4
\end{align*}

The rest of the paper consists mostly of the proofs of all the results in Section \ref{S.MainResults} and is organized as follows.  Section~\ref{S.UniformRW} states an approximation for $P(S_n > x)$ that is valid for all pairs $(x,n)$ and that will be used to derive uniform asymptotics for $P(W_\infty(\rho) > x)$.  Section \ref{S.SumApproxProof} contains the proof of Theorem \ref{T.SumApprox}; and Section \ref{S.MainProof} contains the proofs of Lemma \ref{L.ustar} and Theorem \ref{T.Main}.  We conclude the paper by giving a couple of numerical examples comparing the two suggested approximations for the tail distribution of $W_\infty(\rho)$, $Z_\kappa(\rho,x)$ and $A_\kappa(\rho,x)$, in Section~\ref{S.Numerical}. A table of notation is included at the end of the paper. 

\bigskip

\section{Uniform asymptotics for $P(S_n > x)$} \label{S.UniformRW} 

In this section we will state the uniform approximation for $P(S_n > x)$ that we will substitute in the Pollaczek-Khintchine formula \eqref{eq:Poll-Khin} outside of the heavy-tail region. This approximation was derived in \cite{Roz_93} for mean zero and unit variance random walks and it works on the whole positive line as $n \to \infty$.  Although rather complicated as an approximation for $P(S_n > x)$, it will be useful in the derivation of simpler expressions for the queue with the level of generality that we described in Section \ref{S.ModelDescription}. For the heavy-tail region (small values of $n$) we will use in section \ref{SS.FirstApprox} a result from \cite{Bal_Klupp_04} to prove that $P(S_n > x) = n \overline{F}(x-n\mu) (1+o(1))$ as $x \to \infty$ uniformly in the region $1 \leq n \leq K_r(x)$. 

We start by stating the assumptions needed for the mean zero and unit variance random walk, and after giving the approximation in this setting we will show that under Assumption \ref{A.Hazard}, the random variable $Y_1 = (X_1 - \mu)/\sigma$ satisfies these conditions. Then we will apply a slightly modified version of the approximation to the positive mean case and we will show that it holds uniformly in the region $n \geq K_r(x)$.

The notation $f(t) \asymp g(t)$ as $t \to \infty$ means $0 < \liminf_{t \to \infty} f(t)/g(t) \leq \limsup_{t \to \infty} f(t)/g(t) < \infty$. We will also use $C$ to denote a generic positive constant, i.e., $C = 2 C$, $C = C +1$, etc. 

\begin{assumption} \label{A.StdCase}
Let $Y$ be a random variable with $E[Y] = 0$, $\var(Y) = 1$ and tail distribution 
$$1- V(t) = \overline{V}(t) \asymp  \frac{D(t)}{t^2} \, e^{-\tilde Q(t)}, \qquad t \to \infty,$$
where $D(t) = \int_{|u|<t} t^2 dV(dt)$, $\tilde Q$ has Lebesgue density $\tilde q$, and satisfies
$$\limsup_{t \to \infty} \frac{t \tilde q(t)}{\tilde Q(t)} \triangleq \tilde r < 1 \qquad \text{and} \qquad  \liminf_{t \to \infty} \tilde Q(t)/\log t > \tilde r/(1-\tilde r).$$
Suppose further that $E[ |Y|^{\tilde \kappa+1}] < \infty$, where
$$\tilde \kappa = \max\left\{ l \in \{0, 1, 2,\dots\} : \limsup_{z \to \infty} \frac{\tilde Q(z)}{z^{l/(l+1)}} > 0 \right\} + 2.$$
\end{assumption}

Throughout this section let $Q_{\tilde \kappa}(t) = \sum_{j=2}^{\tilde\kappa} \lambda_j t^j / j!$, where $\lambda_2 = -1$ and $\{\lambda_j\}_{j \geq 3}$ are the coefficients of the Cram\'er series of $Y$, and let $\tilde{S}_n = Y_1 + \dots + Y_n$, where $\{Y_i\}$ are iid with common distribution $V(t)$. We also define the functions
\begin{equation} \label{eq:b_Def}
b(t) = t^2/(\tilde Q(t) \vee 1), \qquad \text{and} \qquad b^{-1}(t) = \inf\{ u \geq 0: t \leq b(u)\}.
\end{equation}

We start by proving some properties about the functions $\tilde Q$, and $b^{-1}$.

\begin{lem} \label{L.g_properties}
Suppose Assumption \ref{A.StdCase} holds. Then, for any $s \in (\tilde r, 1)$ there exists a constant $t_0 \geq 1$ such that 
\begin{enumerate}
\item $\tilde Q(t)/ t^{s}$ is decreasing for all $t \geq t_0$, 
\item $b^{-1}(t) \leq t^{1/(2-s)}$ for all $t \geq t_0$,
\item $b^{-1}(ct) \leq c^{1/(2-s)} b^{-1}(t) \leq c b^{-1}(t)$ for all $t \geq t_0$ and any $c \geq 1$,
\item $b^{-1}(c t) \geq c b^{-1}(t)$ for all $t \geq t_0$ and any $c \leq 1$,
\end{enumerate}
Also, the following limit holds
\begin{enumerate} \setcounter{enumi}{4}
\item $\lim_{t \to \infty}  e^{-\tilde Q(t/ b^{-1}(t))} \tilde Q(b^{-1}(t)) = 0$, 
\end{enumerate}
\end{lem}

\begin{proof}
Part (a.) follows directly from Proposition 3.7 in \cite{Bal_Dal_Klupp_04}.  For part (b.) note that $\tilde Q(t)/ t^{s'}$ is eventually decreasing for any $\tilde r < s' < s$, so
$$\lim_{t \to \infty} \frac{\tilde Q(t)}{t^s} \leq \sup_{z \geq 1} \frac{\tilde Q(z)}{z^{s'}} \lim_{t \to \infty} \frac{1}{t^{s-s'}} = 0.$$
It follows that $\tilde Q(t) \leq t^s$ for all $t \geq t_0$ for some $t_0 > 0$. This in turn implies that $b(t) \geq t^{2-s}$ for all $t \geq t_0$, and therefore, $b^{-1}(t) \leq t^{1/(2-s)}$. 

For part (c.) note that Proposition 3.7 in \cite{Bal_Dal_Klupp_04} gives $\tilde Q(c b^{-1}(t)) \leq c^{s} \tilde Q(b^{-1}(t))$ for any $c \geq 1$ and all sufficiently large $t$, then
$$b(c^{1/(2-s)} b^{-1}(t)) = \frac{c^{2/(2-s)} (b^{-1}(t))^2}{\tilde Q(c^{1/(2-s)} b^{-1}(t))} \geq \frac{c (b^{-1}(t))^2}{ \tilde Q(b^{-1}(t))} = cb(b^{-1}(t)) = c t = b(b^{-1}(c t)).$$
It follows from noting that $b(t)$ is strictly increasing for large enough $t$, that
$$c^{1/(2-s)} b^{-1}(t) \geq b^{-1}(c t).$$

For part (d.) let $c \leq 1$ and define $u(x) = c^{-1} b(x)$, $v(x) = b(c^{-1}x)$. By Proposition 3.7 in \cite{Bal_Dal_Klupp_04}, $\tilde Q(x) \geq c^{s} \tilde Q(c^{-1} x)$, from where we obtain that
$$u(x) = \frac{c^{-1} x^2}{\tilde Q(x) \vee 1}  \leq \frac{c (c^{-1} x)^2}{c^{s} \tilde Q(c^{-1} x) \vee 1} = \frac{(c^{-1} x)^2}{c^{-1 + s} \tilde Q(c^{-1} x) \vee c^{-1}} \leq v(x).$$
It follows that $u^{-1}(x) \geq v^{-1}(x)$, where $u^{-1}(x) = \inf\{ t \geq 0: c x \leq b(t)\} = b^{-1}(cx)$ and 
$v^{-1}(x) = \inf \left\{ t \geq 0: x \leq b(c^{-1}t)\right\} = c \inf\{ t \geq 0: x \leq b(t)\} = c b^{-1}(x)$.

\newpage
For part (e.) let $\nu = \liminf_{t \to \infty} \tilde Q(t)/\log t > \tilde r/(1-\tilde r)$ and note that
\begin{align*}
\lim_{t \to \infty}  e^{-\tilde Q(t/ b^{-1}(t))} \tilde Q (b^{-1}(t)) &= \lim_{t \to \infty} e^{-\tilde Q(t/ b^{-1}(t))} \frac{( b^{-1}(t) )^2}{t}  = \lim_{t \to \infty} e^{-\tilde Q(b(b^{-1}(t))/b^{-1}(t))} \frac{(b^{-1}(t))^2}{b(b^{-1}(t))} \\
&= \lim_{u \to \infty} e^{-\tilde Q(b(u)/u)} \frac{u^2}{b(u)} = \lim_{u \to \infty} e^{-\tilde Q(u/\tilde Q(u))} \tilde Q(u) \\
&\leq  \lim_{u \to \infty} e^{-\nu\log(u/\tilde Q(u))} \tilde Q(u) \\
&= \lim_{u \to \infty} \left( \frac{\tilde Q(u)}{u^{\nu/(\nu+1)}}  \right)^{\nu+1}.
\end{align*}
By part (a) $\tilde Q(u) \leq C u^{s}$ for any $s > \tilde r$ and $u$ sufficiently large, and by assumption $\tilde r < \nu/(\nu+1)$, so simply choose $\tilde r < s < \nu/(\nu+1)$ to see that the last limit is zero.
\end{proof}

\bigskip

\begin{lem} \label{L.Roz_pi}
Suppose Assumption \ref{A.StdCase} holds.  Define
$$L(h) = \int_{-\infty}^{\sqrt{n}} e^{h t} dV(t), \quad \text{and} \quad H(z) = \inf_{h > 0} (n \ln L(h) - zh).$$
Then, for any constant $c >0$, 
\begin{equation} \label{eq:NewPi}
e^{H(z)} =  e^{ n Q_k \left(\frac{z}{n}\right)  } (1+o(1))
\end{equation}
as $n \to \infty$, uniformly for $\sqrt{n} \leq z \leq c b^{-1}(n)$. 
\end{lem}

\begin{proof}
Choose $0 < \delta < 1-\tilde r$ and set $s = \tilde r + \delta$. Define $\eta(z) = b^{-1}(z^2)$ and
\begin{equation} \label{eq:Roz_pi}
\pi(z,n) = \left(1 - \Phi(z/\sqrt{n}) \right) \Indicator(z \leq \sqrt{n}) + \left(1 - \Phi(z/\sqrt{n}) \right) e^{\frac{z^2}{2n} + H(z)}.
\end{equation}

Suppose first that $\tilde r \in [0, 1/2)$ and note that in this case $\tilde\kappa = 2$ and $nQ_{\tilde\kappa}(z/n) = -z^2/(2n)$. Note that we can choose $\delta$ above so that $s < 1/2$. Then, by Lemma~\ref{L.g_properties}~(a.), $\tilde Q(t)/ t^{s}$ decreases for all sufficiently large $t$. Also, 
$$\frac{z^2}{D(z)} V(-z) = 0 \qquad \text{for all } z > \mu,$$
and
\begin{align*}
D(n/\eta(\sqrt{n})) &= \int_{-\mu}^{n/\eta(\sqrt{n})} u^2 dV(t) = 1 - \int_{n/\eta(\sqrt{n})}^\infty ( \tilde q(u)+ 2/u) e^{-\tilde Q(u)}  du \\
&= 1 + O\left(e^{-\tilde Q(n/\eta(\sqrt{n}))} \right) = 1 + o\left( 1/\tilde Q(\eta(\sqrt{n})) \right) \qquad \text{(by Lemma \ref{L.g_properties} (e.))}
\end{align*}
as $n \to \infty$. Define $\chi_n = b^{-1}(n) = \eta(\sqrt{n})$ and note that 
$$\frac{\chi_n^2}{\tilde Q(\chi_n) n} = \frac{b(\chi_n)}{n} = 1.$$
Then, by Lemma 1a in \cite{Roz_93}, we have
\begin{align*}
\pi(z,n) &= \left(1 - \Phi(z/\sqrt{n}) \right) (1 + o(1)) \\
&=  \left(1 - \Phi(z/\sqrt{n}) \right) e^{\frac{z^2}{2 n}  + n Q_{\tilde\kappa} \left(\frac{z}{n}\right)  } (1+o(1))
\end{align*}
as $n \to \infty$, uniformly for $\sqrt{n} \leq z \leq \gamma \chi_n$, where $\gamma > 0$ is an arbitrary constant (see the statement of Remark~1 in \cite{Roz_89} to see that the constant $\gamma$ can be arbitrary). 

Suppose now that $\tilde r \in [1/2, 1)$ and recall that by assumption $E[|Y|^{\tilde\kappa+1}] < \infty$.  Then, by Lemma 1b in \cite{Roz_93}, 
\begin{align*}
\pi(z,n) &= \left(1 - \Phi\left(  z/\sqrt{n} \right) \right) e^{\sum_{\nu = 1}^{\tilde \kappa-2} \frac{\lambda_{\nu+2}}{(\nu+2)!} \frac{z^{\nu+2}}{ n^{\nu+1}}  } (1+o(1)) \\
&=  \left(1 - \Phi(z/\sqrt{n}) \right) e^{\frac{z^2}{2 n}  + n Q_{\tilde\kappa} \left(\frac{z}{n}\right)  } (1+o(1)) 
\end{align*}
as $n \to \infty$, uniformly for $\sqrt{n} \leq z \leq \gamma \eta(\sqrt{n})$, where $\gamma > 0$ is an arbitrary constant. To see that $\gamma$ can be arbitrary see Remark 1 in \cite{Roz_89} where the statement of the result is
$$e^{H(z)} = e^{n Q_{\tilde\kappa}\left( \frac{x}{n} \right)} (1+o(1))$$
as $n \to \infty$, uniformly for $\sqrt{n} \leq z \leq \Lambda_n$, for a function $\Lambda_n$ that in \cite{Roz_93} is taken to be $\Lambda_n = \eta(\sqrt{n})$, and verify that all the arguments go through if we let $\Lambda_n = \eta(\sqrt{\bar\gamma n})$ for any constant $\bar\gamma > 0$. Then, use Lemma \ref{L.g_properties} (c.) to see that $\eta(\sqrt{\bar\gamma n}) = b^{-1}(\bar\gamma n) \leq (\bar\gamma \vee 1)^{1/(2-s)} b^{-1}(n)$. 
\end{proof}

The main approximation is given below.

\begin{thm} \label{T.Rozovskii}
Suppose Assumption \ref{A.StdCase} holds. Fix $\epsilon \in (0,1)$ and set 
\begin{align} 
\tilde \pi(y,n) &= \left( 1- \Phi(y/\sqrt{n}) \right) \Indicator(y \leq \sqrt{n}) + \left( 1- \Phi(y/\sqrt{n}) \right) e^{\frac{y^2}{2n} + nQ_{\tilde\kappa} \left( \frac{y}{n}\right)} \Indicator(y > \sqrt{n}), \label{eq:tildePi} \\
J(y,n) &= \sqrt{n} \left\{ \int_{y - \sqrt{n}}^\infty \overline{V}(t) \Phi'\left( \frac{y-t}{\sqrt{n}} \right) dt  + \frac{1}{\sqrt{2\pi}} \int_{\sqrt{n} \vee(y-b^{-1}(2(1+\epsilon)n))}^{y-\sqrt{n}} \overline{V}(t) e^{nQ_{\tilde\kappa} \left( \frac{y-t}{n} \right) } dt \right\}, \label{eq:Jint} \\ 
C_n &= \min_{t \geq \sqrt{n}} t \left(\frac{1}{2} + \frac{\tilde Q(t)}{t^2} n \right). \label{eq:Roz_Cn}
\end{align}
Then, as $n \to \infty$, uniformly in $y$, 
\begin{align*}
P\left(\tilde{S}_n > y \right) &= \left( \tilde\pi(y,n) \Indicator(y \leq (1+\epsilon) C_n) + J(y,n) \Indicator(y \geq (1-\epsilon) C_n) \right) (1+o(1)).
\end{align*}
Moreover, there exist constants $0 < \gamma_1 \leq 1 \leq \gamma_2$ such that $C_n \in [ \gamma_1 b^{-1}(n), \gamma_2 b^{-1}(n)]$. 
\end{thm}

\begin{proof}
Choose $0 < \delta < 1-\tilde r$ and set $s = \tilde r + \delta$. Note that by Lemma \ref{L.g_properties} (a.) $\tilde Q(t)/ t^s$ is eventually decreasing. Also, since $\var(Y_1) = 1$, 
$$P\left( \tilde S_n \leq t \sqrt{n} \right) \to \Phi(t) $$
by the CLT.  Define $L(h)$ and $H(z)$ as in Lemma \ref{L.Roz_pi} and let $\pi(z,n)$ be given by \eqref{eq:Roz_pi}. 

Set $\eta(z) = b^{-1}(z^2)$ and note that by Lemma \ref{L.g_properties} (b.) $b^{-1}(t) \leq t^{1/(2-s)}$ for all $t$ sufficiently large, so $\eta(z) = o(z^2)$. Since $D(t) \to 1$ as $t \to \infty$, we have 
$$D(z^2/\eta(z)) = 1+o(1) = D(z), \qquad z \to \infty. $$
Let $\gamma = 1/(2(1+\epsilon))$, and define 
$$\omega_n = b^{-1}(n/\gamma), \quad \upsilon_n = s \sqrt{n} \tilde Q(\sqrt{n}) \frac{D(\sqrt{n}/\tilde Q(\sqrt{n}))}{ D(\sqrt{n})}.$$ 

Then, by Theorem 2 and Remark 1 from \cite{Roz_93}, 
\begin{align*}
P\left(\tilde S_n > y\right) &= \left( \pi(y,n) \Indicator(y \leq (1+\epsilon) C_n) + \sqrt{n} \left\{ \int_{y - \sqrt{n}}^\infty \overline{V}(t) \Phi'\left( \frac{y-t}{\sqrt{n}} \right) dt  \right. \right. \\
&\hspace{5mm} \left. \left. + \frac{1}{\sqrt{2\pi}} \int_{\sqrt{n} \vee(y-\lambda)}^{y-\sqrt{n}} \overline{V}(t) e^{H(y-t)} dt \right\} \Indicator(y \geq (1-\epsilon) C_n) \right) (1+o(1)) 
\end{align*}
as $n \to \infty$, uniformly for all $y$ and for any $\lambda \in [\omega_n, \upsilon_n]$. Also, by Lemma \ref{L.Roz_pi}, 
$$e^{H(z)} =  e^{ nQ_\kappa \left( \frac{z}{n} \right)} (1+o(1))$$
uniformly for $\sqrt{n} \leq z \leq c b^{-1}(n)$ for any $c > 0$. We will show below that $C_n \leq b^{-1}(2n) \leq 2^{1/(2-s)} b^{-1}(n)$ (by Lemma \ref{L.g_properties} (c.)), so we can replace $\pi(y,n)$ by $\tilde \pi(y,n)$. Also, by choosing $\lambda = b^{-1}(2(1+\epsilon)n)$ and noting that for $t \geq y -\lambda$ we have $y - t \leq \lambda = b^{-1}(2(1+\epsilon)n) \leq 4^{1/(2-s)} b^{-1}(n)$ (by Lemma \ref{L.g_properties} (c.)), we can replace $e^{H(y-t)}$ with $e^{ nQ_{\tilde\kappa}\left( \frac{y-t}{n} \right)}$. This gives the statement of the theorem.

To verify the order of magnitude of $C_n$ let $h(t) = t\left( \frac{1}{2} + \frac{\tilde Q(t)}{t^2} n \right)$ and note that $h$ is continuous and a.s. differentiable. Recall that by assumption $\tilde Q$ has Lebesgue density $\tilde q$, and  note that $b(t)$ is eventually increasing, since by Lemma \ref{L.g_properties} (a.) $\tilde Q(t)/t^s$ is eventually decreasing. Then, for all $t_0 \leq t \leq b^{-1}(2 n(1-s))$,
$$h'(t) = \frac{1}{2} - n \cdot \frac{\tilde Q(t) - t \tilde q(t)}{t^2} \leq  \frac{1}{2} - n (1-s)  \cdot \frac{1}{b(t)} \leq 0.$$
For $t \geq b^{-1}(2 n)$ note that $\liminf_{t \to \infty} t \tilde q(t) \geq \liminf_{t \to \infty} \sigma t q(\sigma t+\mu) - 2 > a(r) - 2 \geq 0$.  It follows that
$$h'(t) = \frac{1}{2} - n \cdot \frac{\tilde Q(t) - t \tilde q(t)}{t^2} \geq \frac{1}{2} -  n \cdot \frac{1}{b(t)} \geq 0.$$
We conclude that $C_n \in [ b^{-1}(2(1-s) n), b^{-1}(2 n)]$, and by by Lemma \ref{L.g_properties} (c.) and (d.), 
$$b^{-1}(2(1-s)n) \geq (2(1-s) \wedge 1) b^{-1}(n) \quad \text{and} \quad b^{-1}(2n) \leq (2 \vee 1)^{1/(2-s)} b^{-1}(n)  .$$ 
\end{proof}

We now give a lemma stating that under Assumption \ref{A.Hazard}, the random variable $Y_1 = (X_1-\mu)/\sigma$ satisfies Assumption \ref{A.StdCase}. Throughout the rest of the paper,
\begin{equation} \label{eq:tildeQ_Def}
\tilde Q(t) = Q(\sigma t+\mu) - 2\log t,
\end{equation}
and the functions $b$ and $b^{-1}$, as well as the constant $\tilde r$, are defined according to this function.

\begin{lem} \label{eq:StdToNonneg}
Suppose $Q$ satisfies Assumption \ref{A.Hazard}, then $Y_1 = (X_1-\mu)/\sigma$ satisfies Assumption \ref{A.StdCase}. 
\end{lem}

\begin{proof}
Let $\tilde Q(t) = Q(\sigma t+\mu) - 2\log t$, then $\overline{V}(t) = P(Y_1 > t) = e^{-\tilde Q(t)} / t^2$, and since $D(t) = \int_{|u|<t} t^2 dV(t) \to 1$ as $t \to \infty$, then $\overline{V}(t) \asymp D(t) t^{-2} e^{-\tilde Q(t)}$. Also, since $Q$ has Lebesgue density $q$, then $\tilde Q$ has Lebesgue density $\tilde q(t) = \sigma q(\sigma t+\mu) -2/t$. It follows that
\begin{align*}
\tilde r &= \limsup_{t \to \infty} \frac{t \tilde q(t)}{\tilde Q(t)} = \limsup_{t \to \infty} \frac{t \sigma q(\sigma t + \mu) - 2}{Q(\sigma t+\mu) - 2\log t} \leq \limsup_{z \to \infty}  \frac{zq(z)}{Q(z)-2\log z} \\
&\leq r \limsup_{z \to \infty} \frac{Q(z)}{Q(z)-2\log z}.
\end{align*}
By \eqref{eq:LowerBoundQ}, there exists $\beta > a(r) \geq 2$ such that $Q(t) \geq \beta \log t$ for all sufficiently large $t$. It follows that
$$r \limsup_{z \to \infty} \frac{r}{1 - 2(\log z)/Q(z)} < \frac{r}{1- 2/\beta},$$
where if $r > 0$ we have $r/(1-2/\beta) < r/(1-2/a(r)) = 2r/(1+r) < 1$. Therefore, $r \leq \tilde r < 1$ and 
$$\liminf_{t \to \infty} \frac{\tilde Q(t)}{\log t} = \liminf_{t \to \infty} \frac{Q(\sigma t+\mu)}{\log t} - 2 \geq \beta-2.$$
Clearly, if $r = 0$ then $\tilde r = 0$ and $\beta -2 > 0 = \tilde r/(1-\tilde r)$. If $r = 0$ we already showed that $\tilde r < \beta r/(\beta-2)$, which combined with $\beta > a(r) = 4/(1-r)$ gives $\tilde r < 1 - 2/(\beta-2)$, which in turn implies that $\beta-2 > 2/(1-\tilde r) > \tilde r/(1-\tilde r)$.

We also note that for any $l \in \{0, 1, 2, \dots \}$
$$\limsup_{t \to \infty} \frac{\tilde Q(t)}{t^{l/(l+1)}} = \sigma^{l/(l+1)} \limsup_{u\to \infty} \frac{Q(u) - 2\log u + 2\log \sigma}{u^{l/(l+1)}}.$$
Since for any $l \in \{1, 2, 3, \dots\}$ we have $\limsup_{t\to\infty} \tilde Q(t)/t^{l/(l+1)} = \sigma^{l/(l+1)} \limsup_{u \to \infty} Q(u)/ u^{l/(l+1)}$, it follows that $\tilde \kappa = \kappa$. Finally, from the discussion following the definition of $\kappa$, equation \eqref{eq:kappa}, we have that $E[X_1^{\kappa + s}] < \infty$ for any $2 \leq s\leq (2+r)/(1-r)$, which implies $E[|Y_1|^{\tilde\kappa+1}] < \infty$. 
\end{proof}

\bigskip

We are now ready to give a uniform approximation for $P(S_n > x)$ that will work over the region $n \geq K_r(x)$.  We choose not to use this approximation in the heavy tail region $1 \leq n \leq K_r(x)$ to avoid having to show that it is equivalent to the heavy tail asymptotic $n \overline{F}(x-n\mu)$. Instead, we use a result  from \cite{Bal_Dal_Klupp_04} that will give us without much additional work the heavy tail asymptotic directly.  

We point out that we will not apply Theorem \ref{T.Rozovskii} to the positive mean exactly the way it is stated, but instead we use a slight modification that will work better when applied to the queue. In particular, we will substitute the function $\tilde \pi(y,n)$ given by \eqref{eq:tildePi}, where $y = (x-n\mu)/\sigma$, with the following
\begin{equation} \label{eq:hatPi_Def}
\hat \pi_\kappa(x,n) = \begin{cases}  
\Phi(-y/\sqrt{x/\mu}), & \kappa = 2, \\
\Phi(-y/\sqrt{x/\mu}) \Indicator(n > N(x)) + \frac{\sqrt{x}}{y \sqrt{2\pi\mu}} e^{n Q_\kappa \left( \frac{y}{n} \right)} \Indicator(n \leq N(x)), & \kappa > 2. \end{cases}
\end{equation}
The function $J(y,n)$ given in \eqref{eq:Jint} does not need to be modified since its contribution will be shown to be negligible in the queue. 

\begin{lem} \label{L.UglyTail}
Suppose $Q$ satisfies Assumption \ref{A.Hazard}. Let $y = (x-n\mu)/\sigma$, fix $\epsilon \in (0,1)$ and define
$$B_\kappa(x,n) = \hat \pi_\kappa(x,n) \Indicator(y \leq (1+\epsilon) C_n) + J(y,n) \Indicator(y \geq (1-\epsilon) C_n),$$
where $\hat\pi_\kappa(x,n)$, $J(y,n)$ and $C_n$ are given by \eqref{eq:hatPi_Def}, \eqref{eq:Jint} and \eqref{eq:Roz_Cn}, respectively. Then,
$$\lim_{x \to \infty} \sup_{n \geq K_r(x)} \left| \frac{P(S_n > x)}{B_\kappa(x,n)} - 1 \right| = 0.$$
Moreover, there exist constants $0 < \gamma_1\leq 1 \leq \gamma_2$ such that $C_n \in [\gamma_1 b^{-1}(\mu n), \, \gamma_2 b^{-1}(\mu n)]$. 
\end{lem}

\begin{proof}
By Theorem \ref{T.Rozovskii} and Lemma \ref{eq:StdToNonneg}, we have that
$$P(S_n > x) = \left( \tilde \pi(y,n) \Indicator(y \leq (1+\epsilon) C_n) + J(y,n) \Indicator(y \geq (1-\epsilon)C_n) \right)(1+o(1))$$
as $x \to \infty$ for all $n \geq K_r(x)$, where $\tilde \pi(y,n)$ is given in \eqref{eq:tildePi}. Furthermore, by the same theorem and Lemma~\ref{L.g_properties} (c.) and (d.), there exist constants $0 < \gamma_1 \leq 1 \leq \gamma_2$ such that $C_n \in [\gamma_1 b^{-1}(\mu n), \gamma_2 b^{-1}(\mu n)]$. It can be verified that
$$\{y \leq (1+\epsilon) C_n\} \subset \{y \leq 2\gamma_2 b^{-1}(\mu n) \}  \subset \{ x- 2\sigma \gamma_2 b^{-1}(x) \leq n\mu \} = \{ n > l(x) \}$$
for sufficiently large $x$, where $l(x) = (x - 2\sigma \gamma_2 b^{-1}(x))/\mu$, so all that remains to show is that $\tilde \pi(y,n) = \hat \pi_\kappa(x,n)(1+o(1))$ as $x \to \infty$ for all $n > l(x)$. 

Note that after some algebra we can obtain the equivalence
$$\{ y \leq \sqrt{n} \} = \{ n \geq m(x) \}, \quad \text{where} \quad m(x) = \frac{x}{\mu} + \frac{\sigma^2}{2\mu^2} - \frac{\sigma \sqrt{x}}{\mu^{3/2}} \sqrt{ 1 + \frac{\sigma^2}{4\mu x}}.$$
Since $N(x) = \lfloor (x-\sqrt{x\log x})/\mu \rfloor < m(x)$ for sufficiently large $x$, it follows that for $\kappa > 2$,
\begin{align}
&\left| \hat \pi_\kappa(x,n) - \tilde \pi(y,n) \right| \Indicator(n > l(x)) \notag \\
&= \left| \Phi\left( -y/\sqrt{x/\mu} \right) \Indicator( n > N(x)) - \Phi\left( -y/\sqrt{n} \right) \Indicator( n > m(x)) \right. \notag \\
&\hspace{3mm} + \left. \frac{\sqrt{x}}{y \sqrt{2\pi \mu}} e^{n Q_\kappa \left( \frac{y}{n} \right)} \Indicator(l(x) < n \leq N(x)) -  \Phi(-y/\sqrt{n})  e^{\frac{y^2}{2n} + n Q_\kappa\left( \frac{y}{n} \right)} \Indicator(l(x) < n \leq m(x))  \right| \notag \\
&\leq \left|  \Phi\left( -y/\sqrt{x/\mu} \right) - \Phi\left( -y/\sqrt{n} \right)    \right| \Indicator(n > m(x)) \label{eq:large_n} \\
&\hspace{3mm} + \left| \Phi\left( -y/\sqrt{x/\mu} \right) -  \Phi(-y/\sqrt{n})  e^{\frac{y^2}{2n} + n Q_k\left( \frac{y}{n} \right)} \right| \Indicator(N(x) < n \leq m(x)) \label{eq:medium_n} \\
&\hspace{3mm} + \left| \frac{\sqrt{x}}{y \sqrt{2\pi \mu}} e^{n Q_\kappa \left( \frac{y}{n} \right)}  -  \Phi(-y/\sqrt{n})  e^{\frac{y^2}{2n} + n Q_\kappa \left( \frac{y}{n} \right)}   \right| \Indicator(l(x) < n \leq N(x)), \label{eq:small_n}
\end{align}
while for $\kappa = 2$ we have $y^2/(2n) + n Q_\kappa( y/n) = 0$ and 
\begin{align}
&\left| \hat \pi_\kappa(x,n) - \tilde \pi(y,n) \right| \Indicator(n > l(x)) = \left| \Phi\left( -y/\sqrt{x/\mu} \right)  - \Phi\left( -y/\sqrt{n} \right)  \right| \Indicator(n > l(x)). \label{eq:k_is_2}
\end{align}

To analyze \eqref{eq:large_n} and the corresponding segment of \eqref{eq:k_is_2}  define $s(x) = (x+\sqrt{x}\log x)/\mu$, then
\begin{align*}
&\left|  \Phi\left( -y/\sqrt{x/\mu} \right) - \Phi\left( -y/\sqrt{n} \right)    \right| \Indicator(n > m(x)) \\
&\leq \left|  \Phi\left( y/\sqrt{x/\mu} \right) - \Phi\left( y/\sqrt{n} \right)    \right| \Indicator(m(x) < n \leq s(x)) +  2 \Phi\left( y/\sqrt{n} \right) \Indicator(n > s(x)) \\
&\leq \Phi'(0) |y| \frac{|\sqrt{n} - \sqrt{x/\mu}|}{\sqrt{nx/\mu}}  \Indicator(m(x) < n \leq s(x)) +  2 \Phi\left( -(s(x)\mu-x)/\sqrt{\sigma^2 s(x)} \right) \Indicator(n > s(x)) \\
&\leq C \frac{(x-n\mu)^2}{x^{3/2}}  \Indicator(m(x) < n \leq s(x)) + 2 \Phi\left( - \frac{\sqrt{\mu}\log x }{\sigma (1 + o(1))} \right) \Indicator(n > s(x)) \\
&\leq C \min\left\{ \frac{(\log x)^2}{\sqrt{x}}, \,  \Phi\left( - \frac{\sqrt{\mu}\log x }{2\sigma} \right)   \right\} \Indicator(n > m(x)).
\end{align*}
Since for $n > m(x)$ we have $\Phi\left(-y/\sqrt{x/\mu}\right) \geq \Phi\left(-(x-\mu m(x))/\sqrt{\sigma^2 x/\mu}\right) \to \Phi(-1)$, it follows that \eqref{eq:large_n} and the corresponding segment of \eqref{eq:k_is_2} are bounded by
$$C \varphi_1(x) \Phi\left(-y/\sqrt{x/\mu} \right) \Indicator(n > m(x)),$$
where $\varphi_1(x) =  \min\left\{ (\log x)^2/\sqrt{x}, \,  \Phi\left( - \sqrt{\mu}\log x /(2\sigma) \right)   \right\}$. To bound \eqref{eq:medium_n} and the corresponding segment of \eqref{eq:k_is_2} we note that for $N(x) < n \leq m(x)$ we have $ n Q_\kappa(y/n) = -y^2/(2n) + O\left( y^3/n^2 \right)$ (recall that $nQ_\kappa(y/n) = -y^2/(2n)$ if $\kappa = 2$), so
\begin{align*}
&\left| \Phi\left( -y/\sqrt{x/\mu} \right) -  \Phi(-y/\sqrt{n})  e^{\frac{y^2}{2n} + n Q_\kappa\left( \frac{y}{n} \right)} \right| \\
&\leq \left| \Phi\left( -y/\sqrt{x/\mu} \right) -  \Phi(-y/\sqrt{n}) \right| + \left| 1 - e^{\frac{y^2}{2n} + n Q_\kappa\left( \frac{y}{n} \right)} \right|  \Phi(-y/\sqrt{n}) \\
&\leq \Phi'\left( y/\sqrt{x/\mu} \right) y \frac{(\sqrt{x/\mu} - \sqrt{n})}{\sqrt{nx/\mu}} + C \frac{y^3}{n^2} \Phi\left( -y/\sqrt{x/\mu}  \right) \\
&\leq C \frac{(y^2/(x/\mu) + 1)}{y/\sqrt{x/\mu}} \Phi\left(- y/\sqrt{x/\mu} \right) \frac{(x-n\mu)^2}{x^{3/2}} + C \frac{(x-n\mu)^3}{x^2} \Phi\left( -y/\sqrt{x/\mu}  \right) \\
&\leq C \left( \frac{(y^2 + x) (x-n\mu)}{x^{2}} + \frac{(x-n\mu)^3}{x^2}  \right) \Phi\left( -y/\sqrt{x/\mu}  \right) \\
&\leq C \frac{(\log x)^{3/2}}{\sqrt{x}} \Phi\left( -y/\sqrt{x/\mu}  \right),
\end{align*}
where for the third inequality we used the relation $\Phi(-z) \geq \Phi'(z) z/(z^2+1)$ for all $z > 0$. Therefore, \eqref{eq:medium_n} and the corresponding segment of \eqref{eq:k_is_2} are bounded by
$$C \varphi_2(x) \Phi\left( -y/\sqrt{x/\mu}  \right) \Indicator(N(x) < n \leq m(x)),$$
where $\varphi_2(x) = (\log x)^{3/2}/\sqrt{x}$.  To bound the last segment of \eqref{eq:k_is_2} note that the preceding calculation yields
\begin{align*}
&\left|  \Phi\left( -y/\sqrt{x/\mu} \right) - \Phi\left( -y/\sqrt{n} \right)    \right| \Indicator(l(x) < n \leq N(x)) \\
&\leq C \frac{(y^2+x)(x-n\mu)}{x^2} \Phi\left(- y/\sqrt{x/\mu} \right) \Indicator(l(x) < n \leq N(x)) \\
&\leq C \frac{ ((b^{-1}(x))^2 + x) b^{-1}(x)}{x^2}  \Phi\left(- y/\sqrt{x/\mu} \right) \Indicator(l(x) < n \leq N(x)).
\end{align*}
Since $\kappa = 2$ implies that $\tilde Q(t)/\sqrt{t} \to 0$, then 
$$\lim_{x \to \infty} \frac{(b^{-1}(x))^2}{x} = \lim_{x \to \infty} \frac{(b^{-1}(x))^2}{b(b^{-1}(x))} = \lim_{t \to \infty} \frac{t^2}{t^2/\tilde Q(t)} = \infty,$$
so 
$$C \frac{ ((b^{-1}(x))^2 + x) b^{-1}(x)}{x^2}  \leq C \frac{(b^{-1}(x))^3}{x^2} \triangleq C \varphi_3(x),$$
where
$$\lim_{x \to \infty} \varphi_3(x) =  \lim_{x \to \infty} \frac{(b^{-1}(x))^3}{(b(b^{-1}(x)))^2} =  \lim_{t \to \infty} \frac{t^3}{(t^2/\tilde Q(t))^2} = \lim_{t \to \infty} \frac{\tilde Q(t)^2}{t} = 0.$$
We have thus shown that when $\kappa = 2$, 
$$\left|  \Phi\left( -y/\sqrt{x/\mu} \right) - \Phi\left( -y/\sqrt{n} \right)    \right| \Indicator(n > l(x)) \leq C \max_{i \in \{1,2,3\}} \varphi_i(x) \Phi\left(- y/\sqrt{x/\mu} \right) \Indicator(n > l(x)).$$

Finally, to bound \eqref{eq:small_n} we use the inequalitvy $\Phi'(z)z/(z^2+1) \leq \Phi(-z)$ to obtain, for $l(x) < n \leq N(x)$, 
\begin{align*}
\left| \frac{\sqrt{x}}{y \sqrt{2\pi\mu}} e^{n Q_\kappa\left( \frac{y}{n} \right)}  -  \Phi(-y/\sqrt{n})  e^{\frac{y^2}{2n} + n Q_\kappa\left( \frac{y}{n} \right)}   \right|  &= \left( \frac{\sqrt{x}}{y \sqrt{\mu}}  - \frac{\Phi(-y/\sqrt{n})}{\Phi'(y/\sqrt{n})}    \right) \frac{1}{\sqrt{2\pi}} \, e^{n Q_\kappa\left( \frac{y}{n} \right)} \\
&\leq \left( \frac{\sqrt{x}}{y \sqrt{\mu}}  - \frac{y/\sqrt{n}}{y^2/n + 1}    \right) \frac{1}{\sqrt{2\pi}} \, e^{n Q_\kappa\left( \frac{y}{n} \right)} \\
&\leq \left( \sqrt{x/\mu}  - \sqrt{n} + \frac{n^{3/2}}{y^2 }    \right) \frac{1}{y\sqrt{2\pi}} \, e^{n Q_\kappa\left( \frac{y}{n} \right)} \\
&\leq C \left( \frac{x-n\mu}{x} + \frac{x}{(x-n\mu)^2} \right) \frac{\sqrt{x}}{y\sqrt{2\pi\mu}} \, e^{n Q_\kappa\left( \frac{y}{n} \right)} \\
&\leq C \left( \frac{b^{-1}(x)}{x} + \frac{1}{\log x} \right) \frac{\sqrt{x}}{y\sqrt{2\pi\mu}} \, e^{n Q_\kappa\left( \frac{y}{n} \right)}.
\end{align*}
It follows that \eqref{eq:small_n} is bounded by
$$C \varphi_4(x) \frac{\sqrt{x}}{y\sqrt{2\pi\mu}} \, e^{n Q_\kappa\left( \frac{y}{n} \right)} \Indicator(l(x) < n \leq N(x)),$$
where $\varphi_4(x) =  b^{-1}(x)/x + 1/\log x$. We conclude that 
$$\left| \hat \pi_\kappa(x,n) - \tilde \pi(x,n) \right|  \leq C \max_{i\in \{1,2,3,4\}} \varphi_i(x) \hat\pi_\kappa(x,n).$$
for all $n > l(x)$. This completes the proof. 
\end{proof}

\subsection{A first approximation for $P(W_\infty(\rho) > x)$} \label{SS.FirstApprox}

We will now give an approximation for $P(W_{\infty}(\rho) > x)$, that although too complicated to be used in practice, will serve as an intermediate step towards obtaining the more explicit approximations given in Theorems \ref{T.SumApprox} and \ref{T.Main}.  

The idea of this section is to substitute $P(S_n > x)$ in the Pollaczek-Khintchine formula \eqref{eq:Poll-Khin} the heavy-tail approximation $n \overline{F}(x-n\mu)$ in the range $1 \leq n \leq K_r(x)$, and by $B_\kappa(x,n)$, as defined in Lemma \ref{L.UglyTail}, in the range $n > K_r(x)$. 

The intermediate approximation for $P(W_\infty(\rho) > x)$ is given by
\begin{align}
S_\kappa(\rho,x) &= \sum_{n=1}^{K_r(x)} (1-\rho)\rho^n n \overline{F}(x-n\mu) + \sum_{n=K_r(x)+1}^\infty (1-\rho) \rho^n \hat \pi_k(x,n) \Indicator (y \leq (1+\epsilon) C_n)  \notag \\
&\hspace{3mm} + \sum_{n=K_r(x)+1}^\infty (1-\rho) \rho^n  J(y,n) \Indicator(y \geq (1-\epsilon) C_n), \label{eq:S_rho_x}
\end{align}
where $y = (x-\mu n)/\sigma$, and $\hat\pi_\kappa(x,n)$, $J(y,n)$ and $C_n$ are given by \eqref{eq:hatPi_Def}, \eqref{eq:Jint} and \eqref{eq:Roz_Cn}, respectively. The last term in \eqref{eq:S_rho_x} corresponds to the so-called ``intermediate domain", where as mentioned in Section \ref{S.MainResults}, the asymptotic behavior of $P(S_n > x)$ is rather complicated. Under additional (differentiability) assumptions on $Q$, more explicit asymptotics for $J(y,n)$ have been derived in \cite{Roz_93} (see also \cite{Bor00b} for other results applicable to this region). We point out that $S_\kappa(\rho,x)$ is ``very close" to being the  approximation in Theorem \ref{T.SumApprox} if we replace $\Indicator(y \leq (1+\epsilon) C_n)$ with $\Indicator(n \geq M(x))$ and ignore the entire third term of $S_\kappa(\rho,x)$, to see this sum the tail of the second term of $S_\kappa(\rho,x)$ to write it as the expectation of a function of a normal random variable.

We will now show the asymptotic equivalence of $P(W_\infty(\rho) > x)$ and $S_\kappa(\rho,x)$.


\begin{lem} \label{L.HeavyTail_region}
Suppose $Q$ satisfies Assumption \ref{A.Hazard}, then
$$\lim_{x \to \infty} \sup_{1 \leq n \leq K_r(x)}  \left| \frac{P(S_n > x)}{n \overline{F}(x-n\mu)} -1  \right| = 0.$$
\end{lem}

\begin{proof}
Recall that $\omega_2(x) = x^2/(Q(x) \vee 1)^2$ and $\omega_2^{-1}(x) = \inf\{ u \geq 0: x \leq \omega_2(u) \}$. By Lemma \ref{L.RightInverse}, $\omega_2^{-1}$ is non decreasing, $\omega_2(\omega_2^{-1}(x)) = x$ and $\omega_2^{-1}(\omega_2(x)) \leq x$. Let $t_n = \omega_2^{-1}((\mu \wedge 1)n)/2$, and note that
\begin{align*}
\limsup_{n \to \infty} \sqrt{n} \, \frac{Q(t_n)}{t_n} &\leq \limsup_{n \to \infty} \sqrt{n} \, \frac{Q(\omega_2^{-1}((\mu\wedge1)n)/2)}{ \omega_2^{-1}((\mu\wedge1)n)/2} \leq \frac{2}{(\mu\wedge1)^{1/2}}  \limsup_{s \to \infty} \sqrt{s} \, \frac{Q(\omega_2^{-1}(s))}{ \omega_2^{-1}(s)} \\
&= \frac{2}{(\mu\wedge1)^{1/2}} \limsup_{n \to \infty} \frac{\sqrt{s}}{\{\omega_2(\omega_2^{-1}(s)) \}^{1/2}} = \frac{2}{(\mu\wedge1)^{1/2}}. 
\end{align*}
Then by Theorem 3.1 in \cite{Bal_Klupp_04}, 
$$\lim_{n \to \infty} \sup_{t \geq t_n} \left| \frac{P(S_n - \mu n > t)}{n \overline{F}(t)} - 1  \right| = 0.$$
Next, we will show that for $n \leq K_r(x)$ we have $x-\mu n \geq t_n$. 

First, when $0 \leq r < 1/2$ we have $K_r(x) = \lfloor (x-\omega_2^{-1}(x))/\mu \rfloor$, so $n \leq K_r(x)$ implies 
$$x-\mu n \geq x-\mu K_r(x) \geq \omega_2^{-1}(x) \geq \omega_2^{-1}(\mu K_r(x)) \geq t_{K_r(x)} \geq t_n.$$
Similarly, when $1/2 \leq r < 1$ and $K_r(x) = \lfloor \min\{ \omega_2(x), x/(2\mu) \} \rfloor$, we have that $n \leq K_r(x)$ implies 
\begin{align*}
x-\mu n \geq x-\mu K_r(x) &\geq \max\{ x- \mu\omega_2(x), x/2\} \geq \omega_2^{-1}(\omega_2(x)) /2 \\
&\geq \omega_2^{-1}(K_r(x))/2 \geq   t_{K_r(x)} \geq t_n.
\end{align*}

These observations, combined with the fact that the subexponentiality of $F$ implies that $P(S_n > x) = n \overline{F}(x) (1+o(1))$ as $x \to \infty$ uniformly for $1 \leq n \leq a(x)$ for some $a(x) \to \infty$ completes the proof.  
\end{proof}

Combining Lemmas \ref{L.HeavyTail_region} and \ref{L.UglyTail} gives the following result. 

\begin{prop} \label{P.UglyApprox}
Define $S_\kappa(\rho,x)$ according to \eqref{eq:S_rho_x} and suppose $Q$ satisfies Assumption~\ref{A.Hazard}, then,
$$\lim_{x \to \infty} \sup_{0 < \rho < 1} \left| \frac{P(W_\infty(\rho) > x)}{S_\kappa(\rho,x)} - 1 \right| = 0.$$
\end{prop}

\bigskip

This first approximation for $P(W_\infty(\rho) > x)$ might not very useful in practice since it involves two integrals, those in the definition of $J(y,n)$, that are not in general closed-form, and two indicator functions that depend on the quantity $C_n$ (the solution to a certain optimization problem). The approximation given in Theorem~\ref{T.SumApprox} is more explicit, and thus more suitable for computations, both numerical and analytical.

\section{Proof of Theorem \ref{T.SumApprox}} \label{S.SumApproxProof}

The proof of Theorem \ref{T.SumApprox} is rather technical, so we divide into several lemmas, the first of which gives some more properties of the functions $b^{-1}$ and $\omega_1^{-1}$.

\begin{lem} \label{L.b_and_w}
Suppose $Q$ satisfies Assumption \ref{A.Hazard}. Let $\tilde Q$ and $b^{-1}$ be defined according to \eqref{eq:tildeQ_Def} and \eqref{eq:b_Def}, respectively. Then, 
\begin{enumerate}  
\item $\lim_{t\to\infty} Q(t)/ \omega_1^{-1}(t) = \lim_{t\to \infty} Q(t)/b^{-1}(t) = 0$, 
\item $\lim_{t\to \infty} \sqrt{t}/\omega_1^{-1}(t) = 0$. 
\end{enumerate}
\end{lem}

\begin{proof}
To show the first limit in (a.) use Proposition 3.7 in \cite{Bal_Dal_Klupp_04} with some $r < s < 1$ as follows, 
\begin{align*}
\lim_{t \to \infty} \frac{Q(t)}{\omega_1^{-1}(t)} &= \lim_{t \to \infty} \frac{Q(\omega_1(\omega_1^{-1}(t)))}{\omega_1^{-1}(t)} = \lim_{u \to \infty} \frac{Q(\frac{u^2}{Q(u)})}{u} \\
&\leq \lim_{u \to \infty} \frac{ (u/Q(u))^{s} Q(u)}{u} = \lim_{u\to \infty} \left( \frac{Q(u)}{u} \right)^{1-s} = 0.
\end{align*}
For the second limit we first note that the same arguments used above give $\lim_{t \to \infty} \tilde Q(t)/b^{-1}(t)  = 0$, so all we need to show is that $\limsup_{t \to \infty} Q(t)/\tilde Q(t) < \infty$. That this is the case follows from   
$$0 \leq \limsup_{t \to \infty} \frac{Q(t)}{\tilde Q(t)} \leq  \limsup_{t\to \infty} \frac{Q(t)}{Q(\sigma t)-2\log t} \leq \limsup_{u \to \infty} \frac{(\sigma^{-1} \vee 1)^s Q(u)}{Q(u) - 2\log u + 2\log\sigma} = \limsup_{u \to \infty} \frac{C}{1-2\log u/Q(u)}$$
and \eqref{eq:LowerBoundQ}, which gives $Q(u)/\log u \geq \beta > a(r) \geq 2$ for large $u$. 

For part (b.)
\begin{align*}
\lim_{t \to \infty} \frac{\sqrt{t}}{\omega_1^{-1}(t)} &= \lim_{t \to \infty} \frac{\sqrt{\omega_1(\omega_1^{-1}(t))}}{\omega_1^{-1}(t)} = \lim_{u \to \infty} \frac{\sqrt{u^2/Q(u)}}{u} = \lim_{u \to \infty} \frac{1}{\sqrt{Q(u)}} = 0. 
\end{align*}
\end{proof}

Next define $Z_\kappa(\rho,x)$ according to \eqref{eq:Z_Def_2} and \eqref{eq:Z_Def_big2}, and $S_\kappa(\rho,x)$ according to \eqref{eq:S_rho_x}. Let
\begin{align*}
E_1(\rho,x) &= \left| \sum_{n=K_r(x)+1}^{\infty} (1-\rho) \rho^n  \hat\pi_\kappa(x,n) \left\{ \Indicator(y \leq (1+\epsilon)C_n) -  \Indicator( n > M(x))  \right\} \right|, \\
E_2(\rho,x) &= \begin{cases}
\left| \sum_{n=M(x)+1}^{\infty} (1-\rho) \rho^n \hat\pi_\kappa(x,n) - E\left[ \rho^{a(x,Z)} \Indicator\left(\sigma Z \leq \sqrt{\mu} \omega_1^{-1}(x)/\sqrt{x} \right)  \right] \right|, & \kappa = 2, \\
\left| \sum_{n=N(x)+1}^{\infty} (1-\rho) \rho^n \hat\pi_\kappa(x,n) - E\left[ \rho^{a(x,Z)} \Indicator\left(\sigma Z \leq \sqrt{\mu \log x} \right)   \right] \right|, & \kappa > 2, 
\end{cases}
\\
E_3(\rho,x) &= \sum_{n=K_r(x)+1}^{\infty} (1-\rho) \rho^n  J(y,n) \Indicator(y > (1-\epsilon) C_n).
\end{align*}
Then,
$$|S_\kappa(\rho,x) - Z_\kappa(\rho,x)| \leq E_1(\rho,x) + E_2(\rho,x) + E_3(\rho,x).$$
We will split the proof of Theorem \ref{T.SumApprox} into three propositions, each of them showing that $E_i(\rho,x) = o(Z_\kappa(\rho,x))$ as $x \to \infty$ uniformly for $0 < \rho < 1$, and some auxiliary lemmas. We start by giving a result that provides lower bounds for $Z_\kappa(\rho,x)$.

\bigskip

\begin{lem} \label{L.LowerBound}
Fix $c > 0$ and let $\hat \rho(x) = e^{-c\mu Q(x)/x}$. Then, for any $0 < \rho \leq \hat\rho(x)$, 
$$Z_\kappa(\rho,x) \geq \frac{C \rho}{1-\rho} \overline{F}(x),$$
while for $\hat\rho(x) \leq \rho < 1$,
$$Z_\kappa(\rho,x) \geq C e^{-\frac{x}{\mu} \Lambda_\rho(u(\rho))},$$
where $\Lambda_\rho(u(\rho))$ was defined in Lemma \ref{L.ustar}.
\end{lem}

\begin{proof}
Let $J(x) = \lfloor x/\sqrt{Q(x)}  \rfloor \leq K_r(x)$ and note that
\begin{align*}
Z_\kappa(\rho,x) &\geq \sum_{n=1}^{J(x)} (1-\rho) \rho^n n \overline{F}(x-n\mu)  \\
&\geq (1-\rho) \overline{F}(x) \sum_{n=1}^{J(x)} n \rho^n \\
&= \overline{F}(x) \frac{\rho}{1-\rho} \left(  1 - \rho^{J(x)} - (1-\rho) J(x) \rho^{J(x)}  \right).
\end{align*}
The first statement follows from the observation that for $0 < \rho \leq \hat\rho(x)$ we have
$$\rho^{J(x)} \leq e^{- \frac{c\mu Q(x)}{x} J(x)} = e^{- c\mu \sqrt{Q(x)} + o(1)} \to 0.$$

For the second statement consider first the case $\kappa = 2$, for which $e^{\frac{x}{\mu} \Lambda_\rho(u(\rho))} = e^{\frac{x}{\mu} \log\rho + \frac{\sigma^2 x(\log\rho)^2}{2 \mu^{3}}}$ and
\begin{align*}
Z_\kappa(\rho,x) &\geq E\left[ \rho^{a(x,Z)} \Indicator(\sigma Z \leq \sqrt{\mu}\omega_1^{-1}(x)/\sqrt{x} ) \right] \\
&=  e^{\frac{x}{\mu} \log\rho + \frac{\sigma^2 x(\log\rho)^2}{2 \mu^{3}}} \Phi\left( \frac{\sqrt{\mu}\omega_1^{-1}(x)}{\sigma\sqrt{x}} + \frac{\sigma  \sqrt{x}}{\mu^{3/2}} \log\rho \right) \\
&\geq e^{\frac{x}{\mu} \Lambda_\rho(u(\rho))} \Phi\left( \frac{\sqrt{\mu}\omega_1^{-1}(x)}{\sigma\sqrt{x}} - \frac{\sigma cQ(x)}{\sqrt{\mu x}}  \right) = e^{\frac{x}{\mu} \Lambda_\rho(u(\rho))} (1 +o(1))
\end{align*}
as $x \to \infty$, for all $\hat\rho(x) \leq \rho < 1$ (since $Q(x)/\omega_1^{-1}(x) \to 0$ by Lemma \ref{L.b_and_w} (a.)). For $\kappa > 2$ we split the interval $[\hat\rho(x), 1)$ into two parts as follows. Define $\tilde \rho(x) = e^{- \frac{\mu^2 \sqrt{\log x}}{\sigma^2 \sqrt{x}}  }$. Then, for $\tilde \rho(x) \vee \hat\rho(x) \leq \rho < 1$,
$$Z_\kappa(\rho,x) \geq E\left[ \rho^{a(x,Z)} \Indicator\left(Z \leq \sqrt{\mu \log x}/\sigma \right) \right] \geq \Phi(0) e^{\frac{x}{\mu} \log\rho + \frac{\sigma^2}{2\mu^3} x(\log\rho)^2} \geq C e^{\frac{x}{\mu} \Lambda_\rho(u(\rho))}.$$
For the interval $[\hat\rho(x), \tilde \rho(x) \vee \hat\rho(x))$ (assuming $\tilde\rho(x) > \hat\rho(x)$), let $u_n = (x - n\mu)/x$ and use Lemma \ref{L.Lambda} to obtain
$$Z_\kappa(\rho,x) \geq \frac{C(1-\rho) }{\sqrt{\mu x}} \sum_{n=M(x)+1}^{N(x)} \frac{e^{\frac{x}{\mu} \Lambda_\rho(u_n)}}{u_n} \geq C (1-\rho) \sqrt{x} \int_{u_{N(x)-1}}^{u_{M(x)}} \frac{e^{\frac{x}{\mu} \Lambda_\rho(u)}}{u} du.$$
By Lemma \ref{L.ustar}, $\Lambda_\rho$ is concave on $[0, u_{M(x)}]$, and its maximizer, $u(\rho)$, satisfies
$$u(\rho) = -\frac{\sigma^2}{\mu^2}\log\rho + O(|\log\rho|^2) , \qquad \Lambda_\rho(u(\rho)) = \log\rho + \frac{\sigma^2}{2\mu^2} (\log\rho)^2 + O(|\log\rho|^3).$$
Also, the derivatives of $\Lambda_\rho$ satisfy
$$\Lambda_\rho'(t) = -\log\rho - \frac{\mu^2}{\sigma^2} t + O(t^2) \quad \text{and} \quad  \Lambda_\rho'' (t) = - \frac{\mu^2}{\sigma^2} + O(t).$$
Then, for some $\xi_t$ between $t$ and $u(\rho)$ and some constant $\zeta > \mu/\sigma^2$,
$$\Lambda_\rho(t) = \Lambda_\rho(u(\rho)) + \frac{\Lambda_\rho''(\xi_t)}{2} (t-u(\rho))^2 \geq \Lambda_\rho(u(\rho)) - \frac{\zeta\mu}{2}(t-u(\rho))^2.$$
Note that for $\rho \leq \tilde \rho(x)$ we have $u(\rho) \geq \sqrt{\log x/x} + O(\log x/x)$. Therefore, for any $0 < \delta < 1$ and $x$ sufficiently large,
\begin{align*}
&(1-\rho) \sqrt{x} \int_{u_{N(x)-1}}^{u_{M(x)}} \frac{e^{\frac{x}{\mu} \Lambda_\rho(u)}}{u} \,  du \\
&\geq  (1-\rho) \sqrt{x} e^{\frac{x}{\mu} \Lambda_\rho(u(\rho))} \int_{\sqrt{\zeta x}(u_{N(x)-1}-u(\rho))}^{\sqrt{\zeta x} \delta u(\rho)} \frac{e^{- z^2/2}}{z + \sqrt{\zeta x} u(\rho)} dz  \\
&\geq C e^{\frac{x}{\mu} \Lambda_\rho(u(\rho))} \left( \Phi\left(\sqrt{\zeta x} \delta u(\rho)) \right) - \Phi\left(-\sqrt{\zeta x}(u(\rho)-u_{N(x)-1} ) \right)    \right) \\
&\geq C e^{\frac{x}{\mu} \Lambda_\rho(u(\rho))} \qquad \left( \text{since $\sqrt{x}u(\rho) \to \infty$ for $\rho \leq \tilde\rho(x)$} \right).
\end{align*}
\end{proof}

The next lemma will be useful in showing the uniformity in $0 < \rho < 1$ of our bounds.

\begin{lem} \label{L.UniformBound}
Let $\alpha(n,x)$ be any function that does not depend on $\rho$. Then, for any $l(x) \geq 4\mu^{-1} x/Q(x)$ and $m(x) \leq \frac{x}{\mu}$, we have
$$\sup_{0 < \rho < 1} \frac{1}{Z_\kappa(\rho,x)} \sum_{n = l(x)}^{m(x)} (1-\rho) \rho^n \alpha(n,x) \leq \frac{C Q(x)}{x }   \sum_{n=l(x)}^{m(x)} e^{Q(x) -  \frac{\mu n Q(x)}{x} }  \alpha(n,x) ,$$
for sufficiently large $x$. 
\end{lem}

\begin{proof}
Define $\hat\rho(x) = e^{-\mu Q(x)/x}$.  By Lemma \ref{L.LowerBound}, we have that for $0 < \rho \leq \hat\rho(x)$, 
$$\sup_{0 < \rho \leq \hat\rho(x)} \frac{1}{Z_\kappa(\rho,x)} \sum_{n=l(x)}^{m(x)} (1-\rho) \rho^n \alpha(n,x)  \leq  \sup_{0 < \rho \leq \hat \rho(x)} C  \sum_{n=l(x)}^{m(x)} (1-\rho)^2 \rho^{n-1} \frac{\alpha(n,x)}{\overline{F}(x)} .$$
Define $h_n(\rho) = (1-\rho)^2 \rho^{n-1}$ and compute $h_n'(\rho) = (1-\rho)\rho^{n-2}(n(1-\rho) - 1-\rho)$. Note that for $\rho \in (0, \hat\rho(x)]$ we have
$$n(1-\rho)-1-\rho \geq n(1-\hat\rho(x)) -1 - \hat\rho(x),$$
so $h_n'(\rho) \geq 0$ on $(0, \hat\rho(x)]$ for all $n \geq (1+\hat\rho(x))/(1-\hat\rho(x))$ (note that $(1+\hat\rho(x))/(1-\hat\rho(x)) \sim 2\mu^{-1} x/Q(x)$ as $x \to \infty$). Therefore,
\begin{align*}
\sup_{0 < \rho \leq \hat\rho(x)}  \sum_{n=l(x)}^{m(x)} (1-\rho)^2 \rho^{n-1} \frac{\alpha(n,x)}{\overline{F}(x)} &=  \sum_{n=l(x)}^{m(x)} (1-\hat\rho(x))^2 \hat\rho(x)^{n-1} e^{Q(x)} \alpha(n,x) \\
&\leq \frac{C Q(x)^2}{x^2} \sum_{n=l(x)}^{m(x)} e^{Q(x) - \frac{\mu n Q(x)}{x}} \alpha(n,x).
\end{align*}
For the range $\hat\rho(x) \leq \rho < 1$ fix $\epsilon \in (0,1)$ and use Lemma \ref{L.LowerBound} again to obtain
\begin{align*}
\sup_{\hat\rho(x) \leq \rho < 1} \frac{1}{Z_\kappa(\rho,x)} \sum_{n=l(x)}^{m(x)} (1-\rho) \rho^n \alpha(n,x) &\leq \sup_{\hat\rho(x) \leq \rho < 1} \frac{C}{e^{\frac{x}{\mu} \log \rho + (1-\epsilon) \frac{\sigma^2}{2\mu^3} x (\log\rho)^2}} \sum_{n=l(x)}^{m(x)} (1-\rho) \rho^n \alpha(n,x) \\
&\leq \frac{C Q(x)}{x}  \sum_{n=l(x)}^{m(x)} \sup_{\hat\rho(x) \leq \rho < 1} e^{-\frac{x}{\mu} \log \rho + n\log\rho}  \alpha(n,x) \\
&= \frac{C Q(x)}{x}  \sum_{n=l(x)}^{m(x)}  e^{-\left(\frac{x}{\mu} - n\right) \log \hat\rho(x) }  \alpha(n,x) \qquad \left(\text{for all } n \leq x/\mu \right) \\
&= \frac{C Q(x)}{x}  \sum_{n=l(x)}^{m(x)}  e^{\left(\frac{x}{\mu} - n\right) \frac{\mu Q(x)}{x}  }  \alpha(n,x).
\end{align*}
\end{proof}


\begin{prop} \label{P.E1}
Under the assumptions of Theorem \ref{T.SumApprox}, 
$$\lim_{x \to \infty} \sup_{0 < \rho < 1} \frac{E_1(\rho,x)}{Z_\kappa(\rho,x)} = 0.$$
\end{prop}

\begin{proof}
Define $m_\epsilon(x) = \min\{ n \in \{1, 2, \dots\}: n\mu + (1+\epsilon) \sigma C_n \geq x \}$, and recall that $M(x) = \lfloor (x- \omega_1^{-1}(x))/\mu \rfloor$. Let $y = (x-n\mu)/\sigma$. Then,
$$E_1(\rho,x) = \sum_{n=K_r(x)+1}^\infty (1-\rho) \rho^n \hat\pi_\kappa(x,n) \Indicator(\min\{m_\epsilon(x), M(x)\} < n \leq \max\{ m_\epsilon(x), M(x)\}) .$$
Choose $0 < \delta < 1$. By Lemma \ref{L.UglyTail} there exist constants $0 < \gamma_1 \leq 1 \leq \gamma_2$ such that $C_n \in [\gamma_1 b^{-1}(\mu n), \gamma_2 b^{-1}(\mu n)]$.  Then, for any $n \geq l(x) \triangleq (x - \gamma_1 \sigma b^{-1}(x))/\mu$ and $x$ sufficiently large,
$$y \leq \gamma_1 b^{-1}(x) \leq  \left( \frac{x}{\mu n} \vee 1 \right) \gamma_1 b^{-1}(\mu n) < (1+\epsilon) C_n,$$
where in the second inequality we used Lemma \ref{L.g_properties} (c.). Similarly, for any $n \leq k(x) \triangleq (x - 2 \gamma_2 \sigma b^{-1}(x))/\mu$ and $x$ sufficiently large,
$$y \geq 2\gamma_2 b^{-1}(x) \geq 2\gamma_2 b^{-1}(\mu n) \geq (1+\epsilon) C_n.$$
It follows that $\lfloor k(x) \rfloor \leq m_\epsilon(x) \leq \lfloor l(x) \rfloor$ for sufficiently large $x$. Hence,  
$$E_1(\rho,x) \leq \sum_{n= \min\{\lfloor k(x)\rfloor, M(x)\} +1}^{\max\{\lfloor l(x) \rfloor, M(x)\}}  (1-\rho) \rho^n \hat \pi_\kappa(x,n),$$
and by Lemma \ref{L.UniformBound},
\begin{align*}
\sup_{0 < \rho < 1} \frac{E_1(\rho,x)}{Z_\kappa(\rho,x)} &\leq  \frac{CQ(x)}{x}  \sum_{n= \min\{\lfloor k(x) \rfloor, M(x)\} +1}^{\max\{\lfloor l(x) \rfloor, M(x)\}} e^{Q(x) - \frac{\mu n Q(x)}{x}} \hat \pi_k(x,n).
\end{align*}

By using the inequality $\Phi(-z) \leq \Phi'(z)/z$ for any $z > 0$, and observing that $n = (x/\mu) (1+o(1))$ for all $\min\{\lfloor k(x)\rfloor, M(x) \} < n \leq \max\{ \lfloor l(x) \rfloor, M(x) \}$, we obtain, for such $n$ and all sufficiently large $x$,
\begin{align*}
\hat \pi_\kappa(x,n) &\leq  \frac{\sqrt{x}}{y \sqrt{2\pi \mu}} e^{- \frac{\mu y^2}{2x}}  + \frac{\sqrt{x}}{y \sqrt{2\pi\mu}} e^{n Q_\kappa\left( \frac{y}{n} \right)} \leq C \frac{\sqrt{x}}{y\sqrt{\mu}} e^{- \frac{\mu y^2}{2 x} (1+o(1))} \leq C e^{- (1-\delta) \frac{\mu y^2}{2 x}}.
\end{align*}
It follows that 
\begin{align*}
\sup_{0 < \rho < 1} \frac{E_1(\rho,x)}{Z_\kappa(\rho,x)} &\leq  \frac{CQ(x)}{x}  \sum_{n= \min\{\lfloor k(x) \rfloor, M(x)\} +1}^{\max\{ \lfloor l(x) \rfloor, M(x)\}} e^{\frac{\sigma y Q(x)}{x} -(1-\delta) \frac{\mu y^2}{2x}} \\
&\leq \frac{CQ(x)}{x}  \sum_{n= \min\{\lfloor k(x) \rfloor, M(x)\} +1}^{\max\{\lfloor l(x) \rfloor, M(x)\}} e^{-(1-\delta) \frac{\mu y^2}{2x}\left(1 - \frac{2\sigma^2 Q(x)}{(1-\delta)\mu (x-\mu(l(x) \vee M(x)))}  \right)}.
\end{align*}
Note that by Lemma \ref{L.b_and_w} (a.), 
\begin{align} 
\lim_{x \to \infty} \frac{Q(x)}{x-\mu(\lfloor l(x) \rfloor \vee M(x))} &\leq C \lim_{x \to \infty} \frac{Q(x)}{b^{-1}(x) \wedge \omega_1^{-1}(x)} = 0, \label{eq:LimitQb}
\end{align}
which implies that for sufficiently large $x$,
\begin{align*}
\sup_{0 < \rho < 1} \frac{E_1(\rho,x)}{Z_\kappa(\rho,x)} &\leq \frac{CQ(x)}{x}  \sum_{n= \min\{\lfloor k(x) \rfloor, M(x)\} +1}^{\max\{\lfloor l(x) \rfloor, M(x)\}} e^{-(1-\delta)^2 \frac{\mu y^2}{2x}} \\
&\leq \frac{CQ(x)}{x}   \int_{x- \mu \max\{\lfloor l(x) \rfloor, M(x)\}}^{x - \mu \min\{ \lfloor k(x) \rfloor, M(x)\}} e^{-(1-\delta)^2 \frac{\mu u^2}{2\sigma^2 x}} du \\
&\leq \frac{CQ(x)}{\sqrt{x}} \int_{\frac{(1-\delta)\sqrt{\mu}}{\sigma \sqrt{x}} ( \gamma_1 \sigma b^{-1}(x) \wedge \omega_1^{-1}(x) ) }^\infty e^{-v^2/2} dv.
\end{align*}
Finally, by using the inequality $\Phi(-z) \leq \Phi'(z)/z$ for $z>0$ again, and \eqref{eq:LimitQb}, we obtain that
\begin{align*}
\lim_{x \to \infty} \sup_{0 < \rho < 1} \frac{E_1(\rho,x)}{Z_\kappa(\rho,x)} &\leq \lim_{x \to \infty} \frac{CQ(x) e^{-\frac{(1-\delta)^2\mu}{2\sigma^2 x}  ( \gamma_1 \sigma b^{-1}(x) \wedge \omega_1^{-1}(x) )^2}}{ \gamma_1 \sigma b^{-1}(x) \wedge \omega_1^{-1}(x) }  = 0.
\end{align*}
\end{proof}

\bigskip

\begin{lem} \label{L.E2bound}
Let $h_\kappa(x) = \omega_1^{-1}(x)$ if $\kappa = 2$ and  $h_\kappa(x) = \sqrt{x\log x}$, if $\kappa > 2$, then
\begin{align*}
E_2(\rho,x) &\leq  \frac{C\sqrt{x}}{h_\kappa(x)} \rho^{(x-h_\kappa(x))/\mu} e^{-\frac{\mu (h_\kappa(x))^2}{2\sigma^2x }} + \left( 1-\rho  \right) E\left[ \rho^{a(x,Z)} \Indicator\left( \sigma Z \leq \sqrt{\mu} h_\kappa(x)/\sqrt{x}\right) \right] . 
\end{align*}
\end{lem}

\begin{proof}
Recall that $a(x,z) = \mu^{-1}\left(x-\sigma z\sqrt{x/\mu}\right)$ and $Z\sim$ N(0,1). Define $L_\kappa(x) = \lfloor (x - h_\kappa(x))/\mu \rfloor$. Note that exact computation gives, 
\begin{align*}
&\sum_{n=L_\kappa(x)+1}^\infty (1-\rho) \rho^n \Phi\left( -y/\sqrt{x/\mu} \right) \\
&= E\left[ \sum_{n=L_\kappa(x)+1}^\infty (1-\rho) \rho^n \Indicator(Z > y/\sqrt{x/\mu})  \right] \\
&= E\left[ \rho^{\max\{\lfloor a(x,Z) \rfloor+1, L_\kappa(x)+1\}}   \right]  \\
&= E\left[ \rho^{\lfloor a(x,Z) \rfloor + 1} \Indicator( a(x, Z) \geq L_\kappa(x))   \right] + \rho^{L_\kappa(x)+1} P( a(x,Z) < L_k(x)) \\
&= E\left[ \rho^{\lfloor a(x,Z) \rfloor + 1} \Indicator\left(Z \leq (x-\mu L_\kappa(x))/\sqrt{\sigma^2 x/\mu} \right) \right] + \rho^{L_\kappa(x)+1} \Phi\left( -(x-\mu L_\kappa(x))/\sqrt{\sigma^2 x/\mu} \right).
\end{align*}
Observe that $h_\kappa(x)/\sqrt{\sigma^2 x/\mu} \leq (x- \mu L_\kappa(x))/\sqrt{\sigma^2 x/\mu} \leq  h_\kappa(x)/\sqrt{\sigma^2 x/\mu} + \mu^{3/2}/\sqrt{\sigma^2 x}$, from where it follows that $E_2(\rho,x)$ can further be bounded by
\begin{align}
E_2(\rho,x) &\leq E\left[ \rho^{\lfloor a(x,Z) \rfloor + 1} \Indicator\left( h_\kappa(x)/\sqrt{\sigma^2 x/\mu} < Z \leq (x-\mu L_\kappa(x))/\sqrt{\sigma^2 x/\mu}  \right) \right] \label{eq:indicators_2} \\
&\hspace{5mm} + \left| E\left[ \left(\rho^{\lfloor a(x,Z) \rfloor + 1} - \rho^{a(x,Z)}\right) \Indicator\left(Z \leq h_\kappa(x)/\sqrt{\sigma^2 x/\mu}\right) \right]   \right| \label{eq:integerpart_2} \\
&\hspace{5mm} + \rho^{L_\kappa(x)+1} \Phi\left( -(x-\mu L_\kappa(x))/\sqrt{\sigma^2 x/\mu} \right) \label{eq:extraterm_2}.
\end{align}

Next, note that since $a(x,z)$ is decreasing in $z$, we obtain that \eqref{eq:indicators_2} is bounded by
\begin{align*}
&\rho^{\lfloor a\left(x,  (x-\mu L_\kappa(x))/\sqrt{\sigma^2 x/\mu} \right) \rfloor +1}  E\left[ \Indicator\left( h_\kappa(x)/\sqrt{\sigma^2 x/\mu} < Z \leq h_\kappa(x)/\sqrt{\sigma^2 x/\mu} + \mu^{3/2}/\sqrt{\sigma^2 x}  \right)  \right]  \\
&= \rho^{ L_\kappa(x) +1} \left( \Phi\left( h_\kappa(x)/\sqrt{\sigma^2 x/\mu} + \mu^{3/2}/\sqrt{\sigma^2 x} \right) - \Phi\left( h_\kappa(x)/\sqrt{\sigma^2 x/\mu} \right)  \right) \\
&\leq \rho^{\lfloor \frac{1}{\mu} \left(x- h_\kappa(x) \right) \rfloor +1}  \Phi'\left(  h_\kappa(x)/\sqrt{\sigma^2 x/\mu} \right) \frac{\mu^{3/2}}{\sigma \sqrt{x}} \\
&\leq \frac{C}{\sqrt{x}} \rho^{  \frac{1}{\mu} \left( x - h_\kappa(x)  \right)} e^{-\frac{\mu (h_\kappa(x))^2}{2\sigma^2x }}  .
\end{align*}
For \eqref{eq:integerpart_2} we use the simple bound
$$(1-\rho) E\left[ \rho^{a(x,Z)} \Indicator\left(Z \leq  h_\kappa(x)/\sqrt{\sigma^2 x/\mu} \right) \right].$$
And for \eqref{eq:extraterm_2} we use the inequality $\Phi(-z) \leq \Phi'(z)/z$ for any $z > 0$ to obtain the bound
\begin{align*}
\rho^{(x-h_\kappa(x))/\mu} \Phi\left( - h_\kappa(x)/\sqrt{\sigma^2 x/\mu} \right) &\leq \frac{C \sqrt{x}}{h_\kappa(x)} \rho^{  \frac{1}{\mu} \left( x - h_\kappa(x) \right)} e^{-\frac{\mu (h_\kappa(x))^2}{2\sigma^2x }} . 
\end{align*}
\end{proof}

\bigskip


\begin{prop} \label{P.E2}
Under the assumptions of Theorem \ref{T.SumApprox}, 
$$\lim_{x \to \infty} \sup_{0 < \rho < 1} \frac{E_2(\rho,x)}{Z_\kappa(\rho,x)}  = 0.$$
\end{prop}

\begin{proof}
Let $h_\kappa(x) = \omega_1^{-1}(x)$ if $\kappa = 2$ and  $h_\kappa(x) = \sqrt{x\log x}$, if $\kappa > 2$, then, by Lemma~\ref{L.E2bound}, we have that
\begin{align*}
E_2(\rho,x) &\leq  \frac{C\sqrt{x}}{h_\kappa(x)} \rho^{(x-h_\kappa(x))/\mu} e^{-\frac{\mu (h_\kappa(x))^2}{2\sigma^2x }} + \left( 1-\rho  \right) E\left[ \rho^{a(x,Z)} \Indicator\left( \sigma Z \leq \sqrt{\mu} h_\kappa(x)/\sqrt{x}\right) \right] . 
\end{align*}
Fix $c > 1$ and define $\hat\rho(x) = e^{-\frac{c\mu Q(x)}{x} }$. We will first show that $E_2(\rho,x)$ is $o(Z_k(\rho,x))$ as $x \to \infty$ uniformly for $0 < \rho \leq \hat\rho(x)$. Before we proceed note that $h_\kappa(x)/\sqrt{x} \to \infty$ as $x \to \infty$ (by Lemma \ref{L.b_and_w} (b.) for $\kappa = 2$) and
\begin{align*}
&E\left[ \rho^{a(x,Z)} \Indicator\left( \sigma Z \leq \sqrt{\mu} h_\kappa(x)/\sqrt{x}\right) \right]  \\
&=  e^{\frac{x}{\mu} \log\rho + \frac{\sigma^2 x(\log\rho)^2}{2 \mu^{3}}} \Phi\left( \frac{\sqrt{\mu} h_\kappa(x)}{\sigma\sqrt{x}} + \frac{\sigma  \sqrt{x} \log\rho}{\mu^{3/2}}  \right) \\
&\leq e^{\frac{x}{\mu} \log\rho + \frac{\sigma^2 x(\log\rho)^2}{2 \mu^{3}}} \Phi'\left( \frac{\sqrt{\mu} h_\kappa(x)}{\sigma\sqrt{x}} + \frac{\sigma  \sqrt{x} \log\rho}{\mu^{3/2}}  \right)  \Indicator\left( \frac{\sqrt{\mu} h_\kappa(x)}{\sigma\sqrt{x}} + \frac{\sigma  \sqrt{x} \log\rho}{\mu^{3/2}} \leq -1  \right) \\
&\hspace{3mm} + e^{\frac{x}{\mu} \log\rho + \frac{\sigma^2 x(\log\rho)^2}{2 \mu^{3}}} \Indicator\left( \frac{\sqrt{\mu} h_\kappa(x)}{\sigma\sqrt{x}} + \frac{\sigma  \sqrt{x} \log\rho}{\mu^{3/2}} > -1  \right) \\
&= \frac{1}{\sqrt{2\pi}} \rho^{(x-h_\kappa(x))/\mu} e^{- \frac{\mu (h_\kappa(x))^2}{2\sigma^2 x}} \Indicator\left( |\log\rho| \geq \frac{\mu^2 h_\kappa(x)}{\sigma^2 x} + \frac{\mu^{3/2}}{\sigma \sqrt{x}}  \right) \\
&\hspace{3mm} + e^{\frac{x}{\mu} \log\rho + \frac{\sigma^2 x(\log\rho)^2}{2 \mu^{3}}} \Indicator\left( |\log\rho| < \frac{\mu^2 h_\kappa(x)}{\sigma^2 x} + \frac{\mu^{3/2}}{\sigma \sqrt{x}}  \right),
\end{align*}
where for the inequality we used $\Phi(-z) \leq \Phi'(z)/z$ for $z > 0$.  Furthermore,
$$e^{\frac{x}{\mu} \log\rho + \frac{\sigma^2 x(\log\rho)^2}{2 \mu^{3}}} \Indicator\left( |\log\rho| < \frac{\mu^2 h_\kappa(x)}{\sigma^2 x} + \frac{\mu^{3/2}}{\sigma \sqrt{x}}  \right) \leq e^{\frac{x}{\mu} \left(1 -  \frac{h_\kappa(x)}{2 x} - \frac{\sigma}{2\sqrt{\mu x}}   \right) \log\rho }.$$
It follows that for sufficiently large $x$, $E_2(\rho,x)$ is bounded by
$$C \rho^{(x-h_\kappa(x))/\mu} +  \rho^{ \left(x -  \frac{h_\kappa(x)}{2} - \frac{\sigma \sqrt{x}}{2\sqrt{\mu}}  \right) /\mu}.$$
Now we use Lemma \ref{L.LowerBound}  and the observation that $h_\kappa(x)/ x \to 0$ as $x \to \infty$ to obtain
\begin{align*}
\sup_{0 < \rho \leq \hat\rho(x)} \frac{E_2(\rho,x)}{Z_\kappa(\rho,x)} &\leq C \sup_{0 < \rho \leq \hat\rho(x)} \frac{\rho^{(x-h_\kappa(x))/\mu} +  \rho^{ \left(x -  \frac{h_\kappa(x)}{2} - \frac{\sigma \sqrt{x}}{2\sqrt{\mu}}  \right) /\mu} }{\rho (1-\rho)^{-1}e^{-Q(x)}} \\
&\leq C e^{Q(x)} \sup_{0 < \rho \leq \hat\rho(x)}  \left( \rho^{(x-h_\kappa(x)-\mu)/\mu} +  \rho^{ \left(x -  \frac{h_\kappa(x)}{2} - \frac{\sigma \sqrt{x}}{2\sqrt{\mu}} -\mu \right) /\mu}     \right) \\
&= C e^{Q(x)} \left( e^{-\frac{1}{\mu}(x-h_\kappa(x)-\mu) \frac{c\mu Q(x)}{x}} + e^{- \frac{1}{\mu} \left(x -  \frac{h_\kappa(x)}{2} - \frac{\sigma \sqrt{x}}{2\sqrt{\mu}} -\mu \right) \frac{c\mu Q(x)}{x} }     \right) \\
&\leq C \left( e^{- \left(c- 1 - \frac{c h_\kappa(x)}{x} \right) Q(x)} + e^{-\left(c-1 - \frac{c h_\kappa(x)}{2x} - \frac{c\sigma}{2\sqrt{\mu x}}  \right) Q(x) }  \right) \to 0
\end{align*}
as $x\to \infty$.

For the range $\hat\rho(x) \leq \rho < 1$ we first note that $Z_\kappa(\rho,x) \geq E\left[ \rho^{a(x,Z)} I\left( \sigma Z \leq \sqrt{\mu} h_\kappa(x)/\sqrt{x}\right) \right]$, so we have
$$\sup_{\hat\rho(x) \leq \rho < 1} \frac{(1-\rho) E\left[ \rho^{a(x,Z)} \Indicator\left( \sigma Z \leq \sqrt{\mu} h_\kappa(x)/\sqrt{x}\right) \right] }{Z_\kappa(\rho,x)} \leq 1-\hat\rho(x) \to 0$$
as $x \to \infty$. To analyze the remaining term we use Lemmas \ref{L.LowerBound} and \ref{L.ustar} to obtain $Z_\kappa(\rho,x) \geq C ^{\frac{x}{\mu} \Lambda_\rho(u(\rho))} \geq C e^{\frac{x}{\mu} \left( \log\rho + \frac{\sigma^2}{2\mu^2} (\log\rho)^2 - \eta \Indicator(\kappa > 2) |\log\rho|^3  \right)} $ for some $\eta > 0$. It follows that
\begin{align}
&\sup_{\hat\rho(x) \leq \rho < 1} \frac{C\sqrt{x}}{h_\kappa(x)} \cdot \frac{ \rho^{(x-h_\kappa(x))/\mu} e^{-\frac{\mu (h_\kappa(x))^2}{2\sigma^2x }}}{Z_\kappa(\rho,x)} \notag \\
&\leq \frac{C\sqrt{x}}{h_\kappa(x)} e^{-\frac{\mu (h_\kappa(x))^2}{2\sigma^2x }} \sup_{\hat\rho(x) \leq \rho < 1} \frac{e^{\frac{1}{\mu}\left( x - h_\kappa(x) \right)\log \rho} }{e^{\frac{x}{\mu} \left( \log\rho + \frac{\sigma^2}{2\mu^2} (\log\rho)^2 - \eta \Indicator(\kappa > 2) |\log\rho|^3  \right)}} \notag \\
&= \frac{C\sqrt{x}}{h_\kappa(x)} e^{-\frac{\mu (h_\kappa(x))^2}{2\sigma^2x }  } \sup_{0< s \leq c\mu Q(x)/x} e^{ \frac{1}{\mu}  h_\kappa(x) s - \frac{\sigma^2 x}{2\mu^3} s^2 + \frac{\eta \Indicator(\kappa > 2) x}{\mu} s^3} .\label{eq:hk_supremum} 
\end{align}
When $\kappa = 2$ and $h_\kappa(x) = \omega_1^{-1}(x)$, \eqref{eq:hk_supremum} becomes
$$\frac{C \sqrt{x}}{\omega_1^{-1}(x)} \sup_{0 < s \leq c\mu Q(x)/x} e^{- \frac{1}{2}\left( \frac{\sigma \sqrt{x}}{\mu^{3/2}} s - \frac{\sqrt{\mu} \omega_1^{-1}(x)}{\sigma \sqrt{x}}  \right)^2} \leq \frac{C \sqrt{x}}{\omega_1^{-1}(x)},$$
which by Lemma \ref{L.b_and_w} (b.), converges to zero as $x \to \infty$. When $\kappa >2$ and $h_\kappa(x) = \sqrt{x\log x}$ we split the supremum and bound \eqref{eq:hk_supremum} with
\begin{align*}
&\frac{C}{\sqrt{\log x}} e^{-\frac{\mu \log x}{2\sigma^2 }  } \left\{ \sup_{0< s \leq \min\{ c\mu Q(x)/x, x^{-1/3}\} } e^{ \frac{\sqrt{x\log x}}{\mu}  s - \frac{\sigma^2 x}{2\mu^3} s^2 + \frac{\eta x}{\mu} s^3} \right. \\
&\hspace{25mm} \left. + \sup_{x^{-1/3} < s \leq c\mu Q(x)/x} e^{ \frac{\sqrt{x\log x}}{\mu}  s - \frac{\sigma^2 x}{2\mu^3} s^2 + \frac{\eta x}{\mu} s^3} \Indicator(x^{-1/3} < c\mu Q(x)/x) \right\} \\
&\leq \frac{C}{\sqrt{\log x}} \left\{ \sup_{s \geq 0} e^{-\frac{1}{2} \left( \frac{\sigma \sqrt{x}}{\mu^{3/2}} s - \frac{\sqrt{\mu \log x}}{\sigma} \right)^2}  + \sup_{ s > x^{-1/3}} e^{-\frac{\mu \log x}{2\sigma^2 } + \frac{\sqrt{x\log x}}{\mu}  s - \frac{\sigma^2 x}{2\mu^3} \left(1 -  \frac{2 c\mu^3 \eta Q(x)}{\sigma^2 x } \right) s^2   }  \right\} \\
&\leq \frac{C}{\sqrt{\log x}} \left\{ 1  + \sup_{ s > x^{-1/3}} e^{ - \frac{\sigma^2 x}{2\mu^3}  \left(1 -  \frac{2 c\mu^3 \eta Q(x)}{\sigma^2 x } -  \frac{2\mu^2 \sqrt{\log x}}{\sigma^2 x^{1/6}} \right) s^2   }  \right\} \to 0
\end{align*}
as $x \to \infty$.  This completes the proof. 
\end{proof}

\bigskip


\begin{lem} \label{L.In_Bound}
Let $y = (x-n\mu)/\sigma$ and fix $0 < \delta < 1/2$ and $0 < c < 1$. Define $c_\delta = \delta^{-1} (4 \mu^{-1} \vee 1) \sigma$. Then, under the assumptions of Theorem \ref{T.SumApprox}, for all $n \leq (x- c b^{-1}(x))/\mu$ and $x$ sufficiently large,
\begin{align*}
J(y,n) &\leq C n \overline{F}(\sigma y+\mu) e^{\frac{\delta \mu n}{x} Q(\sigma y)} + \frac{C n^{3/2}}{y} \overline{F}(\sigma\sqrt{n}+\mu) e^{-\frac{\sigma y Q(x)}{x} - \frac{\delta^2 y^2}{8n} } \Indicator(\mu n> x-  c_\delta b^{-1}(x) ). 
\end{align*}
\end{lem}

\begin{proof}
Let $\beta_\epsilon(t) = b^{-1}(2(1+\epsilon)t/\mu)$ and $\overline{V}(t) = P( X_1 > \sigma t + \mu) = \overline{F}(\sigma t+\mu)$. Then from \eqref{eq:Jint} we obtain that
\begin{align*}
J(y,n) &\leq   n\overline{V}(y) + n \int_{0}^1 \overline{V}(y-\sqrt{n}z) \Phi'\left( z \right) dz  + Cn  \int_{1}^{((y-\sqrt{n})\wedge \beta_\epsilon(\mu n))/\sqrt{n}} \overline{V}(y-z\sqrt{n}) e^{nQ_{\tilde\kappa} \left( \frac{z}{\sqrt{n}} \right) } dz.   
\end{align*}
To analyze the integral involving $Q_\kappa$ first note that if $\kappa = 2$, then $n Q_\kappa (z/\sqrt{n}) = -z^2/2$, while if $\kappa > 2$ then $nQ_\kappa(z/\sqrt{n}) = -z^2/2 + O(z^3/\sqrt{n})$. Therefore, for $1 \leq z \leq ((y-\sqrt{n}) \wedge \beta_\epsilon(\mu n))/\sqrt{n}$ and $n$ sufficiently large we have
$$nQ_\kappa(z/\sqrt{n}) \leq -  \frac{(1-\delta) z^2}{2},$$
from where it follows that
$$J(y,n) \leq n \overline{V}(y) + C n \int_0^{ ((y-\sqrt{n}) \wedge \beta_\epsilon(\mu n))/\sqrt{n}} \overline{V}(y-z\sqrt{n}) e^{-\frac{(1-\delta) z^2}{2}} dz.$$

We now bound the remaining integral with
\begin{align}
&Cn \int_0^{\delta y/\sqrt{n}} \overline{V}(y- z\sqrt{n}) e^{-\frac{(1-\delta) z^2}{2}} dz \label{eq:FirstInt} \\
&\hspace{3mm} + Cn \int_{\delta y/\sqrt{n}}^{(y-\sqrt{n})/\sqrt{n}} \overline{V}(y-z\sqrt{n}) e^{-\frac{(1-\delta) z^2}{2}} dz \, \Indicator(\beta_\epsilon(\mu n) > \delta y) \label{eq:SecondInt} 
\end{align}
We start by analyzing \eqref{eq:SecondInt}, which is further bounded by
\begin{align*}
&C n \overline{V}(\sqrt{n}) \int_{\delta y/\sqrt{n}}^{(y-\sqrt{n})/\sqrt{n}}  e^{-\frac{(1-\delta) z^2}{2}} dz \, \Indicator(\beta_\epsilon(\mu n) > \delta y) \\
&\leq C n \overline{V}(\sqrt{n}) \Phi\left( -\frac{\sqrt{1-\delta}\delta y}{\sqrt{n}} \right)  \, \Indicator(\beta_\epsilon(x) > \delta y) \\
&\leq C n \overline{V}(\sqrt{n}) \frac{\sqrt{n}}{y} e^{-\frac{(1-\delta) \delta^2 y^2}{2n}}  \, \Indicator(\beta_\epsilon(x) > \delta y) \\
&= \frac{C n^{3/2}}{y} \overline{V}(\sqrt{n}) e^{-\frac{\sigma y Q(x)}{x}} e^{-\frac{(1-\delta) \delta^2 y^2}{2n} \left(1 - \frac{2\sigma n Q(x)}{(1-\delta)\delta^2 x y} \right) }   \, \Indicator(\beta_\epsilon(x) > \delta y),
\end{align*}
where in the second inequality we used the relation $\Phi(-z) \leq \Phi'(z)/z$ for $z > 0$. To obtain the second term in the statement of the lemma note that for $n \leq (x- c b^{-1}(x))/\mu$ we have
$$\frac{2\sigma n Q(x)}{(1-\delta) \delta^2 xy} \leq \frac{2\sigma^2 Q(x)}{(1-\delta)\delta^2 \mu  cb^{-1}(x)},$$
which converges to zero as $x \to \infty$ by Lemma \ref{L.b_and_w} (a.). Then, for sufficiently large $x$,
$$e^{-\frac{(1-\delta) \delta^2 y^2}{2n} \left(1 - \frac{2\sigma n Q(x)}{(1-\delta)\delta^2 x y} \right) } \leq   e^{-\frac{(1-\delta)^2 \delta^2 y^2}{2n}  }  \leq e^{-\frac{ \delta^2 y^2}{8n}  } \qquad (\delta < 1/2). $$
Also, by Lemma \ref{L.g_properties} (c.), $\beta_\epsilon(x) \leq (4\mu^{-1} \vee 1) b^{-1}(x) = \sigma^{-1} \delta c_\delta b^{-1}(x)$. If follows that \eqref{eq:SecondInt} is bounded by
$$\frac{C n^{3/2}}{y} \overline{V}(\sqrt{n}) e^{-\frac{\sigma y Q(x)}{x} - \frac{\delta^2 y^2}{8n} } \Indicator(\mu n> x-  c_\delta b^{-1}(x) ).$$

To bound \eqref{eq:FirstInt} we first note that by Assumption \ref{A.Hazard}, $q(t) \leq (r+\delta) Q(t)/t$ for sufficiently large $t$. Also, by Proposition 3.7 in \cite{Bal_Dal_Klupp_04}, $Q(t)/t$ is eventually decreasing, so we obtain 
$$\overline{V}(y - u) =  \overline{V}(y) e^{\int_{\sigma (y-u)+\mu}^{\sigma y+\mu} q(t) dt} \leq  \overline{V}(y) e^{(r+\delta) \frac{Q(\sigma (y-u))}{ y-u} u}.$$ 
Then, the change of variables $u = z\sqrt{n}$ yields the bound
\begin{align}
&C \sqrt{n} \, \overline{V}(y) \int_{0}^{\delta y} e^{\int_{\sigma (y - u)+\mu}^{\sigma y+\mu} q(t) dt } e^{-\frac{(1-\delta) u^2}{2 n}} du \notag \\
&\leq C \sqrt{n} \, \overline{V}(y) \int_{0}^{\delta y} e^{ (r+\delta) \frac{Q(\sigma (y - u))}{y - u} u -\frac{(1-\delta) u^2}{2 n}} du \notag \\
&\leq  C \sqrt{n} \, \overline{V}(y) \int_{0}^{\delta y} e^{ (r+\delta) \frac{Q(\sigma y)}{(1-\delta) y} u -\frac{(1-\delta) u^2}{2 n}} du \notag \\
&= C n \overline{V}(y) e^{\frac{ (r+\delta)^2 n Q(\sigma y)^2}{2 (1-\delta)^3 y^2}} \int_{-\frac{\sqrt{1-\delta}}{ \sqrt{n}} \cdot \frac{(r+\delta) n Q(\sigma y)}{(1-\delta)^2 y} }^{\frac{\sqrt{1-\delta}}{ \sqrt{n}} \left( \delta y - \frac{(r+\delta) n Q(\sigma y)}{(1-\delta)^2 y}  \right)  } \frac{1}{\sqrt{2\pi}} \, e^{ -\frac{z^2}{2} } dz. \label{eq:ExpNormal}
\end{align}
Now, define the set $A = \left\{ \sigma y \geq \omega_1^{-1} \left( \frac{(r+\delta)^2\sigma^2x}{(1-\delta)^3\delta \mu} \right) \right\}$ and note that $t^2/Q(t) = \omega_1(t)$ is eventually increasing. It follows that for large enough $x$, 
$$A \subseteq \left\{ \omega_1(\sigma y) \geq \frac{(r+\delta)^2\sigma^2x}{(1-\delta)^3\delta \mu} \right\} = \left\{ \frac{(r+\delta)^2 Q(\sigma y)}{(1-\delta)^3y^2} \leq \frac{\delta \mu}{x} \right\}.$$
Also, 
\begin{align*}
A^c &\subseteq \left\{ \omega_1(\sigma y) \leq \frac{(r+\delta)^2\sigma^2x}{(1-\delta)^3\delta \mu} \right\} = \left\{ \frac{(r+\delta) nQ(\sigma y)}{(1-\delta)^2y} \geq \frac{\delta y (1-\delta) n \mu}{(r+\delta) x} \right\}.
\end{align*}
We then have that for $z(x,n) = \frac{\sqrt{1-\delta}}{ \sqrt{n}} \left( \frac{(r+\delta) n Q(\sigma y)}{(1-\delta)^2 y} - \delta y \right)$, \eqref{eq:ExpNormal} is bounded by
\begin{align*}
&Cn\overline{V}(y) \left\{ e^{\frac{\delta \mu n Q(\sigma y)}{2x}} \Indicator(A) + e^{\frac{ (r+\delta)^2 n Q(\sigma y)^2}{2 (1-\delta)^3 y^2} - \frac{z^2(x,n)}{2}} \sup_{t \geq z(x,n)} \frac{\Phi(-t)}{\Phi'(t)} \Indicator(A^c) \right\} \\
&\leq Cn\overline{V}(y) \left\{ e^{\frac{\delta \mu n Q(\sigma y)}{2x}} \Indicator(A) + e^{ \frac{\delta (r+\delta) Q(\sigma y)}{(1-\delta)} - \frac{(1-\delta) \delta^2 y^2}{2n} } \sup_{t \geq \frac{\sqrt{1-\delta}\delta y}{ \sqrt{n}} \left( \frac{ (1-\delta) n \mu}{(r+\delta) x}  - 1 \right)} \frac{\Phi(-t)}{\Phi'(t)} \Indicator(A^c) \right\}.
\end{align*}
Finally, we note that on $A^c = \left\{ n\mu > x-\omega_1^{-1} \left( \frac{(r+\delta)^2\sigma^2x}{(1-\delta)^3\delta \mu} \right) \right\}$, and for sufficiently large $x$, $(1-\delta)n\mu/((r+\delta)x) \geq 1$, so \eqref{eq:ExpNormal} is bounded by
$$Cn\overline{V}(y) \left\{ e^{\frac{\delta \mu n Q(\sigma y)}{x}} \Indicator(A) + e^{ \frac{\delta \mu n Q(\sigma y)}{x}  } \sup_{t \geq 0} \frac{\Phi(-t)}{\Phi'(t)} \Indicator(A^c) \right\} \leq Cn\overline{V}(y) e^{\frac{\delta \mu n Q(\sigma y)}{x}}.$$
\end{proof}

\bigskip


\begin{prop} \label{P.E3}
Under the assumptions of Theorem \ref{T.SumApprox}, 
$$\lim_{x \to \infty} \sup_{0 < \rho < 1} \frac{E_3(\rho,x)}{Z_\kappa(\rho,x)}  = 0.$$
\end{prop}

\begin{proof}
Set $y = (x-n\mu)/\sigma$ and recall that by assumption there exists $\beta > a(r) \geq 2$ such that $Q(t) \geq \beta \log t$ for all sufficiently large $t$. Now choose $0 < \delta < \min\{ (1-\tilde r)/2, (1-r-2/\beta)/2 \}$. Note that by Lemma \ref{L.UglyTail} there exists a constant $0< \gamma_1 \leq 1$ such that $C_n \geq \gamma_1 b^{-1}(\mu n)$.  Define $c_\epsilon = (1-\epsilon)^2 \gamma_1 \sigma$ and $c_\delta = \delta^{-1} (4\mu^{-1} \vee 1)\sigma$. Then, for any $n \geq l_\epsilon(x) \triangleq (x - c_\epsilon b^{-1}(x))/\mu$ and $x$ sufficiently large,
$$y \leq (1-\epsilon)^2 \gamma_1 b^{-1}(x) \leq  \left( \frac{x}{\mu n} \vee 1 \right) (1-\epsilon)^2 \gamma_1 b^{-1}(\mu n) < (1-\epsilon) C_n.$$
Therefore, $\{y > (1-\epsilon) C_n\} \subset \{ n < l_{\epsilon}(x) \}$, and
\begin{align*}
E_3(\rho,x) &\leq  \sum_{n=K_r(x)+1}^{\lfloor l_{\epsilon}(x) \rfloor} (1-\rho) \rho^n  J(y,n) .
\end{align*}
By Lemma \ref{L.UniformBound},
\begin{equation} \label{eq:ThirdError}
\sup_{0 < \rho < 1} \frac{E_3(\rho,x)}{Z_\kappa(\rho,x)}  \leq  \frac{C Q(x)}{x } \sum_{n = K_r(x)+1}^{\lfloor l_{\epsilon}(x) \rfloor} e^{\frac{\sigma y Q(x)}{x} } J(y,n) ,
\end{equation}
for sufficiently large $x$. Define $m_{\delta}(x) = (x - c_\delta b^{-1}(x))/\mu$. Then, by  Lemma \ref{L.In_Bound},  
\begin{align}
\frac{Q(x)}{x} \sum_{n=K_r(x)+1}^{\lfloor l_{\epsilon}(x) \rfloor} e^{\frac{\sigma y Q(x)}{x}} J(y,n) &\leq \frac{C Q(x)}{x} \sum_{n=K_r(x)+1}^{\lfloor l_{\epsilon}(x) \rfloor} e^{\frac{\sigma y Q(x)}{x}} n \overline{F}(\sigma y+\mu) e^{\frac{\delta \mu n}{x} Q(\sigma y)} \label{eq:delta_Int}  \\
&\hspace{3mm} +  \frac{C Q(x)}{x} \sum_{n=\lfloor m_{\delta}(x) \rfloor +1}^{\lfloor l_{\epsilon}(x) \rfloor} \frac{n^{3/2}}{y} \overline{F}(\sigma\sqrt{n}+\mu) e^{ - \frac{\delta^2 y^2}{8n} }. \label{eq:b_Int} 
\end{align}

We start by showing that \eqref{eq:delta_Int} converges to zero. To do so we first bound it with the following integral
\begin{align}
&C Q(x)  \int_{x- \lfloor l_\epsilon(x) \rfloor \mu}^{x-K_r(x)\mu} e^{\frac{u Q(x)}{x}} \overline{F}(u) e^{\frac{\delta (x-u+\mu)}{x} Q(u)} du \notag \\
&\leq C Q(x)  \int_{x-\mu l_{\epsilon}(x)}^{x-\mu K_r(x)} e^{\frac{u Q(x)}{x} - Q(u) + \frac{\delta (x- u)}{x} Q(u)} du. \label{eq:simplerInt}
\end{align}
Now, by Proposition 3.7 in \cite{Bal_Dal_Klupp_04} we have that $Q(x) \leq (x/u)^{r+\delta} Q(u)$ for all $u \leq x$, from where it follows that
\begin{align*}
\frac{uQ(x)}{x} - Q(u) +  \frac{\delta (x-u)}{x} Q(u)  &\leq Q(u) \left( \left(\frac{u}{x}\right)^{1-r-\delta} - 1+  \frac{\delta(x-u)}{x} \right) \\
&\leq Q(u) \left( (1-r-\delta) \left( \frac{u}{x} - 1 \right) + \frac{\delta(x-u)}{x} \right) \\
&= - \frac{Q(u)(x-u)}{x} \left( 1-r-2\delta \right). 
\end{align*}
Let $\eta = \delta^{1/(1-r-\delta)}$. Next we will split \eqref{eq:simplerInt} into three integrals and use one of the above inequalities to bound the exponent as follows
\begin{align}
&C Q(x) \int_{x-\mu l_{\epsilon}(x)}^{\min\{x-\mu K_r(x),\eta x\}} e^{Q(u) \left( \left(\frac{u}{x}\right)^{1-r-\delta} - 1+  \frac{\delta(x-u)}{x} \right)} du \label{eq:small_u} \\
&\hspace{3mm} + C Q(x)  \int_{\min\{x-\mu K_r(x), \eta x\}}^{\min\{x-\mu K_r(x),x/2+\mu\}} e^{- \frac{Q(u)(x-u)}{x} \left( 1-r-2\delta \right)} du \label{eq:inter_u} \\
&\hspace{3mm} + C Q(x) \int_{\min\{ x-\mu K_r(x), x/2+\mu\}}^{x-\mu K_r(x)} e^{- \frac{Q(u)(x-u)}{x} \left( 1-r-2\delta \right)} du. \label{eq:large_u}
\end{align}
To see that \eqref{eq:small_u} converges to zero we note that it is bounded by
\begin{align*}
C Q(x) \int_{x-\mu l_{\epsilon}(x)}^{\eta x} e^{-Q(u) \left( 1-2\delta  \right)} du &\leq CQ(x) \int_{x-\mu l_{\epsilon}(x)}^\infty e^{-\beta (1-2\delta) \log u} du \\
&= C Q(x) (x-\mu l_{\epsilon}(x))^{-\beta(1-2\delta) +1} \\
&\leq \frac{C Q(x)}{x-\mu l_{\epsilon}(x)} \quad \left( \text{since } \beta(1-2\delta) > 2 \right) \\
&\leq \frac{C Q(x)}{b^{-1}(x)},
\end{align*}
where the last expression converges to zero by Lemma \ref{L.b_and_w} (a.). To see that \eqref{eq:inter_u} converges to zero note that it is bounded by
\begin{align*}
&C Q(x)  \int_{\eta x}^{x/2+\mu} e^{- \frac{Q(\eta x)(x-u)}{x} \left( 1-r-2\delta \right)} du \, \Indicator(x-\mu K_r(x) > \eta x) \\
&= \frac{C x Q(x)}{Q(\eta x)} \int_{-(1-r-2\delta) (1-\eta) Q(\eta x)}^{-\frac{(1-r-2\delta)}{2} (1-2\mu/x)Q(\eta x) } e^v dv \, \Indicator(x-\mu K_r(x) > \eta x) \\
&\leq \frac{C x Q(x)}{Q(\eta x)} e^{- \frac{(1-r-2\delta)}{2} Q(\eta x)} \leq C x e^{-\frac{(1-r-2\delta)\beta}{2} \log (\eta x)},
\end{align*}
where in the last inequality we used Proposition 3.7 in \cite{Bal_Dal_Klupp_04} to obtain $Q(x) \leq \eta^{-(r+\delta)} Q(\eta x)$ and then the assumption $Q(t) \geq \beta \log t$. The last thing to notice is that our choice of $\delta$ guarantees that $(1-r-2\delta)\beta/2 > 1$. 

Next, to analyze \eqref{eq:large_u} we follow a similar approach and use the fact that $Q(t)/t$ is eventually decreasing to obtain the bound
\begin{align}
&C Q(x) \int_{x/2+\mu}^{x-\mu K_r(x)}  e^{- \frac{Q(x)u(x-u)}{x^2} \left( 1-r-2\delta \right)} du \, \Indicator(x-\mu K_r(x) > x/2+\mu) \notag \\
&\leq   C Q(x) \int_{x/2}^{x-\mu K_r(x)} e^{- \frac{\mu Q(x)K_r(x) u}{x^2} \left( 1-r-2\delta \right)} du \, \Indicator(\mu (K_r(x)+1) < x/2) \notag \\
&\leq \frac{C x^2}{K_r(x)} e^{- \frac{\mu Q(x) K_r(x) }{2x} \left( 1-r-2\delta \right)} \, \Indicator(\mu (K_r(x)+1) < x/2) \label{eq:semiexp} .
\end{align}
Now note that $\mu(K_r(x) + 1) < x/2$ implies that $r \in [1/2,1)$, since for $r \in (0,1/2)$ we have $x-\mu K_r(x) = o(x)$; and in this case, 
$$\{ \mu (K_r(x)+1) < x/2 \} \subset \{ \min\{ \mu\omega_2(x), x/2 \} < x/2\} = \{ \mu \omega_2(x) < x/2\}.$$ 
 It follows that \eqref{eq:semiexp} is bounded by
 \begin{align*}
\frac{C x^2}{\omega_2(x)} e^{-\frac{\mu Q(x) \omega_2(x)}{2x} (1-r-2\delta)} \Indicator(\mu \omega_2(x) < x/2) &\leq C Q(x)^2 e^{-\frac{\mu x}{2 Q(x)} (1-r-2\delta)}  \\
&\leq C x^2 e^{-\frac{\mu (1-r-2\delta)}{2} x^{1-r-\delta}} \to 0,
\end{align*}
where in the second inequality we used Proposition 3.7 in \cite{Bal_Dal_Klupp_04} to obtain that  $Q(x) \leq x^{r+\delta}$ for large enough $x$. 

Finally, to prove that \eqref{eq:b_Int} converges to zero we first bound it with
\begin{align*}
& \frac{C Q(x)}{x} \int_{m_{\delta}(x)\mu}^{(l_{\epsilon}(x)+1)\mu} \frac{s^{3/2}}{x-s} \overline{F}(\sigma\sqrt{s/\mu}) e^{ - \frac{\delta^2 \mu (x-s)^2}{8\sigma^2 s} } ds \\
&\leq \frac{C \sqrt{x}\, Q(x)}{x-(l_{\epsilon}(x)+1)\mu}  \overline{F}\left(\sigma \sqrt{m_{\delta}(x)} \right) \int_{m_{\delta}(x) \mu}^{(l_{\epsilon}(x)+1) \mu}   e^{ - \frac{\delta^2 \mu (x-s)^2}{8\sigma^2 x} } ds  \\
&\leq \frac{C x \, Q(x)}{ b^{-1}(x)}  \overline{F}\left(\sigma \sqrt{m_{\delta}(x)} \right)  \int_{\delta\sqrt{\mu} (x-(l_\epsilon(x)+1)\mu)/(2\sigma\sqrt{x}) }^\infty   e^{ - \frac{z^2}{2} } dz .
\end{align*}
Clearly, the last integral is bounded by a constant, and for the other terms we have
$$\lim_{x \to \infty} x  \overline{F}\left(\sigma \sqrt{m_{\delta}(x)} \right) = \frac{\mu}{\sigma^2} \lim_{t \to \infty}  t^2 \overline{F}(t) = 0,$$
since $E[X_1^2] < \infty$, and, by Lemma \ref{L.b_and_w} (a.), $\lim_{x \to \infty} Q(x)/b^{-1}(x) = 0$. 
This completes the proof. 
\end{proof}

\bigskip

\begin{proof}[Proof of Theorem \ref{T.SumApprox}]
Propositions \ref{P.E1}, \ref{P.E2} and \ref{P.E1} give
$$\lim_{x \to \infty} \sup_{0 < \rho < 1}  \left| \frac{S_\kappa(\rho,x)}{Z_\kappa(\rho,x)} - 1 \right| = 0,$$
which combined with Proposition \ref{P.UglyApprox} give
$$\lim_{x \to \infty} \sup_{0 < \rho < 1}  \left| \frac{P(W_\rho(\infty) > x)}{Z_\kappa(\rho,x)} - 1 \right| = 0.$$
\end{proof}

\section{Proof of Theorem \ref{T.Main}} \label{S.MainProof}

In this section we prove Lemma \ref{L.ustar} and Theorem \ref{T.Main}. To ease the reading we restate the definition of $A_\kappa(\rho,x)$ below.
\begin{equation*} 
A_\kappa(\rho,x) = \sum_{n=1}^{K_r(x)} (1-\rho) \rho^n n \overline{F}(x-n\mu) + e^{\frac{x}{\mu} \Lambda_\rho (w(\rho,x))} ,
\end{equation*}
where $\Lambda_\rho$ is given by \eqref{eq:LambdaDef}, $w(\rho,x) = \min\{ u(\rho), \omega_1^{-1}(x)/x\}$ and $u(\rho)$ is the smallest positive solution to $\Lambda_\rho'(t) = 0$.  

We start with the proof of Lemma \ref{L.ustar} and then split the proof of Theorem \ref{T.Main} into three parts. 

\begin{proof}[Proof of Lemma \ref{L.ustar}]
That $\Lambda_\rho$ is concave in a neighborhood of the origin follows from 
$$\Lambda_\rho'(t) = -\log\rho - \frac{\mu^2}{\sigma^2} t + O(t^2) \quad \text{and} \quad \Lambda_\rho''(t) = - \frac{\mu^2}{\sigma^2} + O(t).$$

If $\kappa = 2$ we have $\Lambda_\rho(t) = (1-t)\log\rho - \frac{\mu^2}{2\sigma^2} t^2$, which is maximized at $u(\rho) = -\frac{\sigma^2}{\mu^2} \log\rho$ and satisfies $\Lambda_\rho(u(\rho)) = \log\rho + \frac{\sigma^2}{2\mu^2}(\log\rho)^2 $. 

In general, for  $\kappa \geq 2$ recall that $P_\kappa(t) = \Lambda_\rho'(t) + \log\rho$, so $u(\rho)$ is the solution to the equation $P_\kappa(t) = \log\rho$. By Lagrange's inversion theorem,
$$u(\rho) = \sum_{n=1}^\infty \frac{d^{n-1}}{d t^{n-1}} \left. \left( \frac{t}{P_\kappa(t)} \right)^n \right|_{u = 0} \frac{(\log\rho)^n}{n!},$$
where
$$\frac{P_\kappa(t)}{t} = \sum_{i=2}^\kappa \sum_{j=2}^i \frac{\lambda_j \mu^j}{j! \sigma^j} \binom{i-1}{i-j} i t^{i-2} \triangleq \sum_{j=0}^{\kappa-2} a_j t^{j} .$$
Furthermore, by Fa\`{a} di Bruno's formula, 
\begin{align*}
b_n &= \frac{d^{n-1}}{d t^{n-1}} \left. \left( \frac{t}{P_\kappa(t)} \right)^n \right|_{t = 0} \\
&= \sum_{(m_1, \dots, m_{n-1}) \in \mathcal{A}_{n-1}} (n+s_{n-1}-1)! (-1)^{s_{n-1}} a_0^{-n-s_{n-1}}  \prod_{j=1}^{n-1} \frac{1}{m_j!} (a_j 1_{(j \leq \kappa-2)})^{m_j},
\end{align*}
where $\mathcal{A}_j = \{ (m_1, \dots, m_{j}) \in \mathbb{N}^{j} : \, 1m_1 + 2 m_2 + \dots + j m_j = j \}$, $s_j = m_1 + \dots + m_j$, and $a_0 = - \frac{\mu^2}{\sigma^2}$. Note that $b_1 = -\sigma^2/\mu^2$. Finally, since
$$\Lambda_\rho(t) = (1-t) \log \rho - \frac{\mu^2 t^2}{2\sigma^2} + O( t^3),$$
we have
$$\Lambda_\rho(u(\rho)) = \log\rho + \frac{\sigma^2}{2\mu^2}( \log \rho)^2  + O (|\log\rho|^3).$$
\end{proof}

\bigskip

We now prove two preliminary results before we proceed to the proof of Theorem~\ref{T.Main}. 

\begin{lem} \label{L.Lambda}
Let $\Lambda_\rho$ be given by \eqref{eq:LambdaDef} and set $u_n = (x-n\mu)/x$. Then, 
$$\sup_{0 < \rho < 1} \left|  \frac{ \sum_{n=M(x)+ 1}^{N(x)} \rho^n \frac{e^{n Q_\kappa\left( \frac{x-n\mu}{\sigma n} \right)}}{x-n\mu}}{ \sum_{n= M(x)+ 1}^{N(x)} \frac{e^{\frac{x}{\mu} \Lambda_\rho(u_n)}}{ x u_n}} - 1 \right| \to 0$$
as $x \to \infty$.
\end{lem}

\begin{proof}
Define the function
$$\tilde \Lambda_\rho(t) = (1-t)\log\rho + (1-t) Q_\kappa\left( \frac{\mu t}{\sigma (1 - t)} \right)$$
and note that 
$$ \sum_{n=M(x)+ 1}^{N(x)} \rho^{n} \frac{e^{n Q_\kappa\left( \frac{x-n\mu}{\sigma n} \right)}}{x-n\mu} =  \frac{1}{x} \sum_{n= M(x)+ 1}^{N(x)} \frac{e^{\frac{x}{\mu} \tilde \Lambda_\rho(u_n)}}{u_n}.$$
By expanding $1/(1-t)^j$ into its Taylor series centered at zero we obtain 
\begin{align*}
(1-t) Q_\kappa \left( \frac{\mu t}{\sigma (1-t)} \right) &= \sum_{j = 2}^\kappa \frac{\lambda_j \mu^j t^j}{j! \sigma^j} \cdot \frac{1}{(1-t)^{j-1}} \\
&= \sum_{j = 2}^\kappa \frac{\lambda_j \mu^j w^j}{j! \sigma^j} \sum_{i = 0}^\infty \binom{j+ i-1 }{i} t^i \\
&= \sum_{j= 2}^\kappa \sum_{i=0}^{\kappa-j} \frac{\lambda_j \mu^j}{j! \sigma^j} \binom{j+i-1}{i} t^{i+j} + O(t^{\kappa+1}) \\
&= \sum_{j= 2}^\kappa \sum_{r=j}^{\kappa} \frac{\lambda_j \mu^j}{j! \sigma^j} \binom{r-1}{r-j} t^{r} + O(t^{\kappa+1}) \\
&= \sum_{r=2}^\kappa \sum_{j = 2}^r \frac{\lambda_j \mu^j}{j! \sigma^j} \binom{r-1}{r-j} t^{r} + O(t^{\kappa+1}).
\end{align*}
Recall from Section \ref{S.ModelDescription} (after equation \eqref{eq:kappa}) that $\kappa \leq (2-r)/(1-r)$ and $\omega_1^{-1}(x) \leq C x^{1/(2-r-\delta)}$ for any $0 < \delta < (1-r)^2/(2-r)$ and $x$ sufficiently large. It follows that for  $0 \leq t \leq u_{M(x)+1} \leq \omega_1^{-1}(x)/x$ we have
$$x t^{k+1} \leq x \left( \frac{\omega_1^{-1}(x)}{x} \right)^{\frac{2-r}{1-r} +1} \leq C x \left( x^{\frac{-(1-r-\delta)}{2-r-\delta} } \right)^{\frac{3-2r}{1-r}} = C x^{\frac{-(1-r)^2  +\delta (2-r)}{(2-r-\delta)(1-r)}} \to 0$$
as $x \to \infty$. Hence, 
$$ \sum_{n=M(x)+ 1}^{N(x)} \rho^{n} \frac{e^{n Q_\kappa\left( \frac{x-n\mu}{\sigma n} \right)}}{x-n\mu} =  \frac{1}{x} \sum_{n= M(x)+ 1}^{N(x)} \frac{e^{\frac{x}{\mu}  \Lambda_\rho(u_n)}}{u_n} (1+o(1))$$
as $x \to \infty$, uniformly in $0 < \rho < 1$.
\end{proof}

\bigskip

The second preliminary result is an application of Laplace's method, which states that the asymptotic behavior of an integral of the form
$$\int_c^d e^{-x \phi(t)} f(t) dt,$$
as $x \to \infty$, is determined by the value of the integral in a small interval around the maximizer of $\phi$ on the interval $[c,d]$. What makes the proof below very technical is that the limits of integration are functions of $x$.  

\begin{lem} \label{L.Watson}
Let $\hat\rho(x) = e^{-c\mu Q(x)/x}$, $c > 0$, $\kappa > 2$, and define $\gamma(x,\rho) = \sqrt{\mu \log x} / \sigma + \sigma\sqrt{x}\log\rho/\mu^{3/2}$. Then, under the assumptions of Theorem~\ref{T.Main}, as $x \to \infty$,
$$\sup_{\hat\rho(x) \leq \rho < 1} \frac{1}{e^{\frac{x}{\mu} \Lambda_\rho(u(\rho))}} \left|  \frac{\sigma \sqrt{x}}{\sqrt{2\pi \mu}} \sum_{n= M(x)+1}^{N(x)} (1-\rho) \rho^{n} \frac{e^{n Q_\kappa\left( \frac{x-n\mu}{\sigma n} \right)}}{x-n\mu}  -  e^{\frac{x}{\mu} \Lambda_\rho(u(\rho))} \Phi\left( -\gamma(\rho,x)  \right) \right| \to 0. $$
\end{lem}

\begin{proof}
Let $u_n = (x-n\mu)/x$ and define $c(x) = \sqrt{\log x/x}$, $d(x) = \omega_1^{-1}(x)/x$. Then by Lemma \ref{L.Lambda}, 
\begin{align*}
\sqrt{x} \sum_{n= M(x)+1}^{N(x)} \rho^{n} \frac{e^{n Q_\kappa\left( \frac{x-n\mu}{\sigma n} \right)}}{x-n\mu}  &= \frac{1}{\sqrt{x}} \sum_{n= M(x)+1}^{N(x)}  \frac{e^{\frac{x}{\mu} \Lambda_\rho(u_n)} }{w_n} (1+o(1)) \\
&=  \frac{\sqrt{x}}{\mu} \sum_{n= M(x)+1}^{N(x)} \int_{u_{n+1}}^{u_n} \frac{e^{\frac{x}{\mu} \Lambda_\rho(t)} }{t}  dt \, (1+o(1)) \\
&=  \frac{\sqrt{x} }{\mu} \int_{c(x)}^{d(x)} \frac{e^{\frac{x}{\mu} \Lambda_\rho(t)} }{t}  dt \, (1+o(1))
\end{align*}
as $x \to \infty$, uniformly for $\hat\rho(x) < \rho < 1$.  It remains to show that
$$\sup_{\hat\rho \leq \rho < 1} \frac{1}{e^{\frac{x}{\mu} \Lambda_\rho(u(\rho))}} \left|  \frac{\sigma \sqrt{x} (1-\rho) }{\mu^{3/2}\sqrt{2\pi}} \int_{c(x)}^{d(x)} \frac{e^{\frac{x}{\mu} \Lambda_\rho(t)} }{t}  dt -  e^{\frac{x}{\mu} \Lambda_\rho(u(\rho))} \Phi\left( -\gamma(\rho,x)  \right)  \right| \to 0$$
We start by computing the derivatives of $\Lambda_\rho(t)$: 
\begin{align*}
\Lambda_\rho'(t) = -\log\rho - \frac{\mu^2}{\sigma^2} t + O( t^2)  \qquad \text{and} \qquad  \Lambda_\rho''(t) = - \frac{\mu^2}{\sigma^2 } + O(t),
\end{align*}
and note that $t \to 0$ for all $0 \leq t \leq d(x)$.  Also, we have $u(\rho) = -\frac{\sigma^2}{\mu^2} \log\rho + O((\log\rho)^2) = \frac{\sigma^2}{\mu^2} (1-\rho) + O((1-\rho)^2) \in (0, d(x))$ for all $\hat\rho(x) \leq \rho < 1$.  Set $\varepsilon = \varepsilon(x) = 1/\log\log x$ and note that for $\hat\rho(x) \leq \rho < 1$, $u(\rho) = o(d(x))$, so for sufficiently large $x$ we have
\begin{align}
&E(\rho,x) \triangleq \left|  \frac{\sigma \sqrt{x} (1-\rho) }{\mu^{3/2}\sqrt{2\pi}} \int_{c(x)}^{d(x)} \frac{e^{\frac{x}{\mu} \Lambda_\rho(t)} }{t}  dt -  e^{\frac{x}{\mu} \Lambda_\rho(u(\rho))} \Phi\left( -\gamma(\rho,x)  \right)  \right| \notag \\
&\leq \left|  \frac{\sigma \sqrt{x} (1-\rho) }{\mu^{3/2}\sqrt{2\pi}} \int_{c(x)\vee (1-\varepsilon)u(\rho)}^{c(x) \vee (1+\varepsilon)u(\rho) } \frac{e^{\frac{x}{\mu} \Lambda_\rho(t)} }{t}  dt -  e^{\frac{x}{\mu} \Lambda_\rho(u(\rho))} \Phi\left( -\gamma(\rho,x)  \right)  \right|  \label{eq:bigTerm} \\
&\hspace{5mm} + C \sqrt{x} (1-\rho) \left( \int_{c(x) \vee (1+\varepsilon)u(\rho)}^{d(x)} \frac{e^{\frac{x}{\mu} \Lambda_\rho(t)} }{t}  dt + \int_{c(x)}^{c(x)\vee (1-\varepsilon)u(\rho)} \frac{e^{\frac{x}{\mu} \Lambda_\rho(t)} }{t}  dt \right) \label{eq:smallTerm}
\end{align}
To bound \eqref{eq:bigTerm} note that for some $\xi_t$ between $t$ and $u(\rho)$, 
\begin{align*}
&\int_{c(x)\vee (1-\varepsilon)u(\rho)}^{c(x) \vee (1+\varepsilon)u(\rho) } \frac{e^{\frac{x}{\mu} \Lambda_\rho(t)} }{t}  dt  = e^{\frac{x}{\mu} \Lambda_\rho(u(\rho)) } \int_{c(x)\vee (1-\varepsilon)u(\rho)}^{c(x) \vee (1+\varepsilon)u(\rho) } \frac{e^{ \frac{x\Lambda_\rho''(\xi_t)}{2\mu} (t-u(\rho))^2}}{t}  dt ,
\end{align*}
so \eqref{eq:bigTerm} is bounded by $e^{\frac{x}{\mu} \Lambda_\rho(u(\rho)) } F(\rho,x)$, where
$$F(\rho,x) = \left|   \frac{\sigma \sqrt{x} (1-\rho) }{\mu^{3/2}\sqrt{2\pi}} \int_{c(x)\vee (1-\varepsilon)u(\rho)}^{c(x) \vee (1+\varepsilon)u(\rho) } \frac{e^{ \frac{x\Lambda_\rho''(\xi_t)}{2\mu} (t-u(\rho))^2}}{t}  dt - \Phi\left(-\gamma(\rho,x)\right) \right|.$$
To see that $\sup_{\hat\rho(x) \leq \rho < 1} F(\rho,x) \to 0$ note that for $(1-\varepsilon) u(\rho) \leq t \leq (1+\varepsilon) u(\rho)$ we have $\Lambda_\rho''(t) = -\frac{\mu^2}{\sigma^2} + O( |\log\rho|)$ and also $t = u(\rho)(1+o(1)) = -\frac{\sigma^2}{\mu^2} \log\rho(1+o(1)) = \frac{\sigma^2}{\mu^2}(1-\rho)(1+o(1))$.  Let $A = \{ c(x) < (1-\varepsilon)u(\rho))$ and let $\zeta(\rho) = \max_{(1-\varepsilon)u(\rho) \leq t \leq (1+\varepsilon)u(\rho)} |\Lambda_\rho''(t)|$. We start by analyzing $F(\rho,x) \Indicator(A)$, for which we have
\begin{align*}
&F(\rho,x) \Indicator(A) \\
&\leq \left\{ \left| \frac{\sqrt{x\mu}}{\sigma\sqrt{2\pi}} (1+o(1)) \int_{ (1-\varepsilon)u(\rho)}^{(1+\varepsilon)u(\rho) } e^{ \frac{x\Lambda_\rho''(\xi_t)}{2\mu} (t-u(\rho))^2}  dt  - 1 \right| + \Phi\left(\gamma(\rho,x)\right) \right\} \Indicator(A) \\
&\leq \left\{ 1 - \frac{\mu}{\sigma\sqrt{ \zeta(\rho)}} (1+o(1)) \int_{ -\frac{\sqrt{x\zeta(\rho)}}{\sqrt{\mu}} \varepsilon u(\rho)}^{\frac{\sqrt{x\zeta(\rho)}}{\sqrt{\mu}} \varepsilon u(\rho) } \frac{e^{ -\frac{z^2}{2} }}{\sqrt{2\pi}}  dz + \Phi\left( \frac{\sqrt{\mu x}(1-\varepsilon)u(\rho)}{\sigma} + \frac{\sigma \sqrt{x}\log\rho}{\mu^{3/2}} \right)   \right\} \Indicator(A) \\
&\leq \left\{ 1 - (1+o(1))\left( 1 - 2\Phi\left( -\frac{\sqrt{x \zeta(\rho)}}{\sqrt{\mu}} \varepsilon u(\rho) \right) \right) +   \Phi\left(-\frac{\sqrt{\mu x}\varepsilon u(\rho)}{\sigma} (1 + O( u(\rho)/\varepsilon) ) \right)   \right\} \Indicator(A) \\
&\leq 3 \Phi\left( - \frac{\sqrt{\mu x} \varepsilon c(x)}{\sigma} (1+ o(1))  \right) + o(1) \to 0
\end{align*}
as $x \to \infty$, uniformly for $\hat\rho(x) \leq \rho < 1$. To analyze $F(\rho,x) \Indicator(A^c)$ we note that on $A^c$ we have $e^{\frac{x \Lambda_\rho''(\xi_t)}{2\mu} (t - u(\rho))^2} = e^{- \frac{x\mu}{2\sigma^2} (t-u(\rho))^2 + O(x \varepsilon^2 u(\rho)^3)} = (1+o(1)) e^{-\frac{x\mu}{2\sigma^2} (t-u(\rho))^2}$, which yields
\begin{align*}
&F(\rho,x) \Indicator(A^c) \\
&= \left| \frac{\sqrt{x\mu}}{\sigma\sqrt{2\pi}} (1+o(1)) \int_{c(x)}^{c(x) \vee (1+\varepsilon) u(\rho)} e^{-\frac{x\mu}{2\sigma^2} (t-u(\rho))^2} dt - \Phi(-\gamma(\rho,x))  \right| \Indicator(A^c) \\
&= \left|  (1+o(1)) \int_{\frac{\sqrt{x\mu}}{\sigma} (c(x)-u(\rho))}^{\frac{\sqrt{x\mu}}{\sigma}  \varepsilon u(\rho) } \frac{e^{-\frac{z^2}{2}}}{\sqrt{2\pi}} dz \, \Indicator(c(x) < (1+\varepsilon)u(\rho)) - \Phi(-\gamma(\rho,x))  \right| \Indicator(A^c) \\
&\leq \left| \Phi\left( \frac{\sqrt{\mu x}}{\sigma}(u(\rho)-c(x)) \right) - \Phi(-\gamma(\rho,x)) \right|  \Indicator(c(x) < (1+\varepsilon) u(\rho)) \Indicator(A^c) \\
&\hspace{3mm} + \Phi\left(- \frac{\sqrt{\mu x}}{\sigma} \varepsilon u(\rho) \right) \Indicator(c(x) < (1+\varepsilon) u(\rho)) + \Phi(-\gamma(\rho,x)) \Indicator(c(x) \geq (1+\varepsilon) u(\rho)) + o(1) \\
&\leq C \left| \frac{\sqrt{\mu x}}{\sigma}(u(\rho)-c(x)) + \gamma(\rho,x)  \right| \Indicator(c(x) < (1+\varepsilon)u(\rho)) \Indicator(A^c) \\
&\hspace{3mm} + \Phi\left( -\frac{\sqrt{\mu x}}{\sigma (1+\varepsilon)} \varepsilon c(x)  \right) +  \Phi\left( - \frac{\sqrt{\mu x}}{\sigma(1+\varepsilon)} \varepsilon c(x) + O(xc(x)^2)  \right)  + o(1) \\
&\leq C \sqrt{x} c(x)^2  + 2 \Phi\left( -\frac{\sqrt{\mu x}}{\sigma} \varepsilon c(x) (1+o(1))   \right) +o(1) \to 0
\end{align*}
as $x \to \infty$. We have thus shown that \eqref{eq:bigTerm} i s $o\left( e^{\frac{x}{\mu} \Lambda_\rho(u(\rho))} \right)$ as $x \to \infty$, uniformly for $\hat\rho(x) \leq \rho < 1$. We now need to show that the same is true of \eqref{eq:smallTerm}. 

Note that \eqref{eq:smallTerm} is bounded by
\begin{align*}
&C\sqrt{x} (1-\rho) \left( e^{\frac{x}{\mu} \Lambda_\rho(c(x)\vee(1+\varepsilon)u(\rho))} + e^{\frac{x}{\mu} \Lambda_\rho((1-\varepsilon)u(\rho))} \Indicator(c(x) < (1-\epsilon)u(\rho))    \right) \int_{c(x)}^{d(x)} \frac{1}{t} dt \\
&\leq C \sqrt{x} (1-\rho) \log \left( \frac{\omega_1^{-1}(x)}{\sqrt{x\log x}} \right) \left( e^{\frac{x}{\mu} \Lambda_\rho((1+\varepsilon)u(\rho))} \Indicator((1+\varepsilon)u(\rho) > c(x))  \right. \\
&\hspace{5mm} \left. + e^{\frac{x}{\mu}\Lambda_\rho(c(x))}  \Indicator((1+\varepsilon)u(\rho) \leq c(x)) +  e^{\frac{x}{\mu} \Lambda_\rho((1-\varepsilon)u(\rho))} \Indicator((1-\varepsilon)u(\rho) > c(x)) \right).
\end{align*}
Note that for any $t$ and some $\xi_t$ between $t$ and $u(\rho)$, 
\begin{align*}
\Lambda_\rho(t)  &= \Lambda_\rho(u(\rho)) - |\Lambda_\rho'(t)| |u(\rho)-t| + \frac{|\Lambda_\rho''(\xi_t)|}{2}(u(\rho)-t)^2 ,
\end{align*}
which gives that for $(1\pm \varepsilon) u(\rho) > c(x)$, $\Lambda_\rho'((1\pm \varepsilon) u(\rho) ) = \mp \frac{\varepsilon \mu^2 u(\rho)}{\sigma^2}  + O( u(\rho)^2)$, and
\begin{align*}
\Lambda_\rho((1\pm\varepsilon)u(\rho)) &\leq \Lambda_\rho(u(\rho)) - |\Lambda_\rho'((1\pm \varepsilon)u(\rho))| \varepsilon u(\rho) + \frac{\mu^2}{2\sigma^2} \varepsilon^2 u(\rho)^2 (1+o(1)) \\
&= \Lambda_\rho(u(\rho)) -  \frac{\mu^2}{2\sigma^2} \varepsilon^2 u(\rho)^2(1+ o(1)) \\
&\leq \Lambda_\rho(u(\rho)) - \frac{\mu^2}{2\sigma^2} \varepsilon^2 c(x)^2(1+o(1)). 
\end{align*}
For $(1+\varepsilon) u(\rho) < c(x)$, $\Lambda_\rho'(c(x)) = -\frac{\mu^2}{\sigma^2} (c(x)-u(\rho))+ O(c(x)^2)$, and
\begin{align*}
\Lambda_\rho(c(x)) &\leq \Lambda_\rho(u(\rho)) - |\Lambda_\rho'(c(x))| |u(\rho)-c(x)| + \frac{\mu^2}{2\sigma^2} (u(\rho)-c(x))^2 (1+o(1)) \\
&= \Lambda_\rho(u(\rho))  -\frac{\mu^2}{2\sigma^2} (c(x)-u(\rho))^2 (1+ o(1)) \\
&\leq \Lambda_\rho(u(\rho)) -  \frac{\mu^2}{2\sigma^2} \varepsilon^2 c(x)^2 (1+ o(1)).
\end{align*}
Therefore, \eqref{eq:smallTerm} is bounded by
\begin{align*}
&C \sqrt{x} (1-\rho) (\log x) e^{\frac{x}{\mu} \Lambda_\rho(u(\rho)) -  \frac{x\mu}{2\sigma^2} \varepsilon^2 c(x)^2 (1+ o(1))} = o\left( e^{\frac{x}{\mu} \Lambda_\rho(u(\rho))} \right).
\end{align*}
This completes the proof. 
\end{proof}

Finally, we give below the proof of Theorem \ref{T.Main}.

\begin{proof}[Proof of Theorem \ref{T.Main}]
Note that by Theorem \ref{T.SumApprox} we know that
$$\lim_{x \to \infty} \sup_{0 < \rho < 1} \left| \frac{P(W_\rho(\infty) > x)}{Z_\kappa(\rho, x)} - 1 \right| = 0,$$
so for the first statement of Theorem \ref{T.Main} it suffices to show that
\begin{equation} \label{eq:ZandA}
\lim_{x \to \infty} \sup_{0 < \rho < 1} \left| \frac{A_\kappa(\rho,x)}{Z_\kappa(\rho, x)} - 1 \right| = 0.
\end{equation}
The second statement, which refers to the uniformity in $x$ as $\rho \nearrow 1$ will follow from Lemma 3.3 in \cite{OlBlGl_10} once we show that $\sup_{0 < x < (1-\rho)^{-1/4}} |A_\kappa(\rho,x) -1| \to 0$ as $\rho \nearrow 1$. To see this is the case simply note that for all $0 < x < (1-\rho)^{-1/4}$
\begin{align*}
\left| A_\kappa(\rho,x) - 1\right| &\leq (1-\rho) \sum_{n=1}^{K_r(x)} n + \left| e^{\frac{x}{\mu} \Lambda_\rho(w(\rho,x))}  - 1 \right| \leq C(1-\rho)K_r(x)^2 + \left| e^{\frac{x}{\mu} \Lambda_\rho(w(\rho,x))} - 1\right| \\
&\leq C(1-\rho) x^2 + C x |\log \rho| \leq C (1-\rho)^{1/2} + C(1-\rho)^{-1/4} |\log\rho| \to 0
\end{align*}
as $\rho \nearrow 1$. We now proceed to establish \eqref{eq:ZandA}. 

Let $h_\kappa(x) = \omega_1^{-1}(x)$ if $\kappa = 2$ and $h_\kappa(x) = \sqrt{x\log x}$ if $\kappa > 2$, and set $u_n = (x-n\mu)/x$. Then,
\begin{align*}
\left| Z_\kappa(\rho,x) - A_\kappa(\rho,x) \right| &= \left| \frac{\sigma \sqrt{x}}{\sqrt{2\pi \mu}} \sum_{n=M(x)+1}^{N(x)} (1-\rho) \rho^n \frac{e^{n Q_k\left( \frac{x-n\mu}{\sigma n}\right)} }{x-n\mu} \Indicator(\kappa > 2) \right. \\
&\hspace{3mm} \left. \phantom{\sum_M^N} +  E\left[ \rho^{a(x,Z)} \Indicator\left(\sigma Z \leq \sqrt{\mu} h_k(x)/\sqrt{x} \right) \right] - e^{\frac{x}{\mu} \Lambda_\rho(w(\rho,x))}  \right|  \\
&\leq \frac{\sigma \sqrt{x}(1-\rho)}{\sqrt{2\pi \mu}}  \left| \sum_{n=M(x)+1}^{N(x)}  \rho^n \frac{e^{n Q_k\left( \frac{x-n\mu}{\sigma n}\right)} }{x-n\mu} -  \sum_{n=M(x)+1}^{N(x)} \frac{e^{\frac{x}{\mu} \Lambda_\rho(u_n) } }{xu_n} \right| \Indicator(\kappa > 2)  \\
&\hspace{3mm}  + \left|  \frac{\sigma (1-\rho)}{\sqrt{2\pi \mu x}} \sum_{n=M(x)+1}^{N(x)}  \frac{e^{\frac{x}{\mu} \Lambda_\rho(u_n) } }{u_n} \Indicator(\kappa > 2) \right. \\
&\hspace{3mm} \left. +  E\left[ \rho^{a(x,Z)} \Indicator\left(\sigma Z \leq \sqrt{\mu} h_\kappa(x)/\sqrt{x} \right) \right] - e^{\frac{x}{\mu} \Lambda_\rho(w(\rho,x))}  \right|.
\end{align*}

Define $\hat\rho(x) = e^{-2 \mu Q(x)/x}$. We separate our analysis into two cases.

{\bf Case 1:} $\hat\rho(x) \leq \rho < 1$.

Note that for this range of values of $\rho$ we have, by Lemma~\ref{L.ustar}, that $u(\rho) = -\frac{\sigma^2}{\mu^2} \log\rho(1+o(1))$, and by Lemma~\ref{L.b_and_w} (a.), that $u(\hat\rho(x)) \approx \frac{2\sigma^2 Q(x)}{\mu x} \leq \frac{\omega_1^{-1}(x)}{x}$ for all sufficiently large $x$. It follows that $w(\rho,x) = \min\{ u(\rho), \omega_1^{-1}(x)/x  \} = u(\rho)$. Also, by Lemma \ref{L.Lambda}  we have that there exists a function $\varphi_1(x) \downarrow 0$ as $x \to \infty$ such that
\begin{align}
&\left| Z_\kappa(\rho,x) - A_\kappa(\rho,x) \right| \notag \\
&\leq \varphi_1(x)  \frac{\sigma\sqrt{x}(1-\rho)}{\sqrt{2\pi \mu}} \sum_{n=M(x)+1}^{N(x)} \frac{e^{\frac{x}{\mu} \Lambda_\rho(u_n)}}{x u_n} \Indicator(\kappa > 2)  \label{eq:watson1} \\
&\hspace{3mm} + \left| \frac{\sigma (1-\rho)}{\sqrt{2\pi \mu x}} \sum_{n=M(x)+1}^{N(x)}  \frac{e^{\frac{x}{\mu} \Lambda_\rho(u_n) } }{u_n}  - e^{\frac{x}{\mu} \Lambda_\rho(u(\rho))} \Phi(-\gamma(\rho,x)) \right| \Indicator(\kappa > 2) \label{eq:watson2} \\
&\hspace{3mm} + \left|  e^{\frac{x}{\mu} \Lambda_\rho(u(\rho))} \Phi(-\gamma(\rho,x)) \Indicator(\kappa > 2) + E\left[ \rho^{a(x,Z)} \Indicator\left(\sigma Z \leq \sqrt{\mu} h_\kappa(x)/\sqrt{x} \right) \right]  - e^{\frac{x}{\mu} \Lambda_\rho(u(\rho))}  \right| \notag
\end{align}
where $\gamma(x,\rho) = \sqrt{\mu \log x}/\sigma + \sigma\sqrt{x}\log\rho/\mu^{3/2}$. Furthermore, by Lemma \ref{L.Watson} we have that \eqref{eq:watson1} and \eqref{eq:watson2} are bounded by 
$$\varphi_1(x) e^{\frac{x}{\mu} \Lambda_\rho(u(\rho))} \left(   \Phi(-\gamma(\rho,x)) + \varphi_2(x)   \right) + \varphi_2(x) e^{\frac{x}{\mu} \Lambda_\rho(u(\rho))} $$
for some other $\varphi_2(x) \downarrow 0$. Since by Lemma \ref{L.LowerBound} we have that $Z_\kappa(\rho,x) \geq C e^{\frac{x}{\mu} \Lambda_\rho(u(\rho))}$ on $\hat\rho(x) \leq \rho < 1$, it only remains to show that the term following \eqref{eq:watson2} is $o\left(e^{\frac{x}{\mu} \Lambda_\rho(u(\rho))}\right)$. First we notice that exact computation gives
\begin{align*}
& \left|  e^{\frac{x}{\mu} \Lambda_\rho(u(\rho))} \Phi(-\gamma(\rho,x)) \Indicator(\kappa > 2) + E\left[ \rho^{a(x,Z)} \Indicator\left(\sigma Z \leq \sqrt{\mu} h_k(x)/\sqrt{x} \right) \right]  - e^{\frac{x}{\mu} \Lambda_\rho(u(\rho))}  \right| \\
&= e^{\frac{x}{\mu} \Lambda_\rho(u(\rho))} \left|  \Phi(-\gamma(\rho,x)) \Indicator(\kappa > 2) + e^{\frac{x}{\mu} \left( \log\rho + \frac{\sigma^2(\log\rho)^2}{2\mu^2} -  \Lambda_\rho(u(\rho)) \right) } \Phi\left( \frac{\sqrt{\mu} h_\kappa(x)}{\sigma\sqrt{x}} + \frac{\sigma \sqrt{x}\log\rho}{\mu^{3/2}} \right) -  1 \right| \\
&= \begin{cases}
e^{\frac{x}{\mu} \Lambda_\rho(u(\rho))}   \Phi\left( -\frac{\sqrt{\mu} \omega_1^{-1}(x)}{\sigma\sqrt{x}} - \frac{\sigma \sqrt{x}\log\rho}{\mu^{3/2}} \right)  , & \kappa = 2, \\
e^{\frac{x}{\mu} \Lambda_\rho(u(\rho))} \Phi\left( \gamma(\rho,x) \right)  \left|   e^{\frac{x}{\mu} \left( \log\rho + \frac{\sigma^2(\log\rho)^2}{2\mu^2} -  \Lambda_\rho(u(\rho)) \right) } - 1 \right| , & \kappa > 2.
\end{cases}
\end{align*}
When $\kappa = 2$ we simply have
$$\sup_{\hat\rho(x) \leq \rho <1}  \Phi\left( -\frac{\sqrt{\mu} \omega_1^{-1}(x)}{\sigma\sqrt{x}} - \frac{\sigma \sqrt{x}\log\rho}{\mu^{3/2}} \right) = \Phi\left( -\frac{\sqrt{\mu} \omega_1^{-1}(x)}{\sigma\sqrt{x}} + \frac{2\sigma  Q(x)}{\mu^{1/2} \sqrt{x}} \right) \to 0$$ 
as $x \to \infty$, since by Lemma \ref{L.b_and_w} (a.) $Q(x)/\omega_1^{-1}(x) \to 0$. When $\kappa > 2$  note that
\begin{align*}
&\sup_{\hat\rho(x) \leq \rho < 1} \Phi\left( \gamma(\rho,x) \right)  \left|   e^{\frac{x}{\mu} \left( \log\rho + \frac{\sigma^2(\log\rho)^2}{2\mu^2} -  \Lambda_\rho(u(\rho)) \right) } - 1 \right| \\
&\leq C \sup_{\hat\rho(x) \leq \rho < 1}  \Phi\left(  \frac{\sqrt{\mu \log x}}{\sigma} - \frac{\sigma\sqrt{x}|\log\rho|}{\mu^{3/2}} \right)  x|\log\rho|^3 \\
&\leq \frac{C}{\sqrt{x}} \sup_{0 < t \leq 2\sigma Q(x)/ \sqrt{\mu x}} \Phi\left( \frac{\sqrt{\mu \log x}}{\sigma} - t \right) t^3 \\
&\leq \frac{C}{x^{1/8}} + \frac{C}{\sqrt{x}} \sup_{t \geq x^{1/8}} \Phi\left( -t \left(1 - \frac{\sqrt{\mu \log x}}{\sigma x^{1/8}} \right) \right) t^3 \to 0
\end{align*}
as $x \to \infty$.

{\bf Case 2:} $0 < \rho \leq \hat\rho(x)$. 

For this range of values of $\rho$ we use Lemma \ref{L.LowerBound} to obtain that $Z_\kappa(\rho,x) \geq C \rho (1-\rho)^{-1} e^{-Q(x)}$, which together with Lemma \ref{L.Lambda} gives, for $0 < \rho \leq \hat\rho(x)$, 
\begin{align*}
&\frac{|Z_\kappa(\rho,x) - A_\kappa(\rho,x)|}{Z_\kappa(\rho,x)} \leq \frac{C(1-\rho) e^{Q(x)}}{\rho} \left\{ \frac{(1-\rho)}{\sqrt{x}} \sum_{n=M(x)+1}^{N(x)} \frac{e^{\frac{x}{\mu} \Lambda_\rho(u_n)} }{u_n} \Indicator(\kappa > 2) \right. \\
&\hspace{-2mm} \left. \phantom{\sum_M^N} + E\left[ \rho^{a(x,Z)} \Indicator\left(\sigma Z \leq \sqrt{\mu} h_\kappa(x)/\sqrt{x} \right) \right] + e^{\frac{x}{\mu} \Lambda_\rho(w(\rho,x))}  \right\} \\
&\leq C\rho^{-1} e^{Q(x)} \left\{ \sqrt{x} e^{\frac{x}{\mu} \Lambda_\rho( w(\rho,x) )} \int_{\sqrt{x\log x}}^{\omega_1^{-1}(x)+\mu} \frac{1}{u} du \,  \Indicator(\kappa > 2) \right. \\
&\hspace{3mm} \left.  + e^{\frac{x}{\mu} \log\rho + \frac{\sigma^2x(\log\rho)^2}{2\mu^3} } \Phi\left( \frac{\sqrt{\mu} h_\kappa(x)}{\sigma\sqrt{x}} + \frac{\sigma \sqrt{x}\log\rho}{\mu^{3/2}} \right) + e^{\frac{x}{\mu} \Lambda_\rho(w(\rho,x))}  \right\}.
\end{align*}

Let $\zeta = \max\{1, 2\mu^2/\sigma^2\}$ and $\bar\rho(x) = e^{-\zeta \omega_1^{-1}(x)/x}$ and note that 
\begin{align*} 
&\sup_{0 < \rho < \bar\rho(x)} \frac{|Z_\kappa(\rho,x) - A_\kappa(\rho,x)|}{Z_\kappa(\rho,x)} \leq \sup_{0 < \rho < \bar\rho(x)} C\rho^{-1} e^{Q(x)} \left\{ \sqrt{x} \log x \, e^{\frac{x}{\mu} \Lambda_\rho(\omega_1^{-1}(x)/x)} \phantom{\left( \frac{h_\kappa}{|\log\rho|} \right)} \right. \\
&\hspace{15mm} \left. + e^{\frac{x}{\mu} \log\rho + \frac{\sigma^2x(\log\rho)^2}{2\mu^3}  } \Phi\left( - \frac{\sigma \sqrt{x}|\log\rho|}{\mu^{3/2}} \left(1 - \frac{\mu^2 h_\kappa(x)}{\sigma^2 x|\log \rho|}   \right)\right)  \right\} \\
&\leq \sup_{0 < \rho < \bar\rho(x)} C \rho^{-1} e^{Q(x)} \left\{  x e^{\frac{x}{\mu} \Lambda_\rho(\omega_1^{-1}(x)/x)} + \frac{1}{\sqrt{x}|\log\rho|} e^{\frac{x}{\mu} \log\rho + \frac{\sigma^2x(\log\rho)^2}{2\mu^3} - \frac{\sigma^2 x(\log\rho)^2}{2\mu^3} \left(1 - \frac{\mu^2 h_\kappa(x)}{\sigma^2 x|\log\rho|} \right)^2 }   \right\} \\
&\leq C e^{Q(x)} \sup_{0 < \rho < \bar\rho(x)} \left\{  x e^{\frac{(x-\omega_1^{-1}(x) - \mu )}{\mu} \log\rho + O\left(\frac{(\omega_1^{-1}(x))^2}{x}\right)} + \frac{1}{\sqrt{x}|\log\rho|} e^{\frac{(x-h_\kappa(x)-\mu)}{\mu} \log\rho  } \right\} \\
&= C e^{Q(x)} \left\{ x e^{-\frac{\zeta\omega_1^{-1}(x)}{\mu} (1+o(1))} + \frac{\sqrt{x}}{\omega_1^{-1}(x)} e^{- \frac{\zeta\omega_1^{-1}(x)}{\mu} (1+o(1))}  \right\},
\end{align*}
which converges to zero as $x \to \infty$ since by Lemma \ref{L.b_and_w}, $Q(x)/\omega_1^{-1}(x) \to 0$ and $\sqrt{x}/\omega_1^{-1}(x) \to 0$.  For the range $\bar\rho(x) \leq \rho \leq \hat\rho(x)$ we have, by Lemma \ref{L.ustar}, 
\begin{align*}
&\sup_{\bar\rho(x) \leq \rho \leq \hat\rho(x)} \frac{|Z_\kappa(\rho,x) - A_\kappa(\rho,x)|}{Z_\kappa(\rho,x)} \\
&\leq  \sup_{\bar\rho(x) \leq \rho \leq \hat\rho(x)} C e^{Q(x)} \left\{ \sqrt{x}\log x e^{\frac{x}{\mu} \Lambda_\rho(u(\rho))} \Indicator(\kappa > 2) +  e^{\frac{x}{\mu} \log\rho + \frac{\sigma^2x(\log\rho)^2}{2\mu^3} }  + e^{\frac{x}{\mu} \Lambda_\rho(u(\rho))}   \right\} \\
&\leq \sup_{\bar\rho(x) \leq \rho \leq \hat\rho(x)} C e^{Q(x)} x e^{\frac{x}{\mu} \log\rho + \frac{\sigma^2 x (\log\rho)^2}{2\mu^3}  (1 + o(1))} \\
&= C x e^{Q(x) + \frac{x}{\mu} \log\hat\rho(x) (1+o(1))}  = C x e^{-  Q(x) (1+o(1))},
\end{align*}
which also converges to zero as $x \to \infty$ since $Q(x) \geq 2\log x$ by \eqref{eq:LowerBoundQ}. This completes the proof.
\end{proof}

\section{Numerical examples} \label{S.Numerical}

We conclude the paper with two examples comparing simulated values of $P(W_\infty(\rho) > x)$ to the approximations $Z_\kappa(\rho,x)$ and $A_\kappa(\rho,x)$ suggested by Theorems \ref{T.SumApprox} and \ref{T.Main}. For illustration purposes we also plot the heavy-tail and heavy-traffic approximations
$$\frac{\rho}{1-\rho} \overline{F}(x) \qquad \text{and} \qquad \exp\left\{ - \frac{x}{\mu} (1-\rho)\right\}.$$
The simulated values of $P(W_\infty(\rho) > x)$ were obtained using the conditional Monte Carlo algorithm from \cite{AsmKro06}, and each point was estimated using 100,000 simulation runs. We point out that simulating heavy-tailed queues in heavy traffic is very difficult, and in particular, the simulated values of $P(W_\infty(\rho) > x)$ for pairs $(x,\rho)$ in the region around the point where the queue's behavior transitions from the heavy traffic regime into the heavy tail regime, are highly unreliable. In terms of the approximations $Z_\kappa(\rho,x)$ and $A_\kappa(\rho,x)$ suggested in this paper, they tend to be sensitive to the mean and variance of the integrated tail distribution, $\mu$ and $\sigma$, respectively, so we suggest first scaling the queue in such a way that both parameters are small (of order one).  We give two examples below, one in which the integrated tail distribution is lognormal and one where it is heavy-tailed Weibull; note that no M/G/1 queue can have exactly Weibull integrated tail distribution, since its density is not monotone, but there are  valid distributions (with decreasing densities) whose tail is asymptotically Weibull. For the lognormal$(\alpha,\beta)$ example we used $Q(x) = (\log x - \alpha)^2/(2\beta^2)$, which although an approximation to $\log \overline{F}(x)$ works well in practice. 

\begin{figure}[htp]
\centering
\frame{
\begin{picture}(350,270)
\put(0,5){\includegraphics[scale = 0.75]{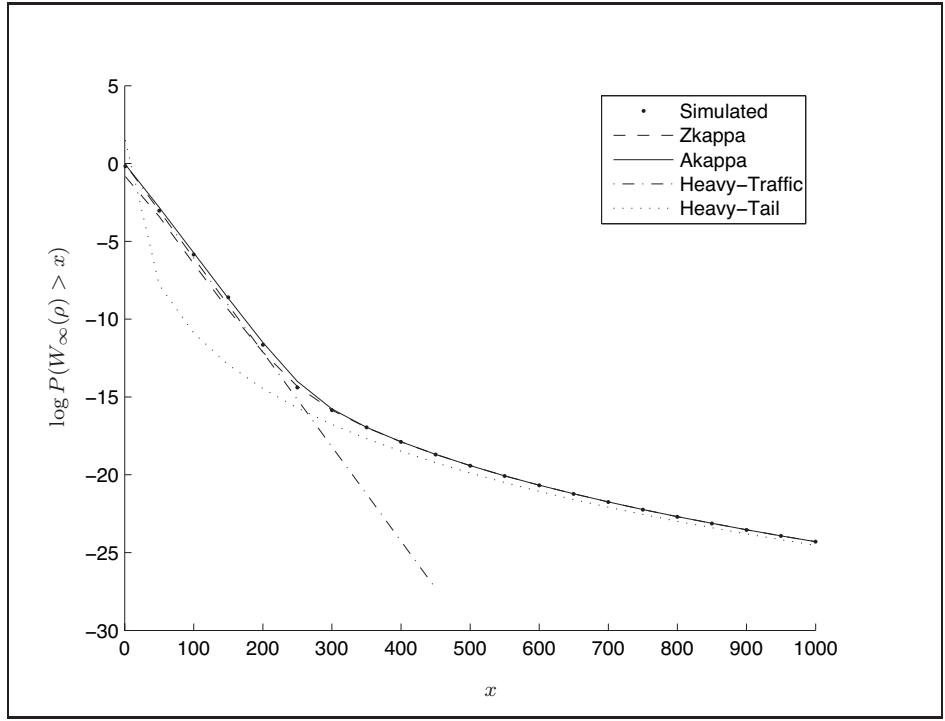}}
\put(15,110){\rotatebox{90}{$\log P(W_\infty(\rho) > x)$}}
\put(180,8){$x$}
\end{picture}
}
\caption{Lognormal$(\alpha,\beta)$ integrated tail with $\rho = 0.9$, $\alpha = 0$, $\beta = 1$.}
\end{figure}

\begin{figure}[htp]
\centering
\frame{
\begin{picture}(350,270)
\put(0,5){\includegraphics[scale = 0.75]{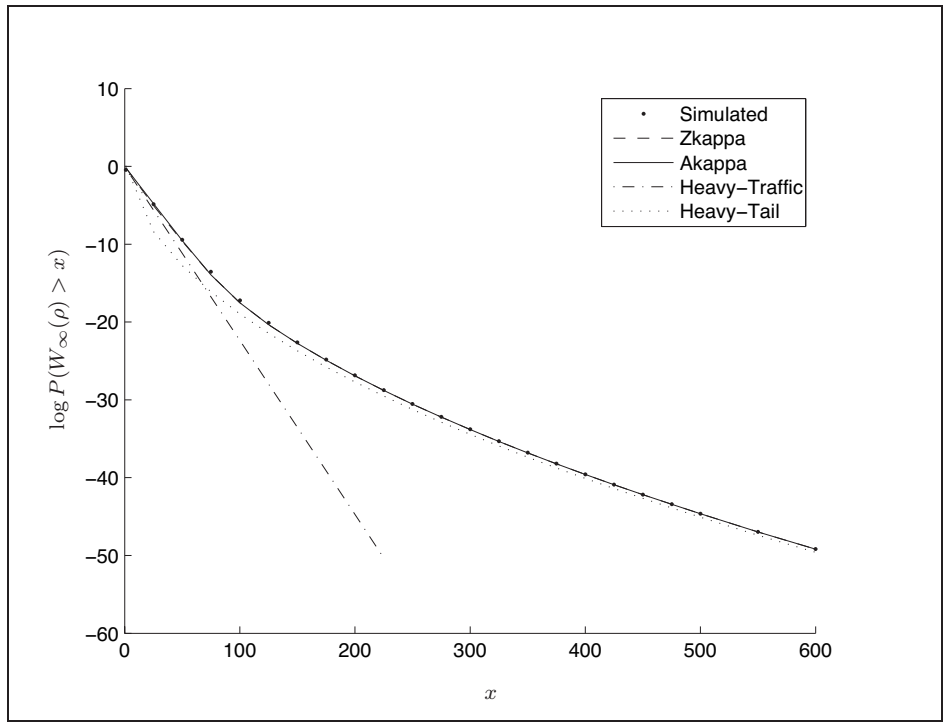}}
\put(15,110){\rotatebox{90}{$\log P(W_\infty(\rho) > x)$}}
\put(180,8){$x$}
\end{picture}
}
\caption{Weibull$(\alpha,\beta)$ integrated tail with $\rho = 0.9$, $\alpha = 0.5$, $\beta = 0.22361$.}
\end{figure}

\section*{Acknowledgements}
The authors would like to thank two anonymous referees for their valuable comments which helped improve the presentation of the paper.

\newpage


\begin{center}
{\bf Table of Notation} \\
(in order of appearance, first time it appears) 

\bigskip

\begin{tabular}[h]{l l l l l}
$F(t)$, $\overline{F}(t)$ & \S \ref{S.ModelDescription}  & & $\Lambda_\rho(t)$ & \eqref{eq:LambdaDef}  \\
$G(t)$, $\overline{G}(t)$ & \S \ref{S.ModelDescription}  & & $u(\rho)$ & \S \ref{S.MainResults} in Lemma \ref{L.ustar}     \\
$Q(t)$ & \S \ref{S.ModelDescription} & & $P_\kappa(\rho,x)$ & \S \ref{S.MainResults} in Lemma \ref{L.ustar}  \\ 
$q(t)$ & \S \ref{S.ModelDescription}  & & $A_\kappa(\rho,x)$ & \eqref{eq:A_Def} \\
$r$ & \eqref{eq:r_Def} & & $w(\rho,x)$ & \eqref{eq:A_Def}    \\
$a(r)$ & \S \ref{S.ModelDescription} in Assumption \ref{A.Hazard} & & $V(t)$, $\overline{V}(t)$ & \S \ref{S.UniformRW} in Assumption \ref{A.StdCase}   \\
$\beta$ & \eqref{eq:LowerBoundQ} & & $\tilde Q(t)$ & \S \ref{S.UniformRW} in Assumption \ref{A.StdCase} and \eqref{eq:tildeQ_Def}  \\
$\mu$ & \S \ref{S.ModelDescription} & & $D(t)$ &  \S \ref{S.UniformRW} in Assumption \ref{A.StdCase}   \\
$\sigma^2$ & \S \ref{S.ModelDescription} & &  $\tilde r$ & \S \ref{S.UniformRW} in Assumption \ref{A.StdCase}  \\
$\kappa$ & \eqref{eq:kappa} & & $\tilde S_n$ & \S \ref{S.UniformRW}    \\
$W_\infty(\rho)$ & \S \ref{S.MainResults} & & $b(t)$, $b^{-1}(t)$ & \eqref{eq:b_Def}  \\
$S_n$ & \S \ref{S.MainResults} & & $L(h)$  & \S \ref{S.UniformRW} in Lemma \ref{L.Roz_pi} \\
$\Phi(t)$ & \eqref{eq:CLT} & & $H(z)$  & \S \ref{S.UniformRW} in Lemma \ref{L.Roz_pi}      \\
$Q_\kappa(t)$ & \eqref{eq:Q_kappa} & & $\pi(z,n)$ & \eqref{eq:Roz_pi}   \\
$Y_1$ & \S \ref{S.MainResults} & & $\tilde \pi(y,n)$  & \eqref{eq:tildePi}    \\
$\omega_1(t), \omega_1^{-1}(t)$ & \S \ref{S.MainResults} & & $J(y,n)$ & \eqref{eq:Jint}  \\
$\omega_2(t), \omega_2^{-1}(t)$ & \S \ref{S.MainResults} & & $C_n$ & \eqref{eq:Roz_Cn}  \\
$K_r(x)$ & \S \ref{S.MainResults} & & $\hat\pi_\kappa(x,n)$ & \eqref{eq:hatPi_Def}  \\
$M(x)$ & \S \ref{S.MainResults} & & $S_\kappa(\rho,x)$ & \eqref{eq:S_rho_x}   \\
$N(x)$ & \S \ref{S.MainResults} & & $\hat\rho(x)$ & \S \ref{S.SumApproxProof} in Lemma \ref{L.LowerBound} \\
$Z_\kappa(\rho,x)$ & \eqref{eq:Z_Def_2} and \eqref{eq:Z_Def_big2} & & $E_i(\rho,x)$, $i = 1,2,3$ & \S \ref{S.SumApproxProof}   \\
$Z$ & \S \ref{S.MainResults} in \eqref{eq:Z_Def_2} and \eqref{eq:Z_Def_big2} & & $h_\kappa(x)$ & \S \ref{S.SumApproxProof} in Lemma \ref{L.E2bound}   \\
$a(x,z)$ & \S \ref{S.MainResults} in \eqref{eq:Z_Def_2} and \eqref{eq:Z_Def_big2}  & & $\gamma(x,\rho)$  & \S \ref{S.MainProof} in Lemma \ref{L.Watson}   \\
\end{tabular}

\end{center}

\bibliographystyle{plain}

\begin{thebibliography}{10}

\bibitem{Asm2003}
S.~Asmussen.
\newblock {\em Applied Probability and Queues}.
\newblock Applications of Mathematics. Springer, New York, 2nd edition, 2003.

\bibitem{AsmKro06}
S.~Asmussen and D.P. Kroese.
\newblock Improved algorithms for rare event simulation with heavy tails.
\newblock {\em Adv. Appl. Prob.}, 38:545--558, 2006.

\bibitem{Bal_Dal_Klupp_04}
A.~Baltr\={u}nas, D.J. Daley, and C.~Kl\"uppelberg.
\newblock Tail behaviour of the busy period of a {GI/GI/1} queue with
  subexponential service times.
\newblock {\em Stochastic Process. Appl.}, 111(2):237--258, 2004.

\bibitem{Bal_Klupp_04}
A.~Baltr\={u}nas and C.~Kl\"uppelberg.
\newblock Subexponential distributions --- large deviations with applications
  to insurance and queueing models.
\newblock {\em Aust. N. Z. J. Stat.}, 46:145--154, 2004.

\bibitem{BlGl07}
J.~Blanchet and P.~Glynn.
\newblock Uniform renewal theory with applications to expansions of random
  geometric sums.
\newblock {\em Adv. in Appl. Probab.}, 39(4):1070--1097, 2007.

\bibitem{Bor00}
A.~Borovkov.
\newblock Estimates for the distribution of sums and maxima of sums of random
  variables without the {Cram\'er} condition.
\newblock {\em Siberian Math J.}, 41(5):997--1038, 2000.

\bibitem{Bor00b}
A.~Borovkov.
\newblock Large deviation probabilities for random walks with semiexponential
  distributions.
\newblock {\em Siberian Math J.}, 41(6):10061--1093, 2000.

\bibitem{Borov_Borov_2008}
A.A. Borovkov and K.A. Borovkov.
\newblock {\em Asymptotic {A}nalysis of {R}andom {W}alks}.
\newblock Cambridge University Press, New York, 2008.

\bibitem{Den_Die_Shn_08}
D.~Denisov, A.B. Dieker, and V.~Shneer.
\newblock Large deviations for random walks under subexponentiality: the
  big-jump domain.
\newblock {\em Ann. Probab.}, 36(5):1946--1991, 2008.

\bibitem{EmVe82}
P.~Embrechts and N.~Veraverbeke.
\newblock Estimates for the probability of ruin with special emphasis on the
  possibility of large claims.
\newblock {\em Insurance: Mathematics and Economics}, 1:55--72, 1982.

\bibitem{IgWh70a}
D.L. Iglehart and W.~Whitt.
\newblock Multiple channel queues in heavy traffic, {I}.
\newblock {\em Adv. Appl. Prob.}, 2:150--177, 1970.

\bibitem{IgWh70b}
D.L. Iglehart and W.~Whitt.
\newblock Multiple channel queues in heavy traffic, {II}.
\newblock {\em Adv. Appl. Prob.}, 2:355--369, 1970.

\bibitem{Jel_Mom_03}
P.~Jelenkovi\'{c} and P.~Momcilovi\'{c}.
\newblock Large deviation analysis of subexponential waiting times in a
  processor sharing queue.
\newblock {\em Math. Oper. Res.}, 28(3):587--608, 2003.

\bibitem{Jel_Mom_04}
P.~Jelenkovi\'{c} and P.~Momcilovi\'{c}.
\newblock Large deviations of square root insensitive random sums.
\newblock {\em Math. Oper. Res.}, 29(2):398--406, 2004.

\bibitem{Nagaev_65}
S.V. Nagaev.
\newblock Some limit theorems for large deviations.
\newblock {\em Theory Probab. Appl.}, 10(2):214--235, 1965.

\bibitem{Nagaev_69a}
S.V. Nagaev.
\newblock Integral limit theorems taking large deviations into account when
  {C}ram\'{e}r's condition does not hold. {I}.
\newblock {\em Theory Probab. Appl.}, 14(1):51--64, 1969.

\bibitem{OlBlGl_10}
M.~Olvera-Cravioto, J.~Blanchet, and P.W. Glynn.
\newblock On the transition from heavy traffic to heavy tails for the {M}/{G}/1
  queue: The regularly varying case.
\newblock To appear in Ann. Appl. Probab., 2010.

\bibitem{Petrov1975}
V.~V. Petrov.
\newblock {\em Sums of Independent Random Variables}.
\newblock Springer-Verlag, Berlin, 1975.

\bibitem{Roz_89}
L.V. Rozovskii.
\newblock Probabilities of large deviations of sums of independent random
  variables with common distribution function in the domain of attraction of
  the normal law.
\newblock {\em Theory Probab. Appl.}, 34:625--644, 1989.

\bibitem{Roz_93}
L.V. Rozovskii.
\newblock Probabilities of large deviations on the whole axis.
\newblock {\em Theory Probab. Appl.}, 38:53--79, 1993.

\bibitem{Roz_99}
L.V. Rozovsky.
\newblock On the {C}ram\'{e}r series coefficients.
\newblock {\em Theory Probab. Appl.}, 43:152--157, 1999.

\bibitem{Wh95}
W.~Whitt, J.~Abate, and G.L. Choudhury.
\newblock Exponential approximations for tail probabilities in queues, {I}:
  Waiting times.
\newblock {\em Oper. Res.}, 43:885--901, 1995.

\bibitem{Zwart_01}
A.P. Zwart.
\newblock Tail asymptotics for the busy period in the {GI/G/1} queue.
\newblock {\em Math. Oper. Res.}, 26(3):485--493, 2001.

\end{thebibliography}

\end{document}